\documentclass{article}

 \usepackage{amsmath}
  \usepackage{amsthm}
 \usepackage{amssymb}
    \usepackage[utf8]{inputenc}
      \usepackage[english]{babel}
         \usepackage{wrapfig}
           \usepackage{graphicx}
             \usepackage{geometry}
               \usepackage[rightcaption]{sidecap}
                  \usepackage{tikz}
                   \usepackage{bbm}
\usepackage{float}
\usepackage{amsfonts}
\usepackage{caption}
\newcommand{\RN}[1]{\textup{\uppercase\expandafter{\romannumeral#1}}}
\newcommand{\Drq}{\mathcal{D}(r;q)}

\newcommand{\vgamma}{{\boldsymbol\gamma}}
\newcommand{\vbeta}{{\boldsymbol\beta}}
\newcommand{\vt}{{\bold t}}
\newcommand{\vn}{\bold n}

\newcommand{\spd}{(\sigma+\delta)}
\newcommand{\huone}{\hat u_1}
\newcommand{\hvone}{\hat v_1}
\newcommand{\FS}{\frac{v}{u_1-u}}
\newcommand{\phiisregular}{\phi\in C^1(\overline\Omega)\cap C^3(\overline\Omega\setminus(\overline\Gamma_{\text{sonic}}\cup\mathcal{C}))}

\newcommand{\linearequationofphi}{ a_{11}\phi_{\xi\xi}+2a_{12}\phi_{\xi\eta}+a_{22}\phi_{\eta\eta}=0}
\newcommand{\equationofu}{a_{11}u_{\xi\xi}+2a_{12}u_{\xi\eta}+a_{22}u_{\eta\eta}+a_{22}(\frac{a_{11}}{a_{22}})_\xi u_\xi+a_{22}(\frac{2a_{12}}{a_{22}})_\xi u_\eta=0}
\newcommand{\equationofv}{a_{11}v_{\xi\xi}+2a_{12}v_{\xi\eta}+a_{22}v_{\eta\eta}+a_{11}(\frac{a_{22}}{a_{11}})_\eta v_\eta+a_{11}(\frac{2a_{12}}{a_{11}})_\eta v_\xi=0}
 \newcommand{\Dn}{D_{\vn}}  
\newcommand{\ft}{{f^{(t)}}}\newcommand{\rt}{{r^{(t)}}}

\newcommand{\G}{\Gamma}
\newcommand{\Gammawedgeminus}{$\Gamma_{\text{wedge}}^-$}
\newcommand{\Gammawedgeplus}{$\Gamma_{\text{wedge}}^+$}
\newcommand{\gammawedgeminus}{\Gamma_{\text{wedge}}^-}
\newcommand{\gammawedgeplus}{\Gamma_{\text{wedge}}^+}
\newcommand{\Gammawedge}{$\Gamma_{\text{wedge}}$}

\newcommand{\gammawedge}{\Gamma_{\text{wedge}}}
\newcommand{\Gammawedgepm}{$\Gamma_{\text{wedge}}^{\pm}$}
\newcommand{\gammawedgepm}{\Gamma_{\text{wedge}}^{\pm} }
\newcommand{\Gammashock}{$\Gamma_{\text{shock}}$}
\newcommand{\gammashock}{\Gamma_{\text{shock}}}
\newcommand{\Gammashockplus}{$\Gamma_{\text{shock}}^+$}

\newcommand{\Gammashockpm}{$\Gamma_{\text{shock}}^\pm$}
\newcommand{\Gammasonic}{$\Gamma_{\text{sonic}}$}
\newcommand{\Gammasonicplus}{$\Gamma_{\text{sonic}}^+$}
\newcommand{\Gammasonicminus}{$\Gamma_{\text{sonic}}^-$}
\newcommand{\Gammasonicpm}{$\Gamma_{\text{sonic}}^\pm$}
\newcommand{\gammasonic}{\Gamma_{\text{sonic}}}
\newcommand{\gammasonicplus}{\Gamma_{\text{sonic}}^+}
\newcommand{\hgammasonicplus}{{\hat\Gamma}_{\text{sonic}}^+}
\newcommand{\gammasonicminus}{\Gamma_{\text{sonic}}^-}

\newcommand{\Statetwobar}{{State$(\overline{\RN{2}})$}}
\newcommand{\Stateone}{State$(\RN{1})$}
\newcommand{\Statetwo}{State$(\RN{2})$}
\newcommand{\Statetwoplus}{State$(\RN{2}_+)$}
\newcommand{\Statetwominus}{State$(\RN{2}_-)$}
\newcommand{\statetwominus}{\text{State}(\RN{2}_-)}
\newcommand{\Statetwopm}{State$(\RN{2}_\pm)$}
\newcommand{\Statecero}{State$(0)$}

\newcommand{\vecp}{p}
\newcommand{\vecq}{q}

\newcommand{\tvphi}{\tilde\varphi}
\newcommand{\tpsi}{\hat\psi}
\newcommand{\tphi}{\tilde\phi}

\newcommand{\rhot}{{\rho^{(t)}}}
\newcommand{\hrho}{\hat\rho}
\newcommand{\hrhoone}{{{\rho_1}}}
\newcommand{\tf}{\hat f}
\newcommand{\trho}{\tilde\rho}
\newcommand{\widetilderhoone}{{\rho_1}}

\newcommand{\varphit}{{\varphi^{(t)}}}
\newcommand{\bs}{\backslash}

\newcommand{\MC}{\mathcal{C}}
\newcommand{\MF}{\mathcal{F}}
\newcommand{\MD}{\mathcal{D}}
\newcommand{\MH}{\mathcal{H}}
\newcommand{\MU}{\mathcal{U}}

\newcommand{\MR}{\mathcal{R}}

\newcommand{\pMR}{\partial\MR}

\newcommand{\tMD}{\widetilde{\mathcal{D}}}
\newcommand{\pSMR}{\partial_S\MR}\newcommand{\pAMR}{\partial_A\MR}\newcommand{\pFMR}{\partial_F\MR}
\newcommand{\pAMU}{\partial_A\MU}\newcommand{\pFMU}{\partial_F\MU}

\newcommand{\pStMD}{\partial_S\tMD}\newcommand{\pAtMD}{\partial_A\tMD}\newcommand{\pFtMD}{\partial_F\tMD}
\newcommand{\sdtc}{\sin\delta\tan\sigma}
\newcommand{\solveregion}{{{\mathbb{R}}^2\setminus W}}

\newtheorem*{notation}{Notation}

\newtheorem{remark}{Remark}
\newtheorem{definition}{Definition}
\newtheorem{problem}{Problem}
\newtheorem{theorem}{Theorem}

\newtheorem{proposition}{Proposition}
\newtheorem{lemma}{Lemma}[section]

\newcommand{\statementstarts}{\begin{equation}\text}
\newcommand{\statementends}{\end{equation}}
\newcommand{\longstatementstarts}{\begin{align}\text}
\newcommand{\longstatementends}{\end{align}}

  \title{ Loss of Regularity of Solutions to Shock Reflection Problems by a Non-symmetric Convex Wedge with Potential Flow Equations}
\author{Jingchen Hu}
\begin{document}
\maketitle

\begin{abstract}
  In this paper, we study the problem of shock reflection by a wedge, with the potential flow equation, which is a simplification of the Euler System.

  In \cite{CF} and \cite{CFbook}, the existence theory of shock reflection problems with the potential flow equation was established, when the wedge is symmetric w.r.t. the direction of the upstream flow. As a natural extension of \cite{CF} and  \cite{CFbook}, we study non-symmetric cases, i.e. when the direction of the upstream flow forms a nonzero angle with the symmetry axis of the wedge.

  The main idea of investigating the regularity of solutions to non-symmetric problems is to study the symmetry of the solution. Then difficulties arise such as free boundaries and degenerate ellipticity, especially when ellipticity degenerates on the free boundary. We developed an integral method to overcome these difficulties.
 
   Some estimates near the corner of wedge is also established, as an extension of G.Lieberman's work. 

    We proved that in non-symmetric cases, the ideal Lipschitz solution to the potential flow equation, which we call regular solution, does not exist. This suggests that the potential flow solutions to the non-symmetric shock reflection problem, should have some singularity which is not encountered  in symmetric case.
\end{abstract}

\section{Introduction}
In \cite{CF}, it was shown that with some simplifications we can reduce the shock reflection problem to a boundary value problem for a quasi-linear second order degenerate elliptic equation.  And in \cite{CF} the problem is rigorously solved when the shock is reflected by a symmetric convex wedge (the symmetric axis of the wedge is perpendicular to the shock).  It's natural to ask if the existence result holds in the non-symmetric case, i.e. when the symmetric axis of the wedge is not perpendicular to the shock. To answer this question, we define regular solutions to the problem, which are characterized by reasonable physical and mathematical properties. Then we derive a contradiction from the existence of such kinds of regular solutions (Theorem \ref{MainTheorem}). So this implies solutions to the non-symmetric problem  should have some type or types of singularity, which is not encountered in the symmetric case.

In this section, we first derive the potential flow equation from conservation laws, then we reduce the problem to a boundary value problem in a domain with a free boundary. Then after stating the existence results and introducing some notations, we define regular solutions to the problem. 

\subsection{The Potential Flow Equation}
 In this paper we study the phenomena of plane shock reflection by a wedge. Precisely, when a plane shock, with upstream state $(u_1,0;\rho_1)$ and downstream state $(0,0;\rho_0)$, hits a wedge $W$,
 \[W=\{(x_1,x_2)\mid-x_1 \cot(\sigma-\delta)<x_2<x_1 \cot (\sigma+\delta), x_1>0\}\] it experiences a reflection-diffraction process.  Here we denote the state with density $\rho$ and velocity $(u,v)$ by $(u,v;\rho)$.

\begin{figure}\centering
    \includegraphics[height=6cm]{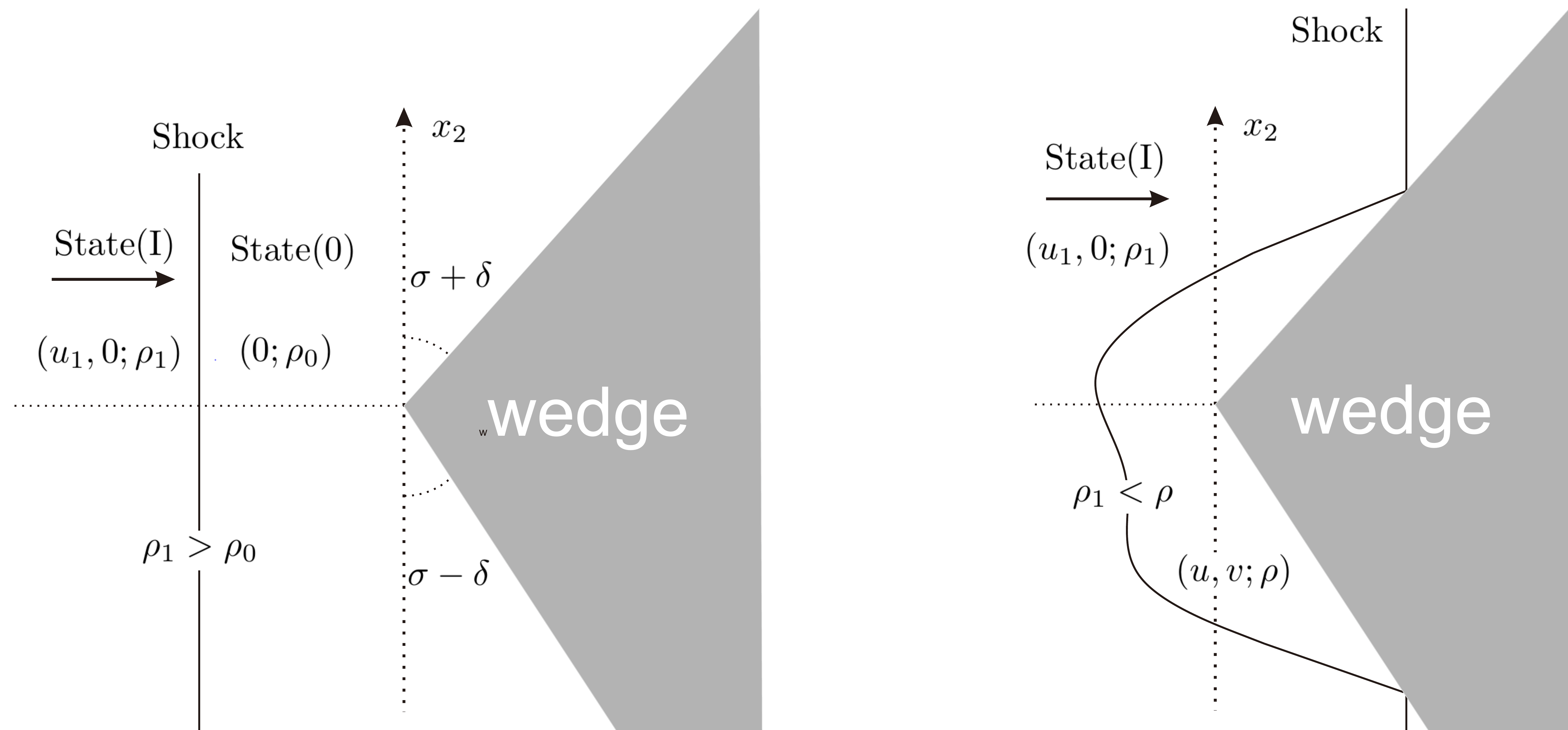}
\caption{Non-symmetric Shock Reflection-Diffraction Problem}
\label{fig:PhysicsStructure}
\end{figure}

Some mathematical models have been used in the past to study this problem, including Euler System \cite{ST} and its simplification potential flow equation \cite{CF} and \cite{CFbook}.

In this paper we consider isentropic $(p=\kappa\rho^\gamma,\gamma>0)$ and non-vortex$(\nabla\times(u,v)=0)$ fluids. So the state of the fluid is governed by the density  $\rho$ and a potential function $\Phi$($\nabla\Phi=(u,v)$).
 
Our equations in $(\vec x, t)$ coordinates are
\begin{align}
\partial_t\rho+\nabla_x\cdot[\rho\nabla_x\Phi]=0\text{\ (mass conservation)}\label{MC},\\
\partial_t\Phi+{1\over 2}\left|\nabla_x\Phi\right|^2+\rho^{\gamma-1}=B_0\text{\ (Bernoulli law)}\label{BL}.
\end{align}
 In the equations above, $\nabla_x=(\partial_{x^1},\partial_{x^2})$. And equation (\ref{BL}) comes from combining the following momentum conservation equation (\ref{MomentumConservation3dim}) with mass conservation equation (\ref{MC}).

The equation of  momentum conservation is 
\begin{align}
\partial_t(\rho\nabla_x\Phi)+\partial_j(\rho\partial_i\Phi\partial_j\Phi)\partial_i+\nabla p=0.\label{MomentumConservation3dim}
\end{align}
Combine it with mass conservation (\ref{MC}), we get:

\[\partial_t\nabla_x\Phi+\nabla_x\frac{|\nabla_x\Phi|^2}{2}+\frac{\nabla p}{\rho}=0.\]
Plugging in $p=\kappa \rho^\gamma$, it becomes:
\[\nabla_x\left(\partial_t\Phi+\frac{|\nabla_x\Phi|^2}{2}+\frac{\gamma\kappa}{\gamma-1}\rho^{\gamma-1}\right)=0.\]

If the potential function $\Phi$ satisfies the equation above weakly on ${\mathbb{R}}^2\setminus W$, we can remove $\nabla_x$ and get
\begin{align}\partial_t\Phi+\frac{|\nabla_x\Phi|^2}{2}+\frac{\gamma\kappa}{\gamma-1}\rho^{\gamma-1}=B_0,\label{pre_Equation2}
\end{align}
where $B_0$ is a constant independent of $x, t$. Previously, for fixed $t$, $\Phi$ is only defined up to a constant, so, if $B_0$ depends on $t$, we can extract a function of $t$ from $B_0$. 

Finally, we can make $\kappa=\frac{\gamma-1}{\gamma}$ by the scaling,
\[(\vec x,t,B_0)\rightarrow (\alpha\vec x, \alpha^2 t, \alpha^{-2}B_0), \alpha^2=\kappa\gamma/(\gamma-1),\] and get 
\begin{align}
\partial_t\Phi+{1\over 2}\left|\nabla_x\Phi\right|^2+\rho^{\gamma-1}=B_0.
\end{align}
And in front of the incident shock $(u, v; \rho)=(0; \rho_0)$, so $B_0$ should be $\rho_0^{\gamma-1}.$
Now we have reduced our reflection-diffraction problem to the following:

\begin{problem}Initial-boundary value problem

Find a solution of system (\ref{MC}) (\ref{BL}), in $(\solveregion)\times \mathbb{R}^{\geq 0}$, with $B_0=\rho_0^{\gamma-1}$, which satisfies the initial condition
\[(\rho,\Phi)\mid_{t=0}=\left\{\begin{array}{cc}(\rho_0,0),              &\text{for } x_2>x_1\cot(\sigma+\delta) \text{ or } x_2<-x_1\cot(\sigma-\delta), x_1>0,\\
                                                                                     (\rho_1, u_1x_1),&\text{for } x_1<0,\end{array}\right.\]
and the slip boundary condition along the boundary of the wedge $\partial W$:
\[\nabla_x\Phi\cdot \vn\mid_{\partial W}=0.\]
\end{problem}
This initial-boundary value problem is invariant under the self-similar scaling:
\[(\vec x, t)\rightarrow (\alpha\vec x, \alpha t),\ (\rho,\Phi)\rightarrow (\rho,\Phi/\alpha)\ \text{for }\alpha> 0,\]
so we look for self-similar solutions which satisfy
\[\rho(\vec x, t)=\rho(\xi,\eta), \ \Phi(\vec x, t)=t \phi(\xi,\eta)\text{ for } (\xi,\eta)=\vec x/t.\]
For self-similar solution mass conservation becomes:
\begin{align}\nabla\cdot(\rho\nabla\phi)-\nabla\rho\cdot(\xi,\eta)=0,\label{selfsimilarequation}
\end{align}
and by Bernoulli Law(\ref{BL}) $c$, $\rho$ can be represented by:
\begin{align}c^2=(\gamma-1)\rho^{\gamma-1}=(\gamma-1)(\rho_0^{\gamma-1}-\phi+\phi_\xi\xi+\phi_\eta\eta-\frac{\phi_\xi^2+\phi_\eta^2}{2}).\label{sonicspeedinintroduction}
\end{align}
If we define $\varphi=\phi-\frac{\xi^2+\eta^2}{2}$, which is named as pseudo-potential function in \cite{CF}, then relations above can be written as
\begin{align}
\nabla\cdot(\rho\nabla\varphi)+2\rho=0, \label{divergenceformequationforpsedopotential}
\end{align}
\begin{align}
c^2=(\gamma-1)\rho^{\gamma-1}=(\gamma-1)\left(\rho_0^{\gamma-1}-{1\over 2}|\nabla\varphi|^2-\varphi\right),\label{sonicspeedwithpseudopotential}
\end{align}
where $c$ stands for the sonic speed. By plugging (\ref{sonicspeedinintroduction}) into (\ref{selfsimilarequation}), the equation for $\phi$ can also be written as 
\[
\left[c^2-(\phi_\xi-\xi)^2\right]\phi_{\xi\xi}-2(\phi_\xi-\xi)(\phi_\eta-\eta)\phi_{\xi\eta}+\left[c^2-(\phi_\eta-\eta)^2\right]\phi_{\eta\eta}=0.
\]
So the equation for $\phi$ or $\varphi$ is elliptic if and only if 
\[|\nabla\varphi|<c.\]


So now we can reduce the initial-boundary value problem in $(\vec x, t)$-space to the following boundary value problem in self-similar $(\xi,\eta)$-space. 

\begin{problem}\label{prob:2}Boundary Value Problem

Find a weak solution $\varphi$ of (\ref{divergenceformequationforpsedopotential}), in ${\mathbb{R}}^2\setminus W$, which satisfies Neumann boundary condition
\[\nabla\varphi\cdot\vn=0,\ \ \text{on }\partial W, \]
and asymptotic boundary condition at infinity: $ \text{when }\xi^2+\eta^2\rightarrow \infty$
\[\varphi\rightarrow\left\{\begin{array}{cc}
						\varphi_0=-{1\over 2}(\xi^2+\eta^2), &\text{for } \xi>X,\eta>\xi \cot (\sigma+\delta) \text{ or } \eta<-\xi \cot(\sigma-\delta)\\
                                               \varphi_1=-{1\over 2}(\xi^2+\eta^2)+u_1(\xi-X),  &\text{for } \xi<X, \eta>0	
						\end{array}\right.,\]
where $X$ is the speed of the incident shock, and  is also the $\xi$-coordinate of the incident shock in $(\xi,\eta)$-plane.
\end{problem}
\subsection{Boundary Condition, Shock and Sonic Circle}
\label{incidentshocknormalandregularreflection}

We try to find a weak solution to the equation (\ref{divergenceformequationforpsedopotential}) (\ref{sonicspeedwithpseudopotential}), however because of the existence of shocks and sonic circles, directly solving (\ref{divergenceformequationforpsedopotential}) (\ref{sonicspeedwithpseudopotential}) is not convenient. So, based on some physical observations, we divide $\solveregion$ into several regions, and solve equations in different regions separately.

Physical observations suggest that,  on $(\xi,\eta)$-plane, $\solveregion$ should be divided into six regions, as illustrated by Figure \ref{fig:GlobalNotation}. The fluid in each region has properties which are different from those of its neighbors. Fluids in front of the incident shock are static, they are in \Statecero. Fluids in { State($\RN{1}$)} are the fluids behind the incident shock, its velocity and density are given as conditions. Fluids in \Statetwopm have uniform velocities and densities, their velocities can be found by algebraic computations.  The essential problem is to describe the behavior of the fluid  in the region $\Omega$. We solve this problem by finding a $\varphi$ solving (\ref{divergenceformequationforpsedopotential}) and (\ref{sonicspeedwithpseudopotential}) in $\Omega$ and satisfying some boundary conditions on $\partial\Omega$,

\begin{figure} 
       \centering
    \includegraphics[height=6cm]{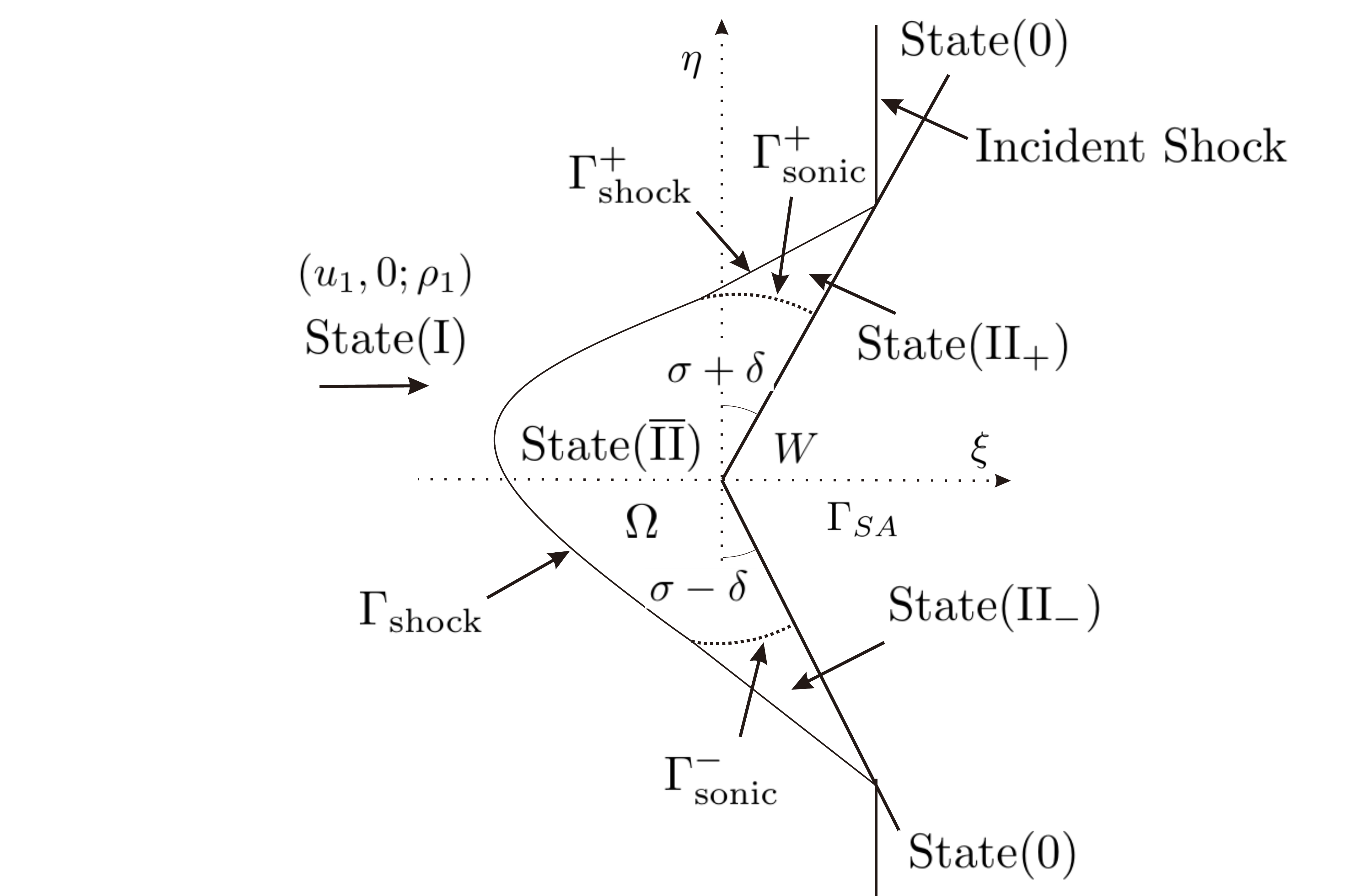}
\caption{Structures and Notations}
\label{fig:GlobalNotation}
\end{figure}

In the following, we explain what are the conditions that should be satisfied on $\partial\Omega$. Actually these conditions are the physical conditions that should be satisfied along shocks and sonic circles.

We first consider the boundary conditions on the shock. Let $\Gamma$ be a shock on $(\xi,\eta)$-plane, separating regions $\Omega_+$ and $\Omega_-$. The potential function, velocity and density of the gas in $\Omega_\pm$ are denoted by $\phi_{\pm}$, $(u_\pm, v_\pm)$ and $\rho_\pm$ respectively.

And since this is self-similar coordinate, point $(\xi,\eta)$ on $\gammashock$, will be moving with velocity  $(\xi,\eta)$. So the relative velocity of point represented by $(\xi,\eta)$ and gas in $\Omega_\pm$, is $(u_\pm-\xi,v_\pm-\eta)$, which is $\nabla\phi_\pm-(\xi,\eta)$ or $\nabla\varphi_\pm$, with pseudo-potential function $\varphi_\pm=\phi_\pm-{\xi^2+\eta^2\over 2}$.  

So Mass Conservation across shock should be 
\begin{equation}[\rho_+(\nabla\phi_+-(\xi,\eta))-\rho_-(\nabla\phi_--(\xi,\eta))]\cdot\vn=0
\label{RHconditionwithphiandn},
\end{equation}
or, with pseudo-potential function 
\begin{equation}(\rho_+\nabla\varphi_+-\rho_-\nabla\varphi_-)\cdot\vn=0
\label{RHconditionwithvarphiandn}.
\end{equation}
And note that, in generic case, the normal direction of $\Gamma$ is $(u_+-u_-,v_+-v_-)$. So (\ref{RHconditionwithphiandn}) and (\ref{RHconditionwithvarphiandn}) can also be written as
\begin{gather}
[\rho_+(\nabla\phi_+-(\xi,\eta))-\rho_-(\nabla\phi_--(\xi,\eta))]\cdot(u_+-u_-,v_+-v_-)=0
\label{RHconditionwithphianduv},\\
(\rho_+\nabla\varphi_+-\rho_-\nabla\varphi_-)\cdot(u_+-u_-,v_+-v_-)=0
\label{RHconditionwithvarphianduv}.
\end{gather}

In this paper, most of time, (\ref{RHconditionwithphiandn}) (\ref{RHconditionwithvarphiandn}) (\ref{RHconditionwithphianduv}) (\ref{RHconditionwithvarphianduv}) will all be referred to as RH condition.

And we also require vorticity vanish, so on $\Gamma$, the tangential velocity of gas in $\Omega_+$ and $\Omega_-$ should be equal. This leads to the continuity assumption of potential function
\begin{equation}
\phi_+=\phi_- \ (\text{or equivalently,}\ \varphi_+=\varphi_-)\label{continuityfreeboundarycondtion}.
\end{equation}

(\ref{continuityfreeboundarycondtion}) and any one of (\ref{RHconditionwithphiandn}) (\ref{RHconditionwithvarphiandn}) (\ref{RHconditionwithphianduv}) (\ref{RHconditionwithvarphianduv}) constitute free boundary condition on shock.

Then for convenience,  we define weak solutions to the potential flow equation. Notice that free boundary condition is singular form of equation(\ref{selfsimilarequation}), which is mass conservation. So it implies that two solutions to equation (\ref{selfsimilarequation}) in two domains separated by a shock, should be considered as a global weak solution, if they satisfy equation weakly, separately in two different domains, and satisfy  free boundary condition.

Assume $g$ is a smooth function with compact support in $\Omega=\Omega_+\cup\Omega_-$, here $\Omega_+,\Omega_-$ are separated by an arc, which is a free boundary $\Gamma$. $\varphi_+, \varphi_-$ are defined in $\Omega_+,\Omega_-$ respectively, satisfying (\ref{selfsimilarequation}) weakly and separately. And they satisfy free boundary condition along $\Gamma$.
\begin{align*}
   &\int_{\Omega}\rho(\nabla\varphi,\varphi)\nabla\varphi\cdot\nabla g-2\rho(\nabla\varphi,\varphi)g \ (\rho \text{ is defined by } (\ref{sonicspeed}),
            \  \varphi=\varphi_\pm \text{ in }\Omega_{\pm})\\
=&\int_{\Omega_+}\rho(\nabla\varphi_+,\varphi_+)\nabla\varphi_+\cdot\nabla g-2\rho(\nabla\varphi_+,\varphi_+)g
    +\int_{\Omega_-}\rho(\nabla\varphi_-,\varphi_-)\nabla\varphi_-\cdot\nabla g-2\rho(\nabla\varphi_-,\varphi_-)g\\
=&\int_\Gamma(\rho(\nabla\varphi_+,\varphi_+)\nabla\varphi_+-\rho(\nabla\varphi_-,\varphi_-)\nabla\varphi_-)\cdot\vn g\ (\vn\text{ is outer normal of }\Omega_+)\\
=&0\ (\text{by free boundary condition})
\end{align*}

In this paper, speed of incident shock will be denoted as $X$, so $X$ is also the $\xi$-coordinate of incident shock on $(\xi,\eta)$-plane. As stated in previous subsection, speed and density of the gas behind incident shock are denoted as 
$u_1$ and $\rho_1$. They satisfy:
\begin{equation}
(X-u_1)\rho_1=X\rho_0\ (\text{Mass Conservation})\label{massconservationforincidentshock}
\end{equation}
\begin{equation}
\rho_1^{\gamma-1}+{1\over 2}u_1^2-u_1X=\rho_0^{\gamma-1}\ (\text{Bernoulli Law})\label{Bernoullilawforincidentshock}
\end{equation}
(\ref{Bernoullilawforincidentshock}) comes from plugging $\phi=0,$ $\phi_\xi=u_1$, $\xi=X$ and $\phi_\eta=\eta=0$ into
(\ref{sonicspeedinintroduction}).

When shock hits a parallel wall ($\sigma=0,\delta=0$ with our notation), the shock will be reflected back, leaving gas static in its wake. This is called {\em\bf normal reflection}. We denote the static gas state behind normally reflected shock as {\Statetwobar}. Density and sonic speed of \Statetwobar\ are denoted as $\overline\rho_2$ and $\overline c_2$ respectively.
They satisfy
\begin{equation}
\overline c_2^2=(\gamma-1)\overline\rho_2^{\gamma-1}.
\end{equation} The speed of reflected shock is denoted as $Z$. $Z$ and $\overline \rho_2$ are determined by 
\begin{equation}
\rho_1(Z+u_1)=\overline\rho_2 Z,\ \ \ \ \ \ \ \ \ \ \ \ (\text{Mass Conservation})\label{massconservationfornormallyreflectedshock}
\end{equation}
\begin{equation}
\rho_0^{\gamma-1}+u_1(Z+X)=\overline\rho_2^{\gamma-1}.\ \ \ \ \ \ \  (\text{Bernoulli Law}) \label{Bernoullilawfornormallyreflectedshock}
\end{equation}
There will be one and only one physically admissible solution to (\ref{massconservationfornormallyreflectedshock}) and (\ref{Bernoullilawfornormallyreflectedshock}), which satisfies
\[\overline \rho_2>\rho_1,\text{ or }Z>0.\]
And by argument in section 3.1 of \cite{CF}, we know 
\begin{equation}
\overline c_2^2>Z^2 \label{subsonicityofnormallyreflectedshock}.
\end{equation}
We will denote $\sqrt{\overline{c}_2^2-Z^2}=Y$, so $(-Z,Y)$ is the coordinate of point, where reflected shock intersects sonic circle.

When incident shock is not parallel to the wall ($\sigma=0, \delta\neq0$, with our notation), a two-shock structure will form. And this is called the {\bf regular reflection}, contrary to the Mach reflection, which is more complicated.

Two regular reflections occur in our case, one is above $\xi$-axis, another is below $\xi$-axis. They are denoted as {\Statetwoplus} and {$\statetwominus$}. The density and velocity  of \Statetwopm\  are denoted as $\rho_2^\pm$ and $(u_2^\pm, v_2^\pm)$ respectively. When $\sigma$ is small enough, the $\rho_2^\pm$ and $(u_2^\pm,v_2^\pm)$, satisfying $\rho_2^\pm>\rho$ are uniquely determined, as shown in section 3.1 of \cite{CF}. And analysis in \cite{CF} also shows or implies
\begin{itemize}
\item  $ u_2^+>0,\ v_2^+>0;\ u_2^->0,\ v_2^-<0$;
\item  $|u_2^+|+|v_2^+|+|u_2^-|+|v_2^-|\leq C_p\sigma$;
\item  $|\rho_2^+-\overline\rho_2|+|\rho_2^--\overline\rho_2|\leq C_p\sigma$.
\end{itemize}
\subsection{Symmetric Case}
\begin{figure}
       \centering
    \includegraphics[height=6cm]{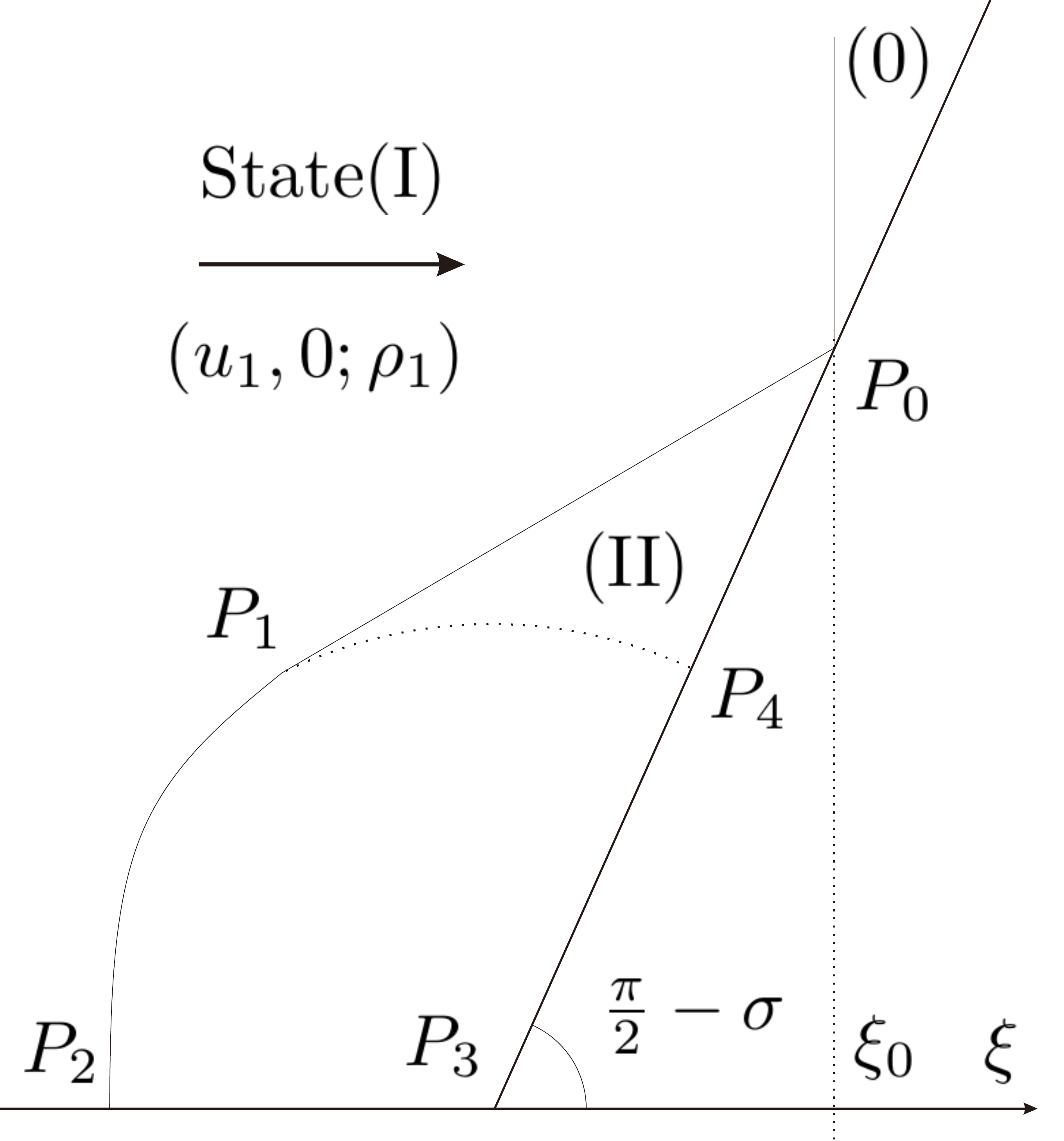}
\caption{Symmetric Case}
\label{fig:SymmetricCase}
\end{figure}
In \cite{CF}, when the wedge is symmetric $(\delta=0)$ and $\sigma$ is very small,  solutions to boundary value problem was rigorously constructed. Later, more general existence results, without requiring $\sigma$ being small, was proved in \cite{CFbook}. In the remaining part of this subsection for the convenience of notation, we denote the region surrounded by arc $A_1A_2$, $A_2A_3$...$A_iA_{i+1}$...$A_kA_1$ by region $A_1A_2...A_k$. And arc $A_1A_2...A_k$ means arc $A_1A_2\cup A_2A_3\cup...A_{k-1}A_k$.

When solving the problem in the symmetric case, many difficulties arise, including free boundary problems and ellipticity degeneration, especially when ellipticity degenerates on free boundary. To overcome these difficulties many new techniques were developed in \cite{CF} and \cite{CFbook}.In \cite{CF}, they first showed the existence of \Statetwo, based on implicit function theorem, when the wedge is close to a half-plane($\sigma$ closes to zero). Then in region $P_1P_2P_3P_4$ they constructed $\varphi$ as a solution to equation (\ref{divergenceformequationforpsedopotential}) (\ref{sonicspeedwithpseudopotential}) with the following boundary condition:
\[
\begin{array}{cc}
\varphi_{\vn}=0, &\text{  \ on  arc } P_2P_3P_4,\\
\varphi=\varphi_2, &\text{ on arc }  P_1P_4,\\
\varphi=\varphi_1, &\text{ on arc }  P_1P_2,\\
(\rho\nabla\varphi-\rho_1\nabla\varphi_1)\cdot \vn=0, &\text{on arc }  P_1P_2.
\end{array}\]

Notice that, in above, arc $P_1P_2$ is a free boundary so we can pose two boundary conditions on it. And arc $P_1P_4$ is a part of the sonic circle of \Statetwo \[P_1P_4\subset\left\{(\xi,\eta)\mid|\xi-u_2|^2+|\eta-v_2|^2=c_2^2\right\}.\] Along it, equation degenerates. And for the solution they constructed, $\nabla\varphi=\nabla\varphi_2$ on arc $P_1P_4$.

If we extend the definition of $\varphi$ to the whole ${\mathbb{R}}^2\setminus W$, by defining
\[\varphi=\left\{
\begin{array}{cc}
\varphi_0, & \text{in } \{\xi>\xi_0\}\setminus W\\
\varphi_1, & \text{in } \{\xi\leq\xi_0\}\setminus (W\cup \text{region } P_0P_1P_2P_3)\\
\varphi_2, & \text{in } \text{region } P_0P_1P_4
\end{array}
\right.\]
then $\varphi$ is a weak solution to Problem 2.

So a natural question is can we solve the boundary value problem for non-symmetric case? And if so, can the solution be as regular as the solution constructed in \cite{CF}?

\subsection{Notations}

Before presenting the result of this paper, we introduce some notations and conventions. Notations introduced here will be universally used through out the paper, except Appendix. Some other notations, which will only be used in a specific section (or subsection), will be introduced later, in the corresponding sections.

In $(\xi,\eta)$-plane, we still use { $W$} to denote the wedge.
 { State($\RN{1}$)} is the state  left to shock (behind incident shock and in front of reflected shock), and  {  State${(\overline{\RN{2}})}$} denotes the state behind normal  reflected shock. Since now we have two regular shock reflections, we use  { State($\RN{2}_+$)} denotes the state above $\xi$-axis, and  use  { State($\RN{2}_-$) }denotes the state below $\xi$-axis. And  { \Statecero} is the state in front of (or right to)  incident shock.

 $ $ {  ${\gammasonic}^\pm$} denotes the sonic circle of State($\RN{2}_{\pm}$), and  { \Gammasonic } denotes their union.
 {  \Gammashock\ } denotes the part of reflected shock between \Gammasonicplus\ and \Gammasonicminus.  And  { $\gammashock^+$} is used to denote the part of reflected shock above \Gammasonicplus. While  { $\gammashock^-$} is used to denote the part of reflected shock below \Gammasonicminus.
 { \Gammawedgeplus}({\Gammawedgeminus})\ denotes upper(lower) part of boundary of wedge, while \Gammawedge\ denotes the union of these two rays.

$ $  {  $\MC$} denotes corner of wedge.
 { $\Omega$} denotes the region surrounded by \Gammashock, \Gammasonic\ and \Gammawedge.
 { $C_p$} is a constant depends only on physical constants, $\rho_1, u_1, \rho_0, \gamma$.

And, for convenience, we may choose different directions for wedge and flow in different parts of this paper, like illustrated in the following pictures. But the angle between symmetry axis of wedge and direction of flow will always be denoted as $\delta$. So in every section we declare either direction of flow or direction of wedge.
\begin{figure}[h]
\begin{minipage}{0.2\linewidth}
    \includegraphics[height=4cm]{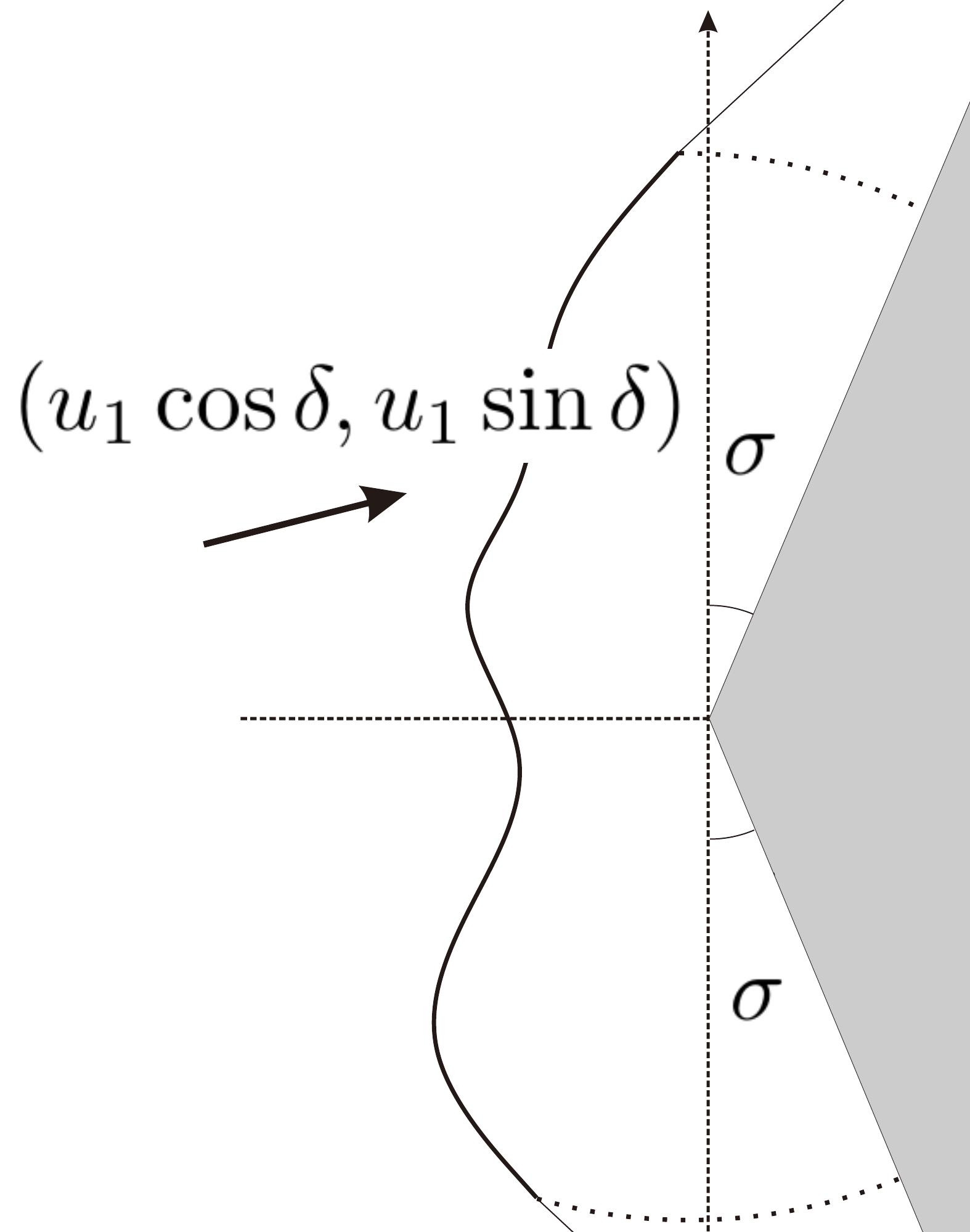}
\caption{Non-Vertical Shock hits Symmetric Wedge}
\label{fig:NonVerticalShockhitsSymmetricWedge}
\end{minipage}\hfill
\begin{minipage}{0.2\linewidth}
    \includegraphics[height=4cm]{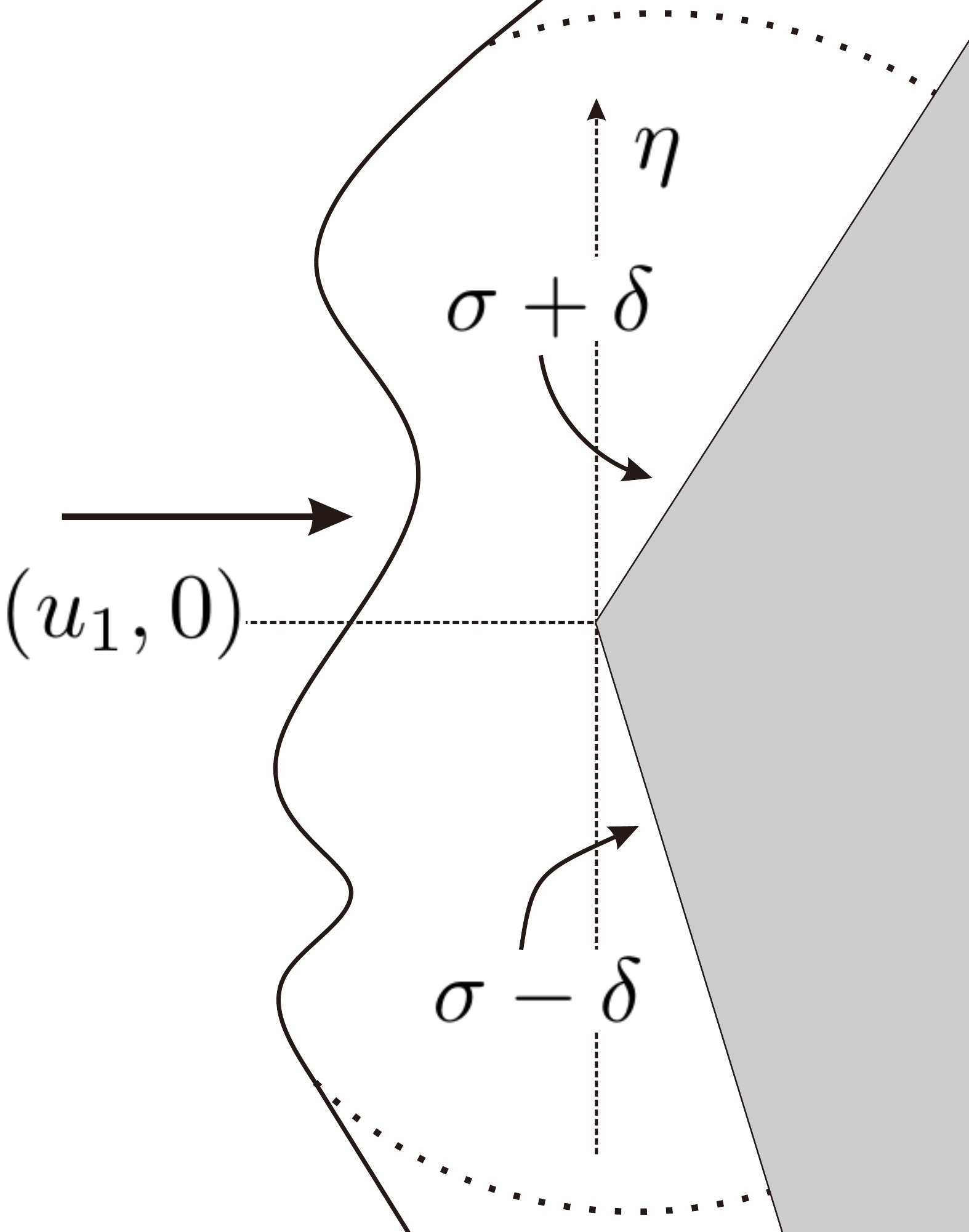}
\caption{Vertical Shock hits Non-Symmetric Wedge}
\label{fig:VerticalShockhitsNonSymmetricWedge}
\end{minipage}\hfill
\begin{minipage}{0.2\linewidth}
    \includegraphics[height=4cm]{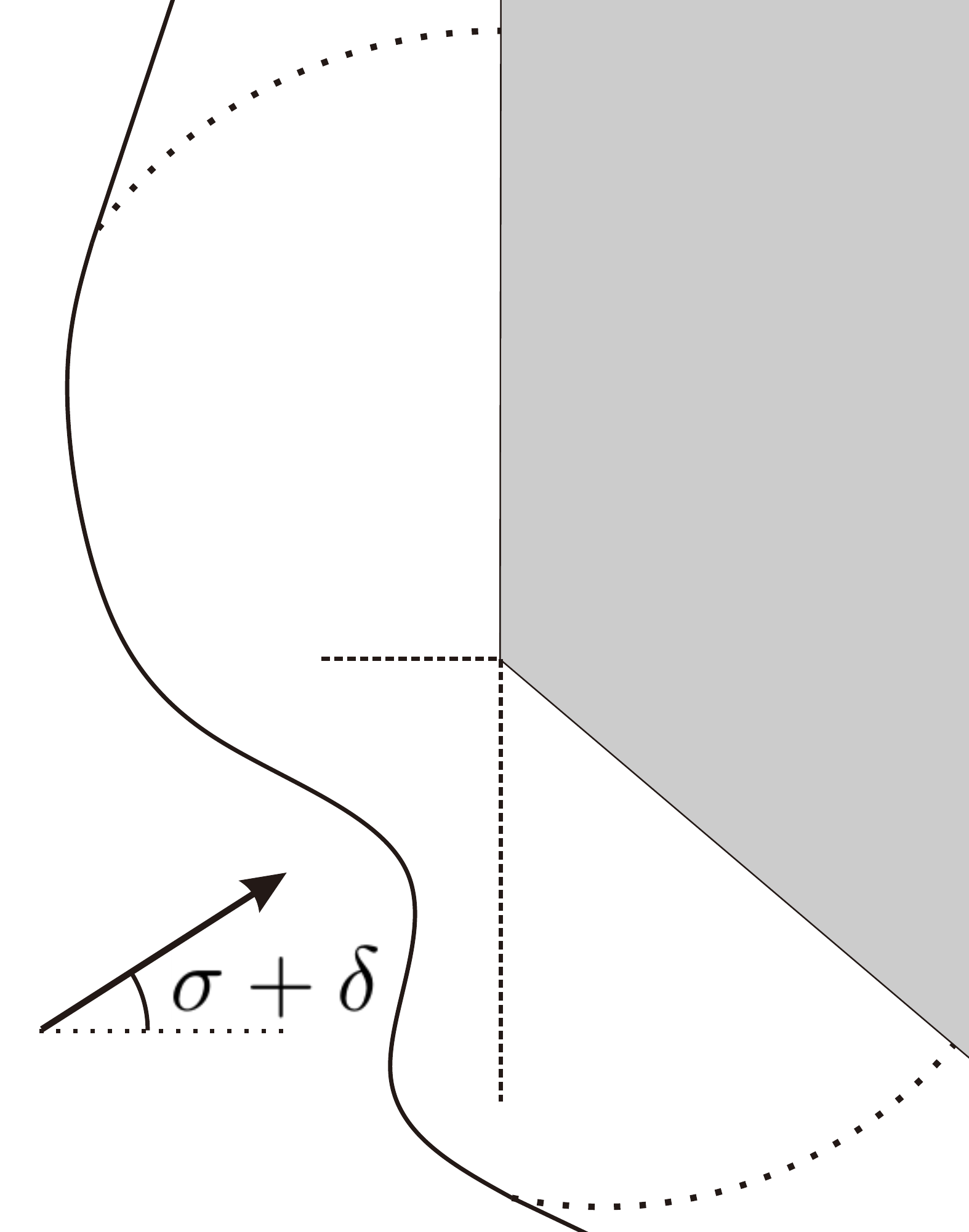}
\caption{\Gammawedgeplus\  coincides with $+\eta$-Axis}
\label{fig:Gammawedgepluscoincideswith+eta-Axis}
\end{minipage}\hfill
\begin{minipage}{0.2\linewidth}
    \includegraphics[height=4cm]{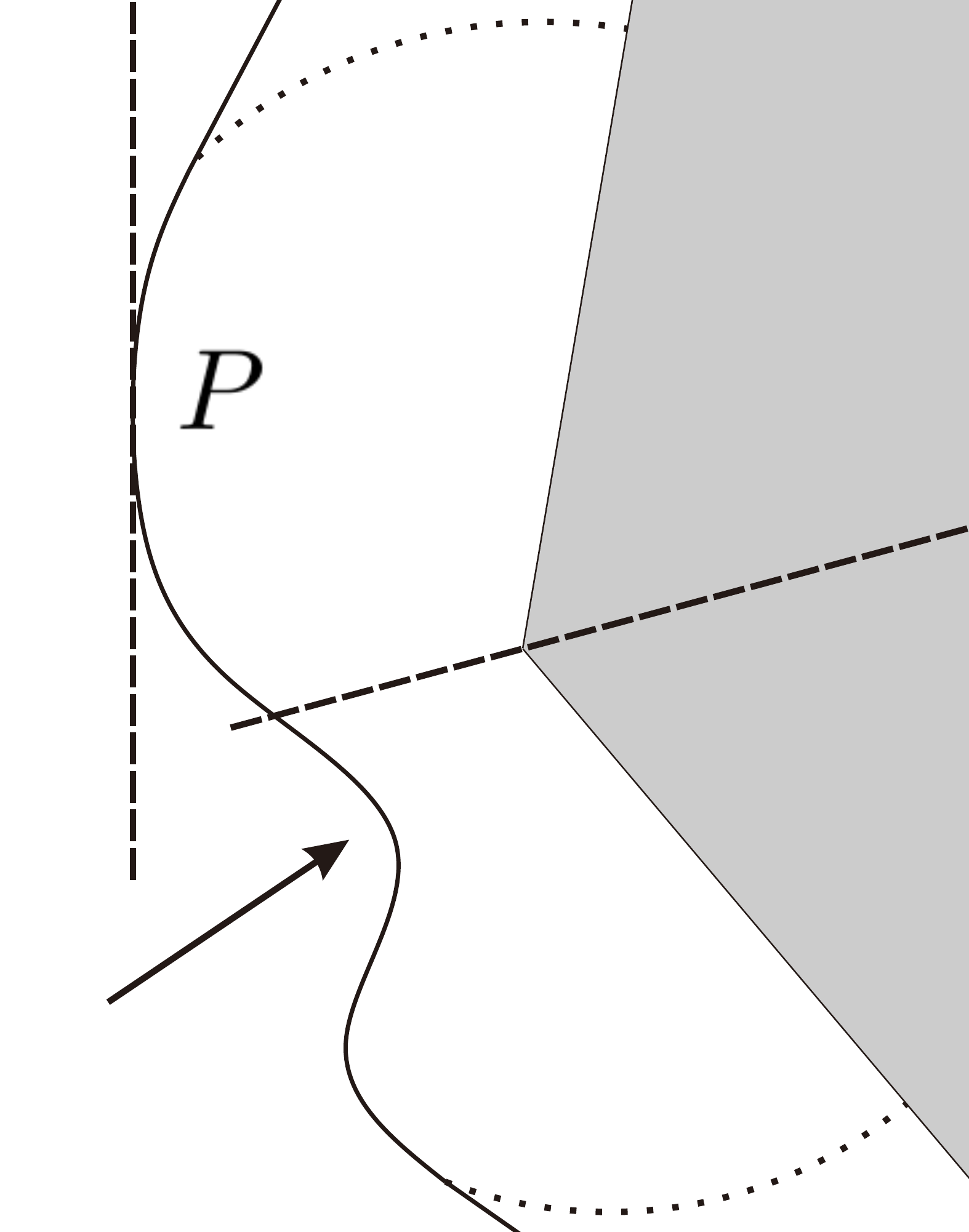}
\caption{A position easy to compute RH condition at $P$}
\label{fig:easytocomputeRHconditionbecausenownormalishorizontal}
\end{minipage}
\end{figure}

\subsection{ Regular Solutions and the Conclusion of the Paper}

{\em In this subsection the direction of the wedge and the flow is shown in Figure \ref{fig:NonVerticalShockhitsSymmetricWedge}, s.t. the symmetry axis of the wedge is $\xi$-axis.}
\begin{definition}{Regular Solution for Shock Reflection}
\label{def:definitionofregularsolution}

We define that a regular solution of self-similar shock reflection problem is a function $\phi\in C^3(\overline\Omega\setminus(\overline{\gammasonic}\cup\mathcal{C}))\cap C^1(\overline\Omega\setminus\mathcal{C})\cap Lip(\overline\Omega)$ which satisfies the following conditions:
\begin{description}
\item[Equation] \begin{equation}\left[c^2-(\phi_\xi-\xi)^2\right]\phi_{\xi\xi}-2(\phi_\xi-\xi)(\phi_\eta-\eta)\phi_{\xi\eta}+\left[c^2-(\phi_\eta-\eta)^2\right]\phi_{\eta\eta}=0\text{,  in $\Omega$,}
\label{quasilinearequationofphi}
\end{equation}
where, 
  \begin{equation}c^2=(\gamma-1)\rho^{\gamma-1}=(\gamma-1)
\left(\rho_0^{\gamma-1}-\phi+\phi_\xi\xi+\phi_\eta\eta-\frac{|\nabla\phi|^2}{2}\right),
\label{sonicspeed}
\end{equation} and equivalently  if we define $\varphi=\phi-\frac{\xi^2+\eta^2}{2}$,
relations above can be written as 
\[\nabla\cdot[\rho\nabla\varphi]+2\rho=0,\ \text{in $\Omega$},\]
where,
\[ c^2=(\gamma-1)\rho^{\gamma-1}=(\gamma-1)\left(\rho_0^{\gamma-1}-{1\over 2}|\nabla\varphi|^2-\varphi\right);\] 

\item[Subsonic Condition] 
\begin{equation}|\nabla\varphi|^2<c^2\text{, in } \overline\Omega\setminus\overline{\gammasonic};
\label{subsoniccondition}      
\end{equation}

\item[Continuity Conditions on Sonic Circles]
\begin{equation}\phi=\phi_2^{\pm},\ \ \text{ on $\gammasonic^\pm $},\label{continuityofpotentialfunctiononsoniccircle}
\end{equation}
\begin{equation}\nabla\phi=\nabla\phi_2^{\pm},\ \ \text{ on $\gammasonic^\pm $};\label{continuityofvelocityonsoniccircle}
\end{equation}

\item[Free Boundary Conditions on the Shock (Non-Vorticity and RH Condition)]
\begin{align}\varphi=\varphi_1,\ \ \ [\rho_1\nabla\varphi_1-\rho\nabla\varphi]\cdot \vn=0, \text{ on }\gammashock,\label{nonvorticityassumptiononshock}
\end{align}
or equivalently,
\begin{equation} \phi=\phi_1, \ \ \ [\rho_1(u_1\cos\delta-\xi,u_1\sin\delta-\eta)-\rho(u-\xi,v-\eta)]\cdot \vn=0,\text{ on }\gammashock;
\end{equation}

\item[Slip (or Neumann) Boundary Condition on the Boundary of Wedge]
\begin{equation}\phi_{\vn}=0 \text{, on } \gammawedge;
\end{equation}

\item[Regularity Assumption of Shock] 
\begin{equation}\centering
\text{\parbox{.8 \textwidth}
{$\gammashock$ is not tangential to the direction of the upcoming flow at any point, and it's a $C^1$ curve up to its ends;}}
\label{regularityassumptionofshock}
\end{equation}
\item[Admissible Condition] \begin{equation}\rho_1<\rho \text{, in } \overline\Omega \label{densityassumptionofregularsolution}.
\end{equation}

\end{description}

\end{definition}

If we can find such a regular solution $\phi$, then we can combine it with potential functions of \Stateone, \Statetwopm\ \ and \Statecero\  to form a function from ${\mathbb{R}}^2\setminus W$ to ${\mathbb{R}}$, which will be a weak solution to  Problem \ref{prob:2}.

In this paper we prove
\begin{theorem}[Main Theorem] \label{MainTheorem}
For isentropic gases (with $c^2=(\gamma-1)\rho^{\gamma-1}$), given a vertical shock with the upstream state $(u_1,0;\rho_1)$ and the downstream state $(0,0;\rho_0)$, we can find $\epsilon_T>0$ small enough, which depends on $\rho_0,\rho_1,u_1$ and $\gamma$, s.t. when the shock hits a convex wedge, with vertex angle $\pi-2\sigma(0<\sigma<\epsilon_T)$ and if the symmetry axis of the wedge forms an angle $\delta(0<\delta\leq \sigma)$ with the direction of the upstream flow, the shock reflection problem does not have a regular solution as defined in Definition \ref{def:definitionofregularsolution}.
\end{theorem}
\subsection{Structure of the Paper}
Section \ref{regularity} provides some relatively rough estimates. Some of these estimates are used to show more regularities of regular solution, some of them imply that when $\sigma$ tends to zero, our corresponding regular solution tends to potential function of normal reflection which is a constant.

With these estimates, we can have a good control on regular solution in later sections.

In section \ref{asymptotic}, with computation we show precisely how geometric structure differs from symmetric shock reflection. And so we can tell the relative position of \Gammasonicplus\ and $\widetilde{\Gamma_{\text{sonic}}^-}$, here $\widetilde{\Gamma_{\text{sonic}}^-}$ is the reflection of ${\gammasonicminus}$ across symmetry axis of wedge. And the geometric structure will be useful when doing integral by parts in section \ref{symmetricestimate}. The technique in this section depends on comparison of derivatives, at $\delta$ or $\sigma=0$, so it's only available for $\sigma$ and $\delta$ small.

In section \ref{symmetricestimate}, we proved ``symmetric estimate". It says, if we take the symmetry axis of wedge as $\xi$-axis, regular solution is almost symmetric w.r.t. $\eta$, precisely, near  symmetry axis of wedge,
\[|\phi(\xi,\eta)-\phi(\xi,-\eta)|\leq C_p\delta.\]

To prove above, we first get an integral type symmetric estimate. With method of \cite{CFcomparison} we are able to control $\int_{\Omega\cap\widetilde\Omega}|\phi-\tilde\phi|^3$ by a divergence integral,  so we can reduce it to integral on boundary of $\Omega\cap\widetilde{\Omega}$, then making use of RH condition and $\nabla\phi=\nabla\phi_2^{\pm}$ on $\gammasonic^{\pm}$, we can further transform boundary integral into integral in $\Omega\triangle\widetilde\Omega$ and on $\partial(\Omega\cup\widetilde\Omega)$. These integrals can either be estimated with the properties of $RH$ function derived in section \ref{regularity}, or by explicit computation and boundary condition.

 Then we get $L^\infty$ symmetric estimate near corner of wedge, with Moser iteration; and around where shock intersects symmetry axis, with a generalized Krylov-Safonov estimate (Lemma \ref{lem:BoundaryKrylovSafonov}). From this $L^\infty$ estimate, we get on symmetry axis of wedge, $|\phi_\eta|\leq C_p\delta$ near shock and $ |\phi_\eta|\leq C_p\delta r^{{\alpha\over 2}-1} $ near corner of wedge. 

Above gradient estimates will be used in section \ref{antisymmetricestimate}, to control boundary integral.

In section \ref{antisymmetricestimate}, we provided a lower bound for $\phi-\widetilde{\phi}$, precisely:

\begin{equation}\phi(\xi,\eta)-\phi(\xi,-\eta) \geq \delta({\eta \over C_p}-C_p\sigma^{1\over 3})\label{antisymmetricestimateinintroduction}
\end{equation}
 So contrary to section \ref{symmetricestimate}, we call this estimate as ``antisymmetric estimate".
To prove above, roughly speaking, we minus a linear function $\eta\sin\delta$ from $\phi$, and denote the new function as $\psi$, then estimate $\psi(\xi,-\eta)-\psi(\xi,\eta)$ from above. We first do the integral estimate as in last section, then with Moser Iteration we get an $L^\infty$ estimate on $\partial B_{r_0}$, $r_0$ is smaller then some physical constants, decided in Lemma \ref{lem:maximumprinciple} and \ref{lem:singularityestimate}.

 With (\ref{antisymmetricestimateinintroduction}), we can use Lemma \ref{lem:singularityestimate} to get a contradiction.

In appendix are some lemmas about linear partial differential equations.

Lemma \ref{lem:maximumprinciple} \ref{Lemma A2} and \ref{lem:singularityestimate} are all about the solution of linear elliptic partial differential equation with singular coefficients near a corner.
Lemma \ref{lem:maximumprinciple} is a maximum principle, Lemma \ref{Lemma A2} is an estimate based on conformal mapping, it gives a singular bound on $|\phi_\eta|$ and Lemma \ref{lem:singularityestimate} implies that out of a convex wedge, a $C^1$ solution to our linear partial differential equation, should not be $C^{1,\alpha}$, if solution is bigger on one side of the symmetry axis of wedge.

Lemma \ref{lem:BoundaryKrylovSafonov} is a generalized Krylov-Safonov estimate, which is used to get $L^\infty$-symmetric estimate near free boundary.



\section{Fundamental Estimates for Regular Solutions}\label{regularity}
In this section, we provide some fundamental estimates for regular solutions defined in {Definition \ref{def:definitionofregularsolution}}.  Roughly speaking, with estimates in this section we have a better control on regular solution.
\begin{itemize}
\item We show the shock reflection corresponds to our regular solution is close to  normal shock reflection (subsection\ref{gradientestimate}), 
\item we provide some regularity estimate of potential function (subsection \ref{continuityofvelocityatcornerofwedge}, \ref{holdergradientestimateatcornerofwedge}, \ref{holdergradientestimateawayfromsoniccircleandcornerofwedge},   \ref{C2chiestimateawayfromsoniccircleandcornerofwedge})
\item we provide a ``coercive'' estimate of the RH function, in subsection \ref{estimateofRHfunction}, which will be used  in section \ref{symmetricestimate} and section \ref{antisymmetricestimate} to estimate boundary integral.
\end{itemize}
\subsection{Continuity of Velocity at the Corner of the Wedge}\label{continuityofvelocityatcornerofwedge}
Based on Estimate in \cite{CFHX}, a regular solution to potential flow equation, which is only assumed to be Lipschitz at corner of wedge should  actually be $C^1$. And so at corner of wedge,
\statementstarts{ $\phi_\xi=u=\phi_\eta=v=0$.\label{cornerspeed}  }\statementends
We cannot immediately get $|\nabla\phi|(\xi,\eta)\leq C_p(|\xi|+|\eta|)^{\alpha}$, which is contained in Proposition \ref{prop:gradientestimateatcornerofwedge}, because now we don't have any control on $|\nabla\phi|_0$ in $\overline\Omega$. We will more precisely estimate $\nabla\phi$ until section \ref{holdergradientestimateatcornerofwedge}.

\subsection{Comparison of $\phi$ and $\phi_1$}
{\em In this subsection, position of wedge and flow is as shown in Figure \ref{fig:VerticalShockhitsNonSymmetricWedge}, so velocity of upstream flow is $(u_1,0)$.} And we want to show $\phi<\phi_1$ in $\Omega$. 

First, in $\Omega$, $\phi$ satisfies (\ref{quasilinearequationofphi}). In this paper, except Appendix, when we only use linear property of (\ref{quasilinearequationofphi}), we denote this linear equation by 
\begin{equation}a_{ij}\phi_{ij}=0.\label{linearequationofphi}
\end{equation} ${a_{ij}}'s$ are polynomials of $\xi,\eta,u,v$, so $a_{ij}\in C^1(\overline\Omega\setminus(\overline{\gammasonic}\cup\mathcal{C}))$. 

So,
\begin{equation}
a_{ij}(\phi-\phi_1)_{ij}=a_{ij}\phi_{ij}=0,\label{equationofphiminusphione}
\end{equation}
since $\phi_1$ is a linear function. So
\statementstarts
{$\phi-\phi_1$ cannot achieve maximum in $\Omega$}\label{phiminusphionecannotachievemaxinomega}.
\statementends

Then, on \Gammasonicpm, 
\[\phi-\phi_1\underset{(\ref{continuityofpotentialfunctiononsoniccircle})}=\phi_2^\pm-\phi_1.\]
Since in section \ref{incidentshocknormalandregularreflection}, we have shown that $\phi_1=\phi_2^\pm$ on \Gammashockpm, and when $\sigma$ small enough $|\nabla\phi_2^\pm|\leq C_p\sigma$, so on right-hand side of \Gammashockpm
\[\phi^{\pm}_{2,\xi}<\phi_{1,\xi}=u_1,\ \text{(providing $\sigma$ small enough)}\]
and so 
\begin{equation}
\phi=\phi_2^{\pm}<\phi_1 \text{ on \Gammasonicpm.}\label{phismallerthanphioneongammasonic}
\end{equation}

For $\vn_{\pm}=(\cos(\sigma\pm\delta),\mp\sin(\sigma\pm\delta))$, which are outer unit normal vectors of $\Omega$ on \Gammawedgepm, 
\[(\phi-\phi_1)_{\vn_{\pm}}=\phi_{\vn_{\pm}}-\phi_{1,\vn_\pm}=-u_1\cos(\sigma\pm\delta)<0,\ \text{on \Gammawedgepm} \text{\ (providing $\sigma$ small enough)}\]
So,
\statementstarts
{$\phi-\phi_1$ cannot  achieve maximum on ${\gammawedgepm}$}.\label{phiminusphionecannotachievemaximumongammawedge}
\statementends
Note, that at corner of wedge when $\overline\gammawedgeplus$ meets $\overline\gammawedgeminus$, we can apply both $\vn_+$ and $\vn_-$ to $\phi$.

Combine (\ref{equationofphiminusphione})
		(\ref{phiminusphionecannotachievemaxinomega})
		(\ref{phismallerthanphioneongammasonic})
		  (\ref{phiminusphionecannotachievemaximumongammawedge}) 
	and (\ref{nonvorticityassumptiononshock}) together, we get
\begin{equation}
\phi<\phi_1\text{, in } \Omega\label{philessthanphioneinOmega}
\end{equation}

\subsection{Gradient Estimates}
\label{gradientestimate}
{\em In this subsection the direction of upcoming flow  is $(u_1,0)$, as shown in Figure {\ref{fig:VerticalShockhitsNonSymmetricWedge}}.}

Take $\xi-$derivative of  linear equation of $\phi$, (\ref{linearequationofphi}), we get:
\begin{equation}
a_{11}u_{\xi\xi}+2a_{12}u_{\xi\eta}+a_{22}u_{\eta\eta}+a_{22}(\frac{a_{11}}{a_{22}})_{\xi}u_{\xi}+a_{22}(\frac{2a_{12}}{a_{22}})_{\xi}u_{\eta}=0.
\label{linearequationofu}
\end{equation}
\statementstarts{So $u$ can't achieve maximum(or minimum) in $\Omega^{\circ}$ .}\label{uinterior}\statementends
If at some interior point of \Gammawedgeplus, $u$ achieves a local maximum (or minimum), then at this point,
\[u_{\vt}=0,\ D_{\vt}(\phi_{\vn})=0,\ a_{ij}\phi_{ij}=0.\]
In above $\vt=(\sin(\sigma+\delta),\cos(\sigma+\delta)),\ \vn=(\cos(\sigma+\delta),-\sin(\sigma+\delta))$, write above  equations in matrix form:
\[
\left(\begin{array}{ccc}
         \sin(\sigma+\delta)                                   &\cos(\sigma+\delta)                          &0\\
        \sin(\sigma+\delta)\cos(\sigma+\delta)&\cos^2(\sigma+\delta)-\sin^2(\sigma+\delta)&-\sin(\sigma+\delta)\cos(\sigma+\delta)\\
        a_{11}                                                          &2a_{12}                                                                 &a_{22}
        \end{array}
\right)
\left(\begin{array}{c}u_\xi\\u_\eta\\v_\eta\end{array}\right)=0.
\]
Determinant of this $3\times 3$ matrix equals
\[-\sin(\sigma+\delta)[a_{11}\cos^2(\sigma+\delta)-2a_{12}\sin(\sigma+\delta)\cos(\sigma+\delta)+a_{22}\sin^2(\sigma+\delta)]\neq0.\]
This implies at this maximum point $\nabla u=0$, which is a contradiction to Hopf lemma.
\statementstarts{So $u$ can't achieve maximum(or minimum) value at any  interior point of \Gammawedgeplus.
\label{uwedgeplus}}\statementends
\statementstarts{
Similarly, when $\sigma\neq\delta$,
 $u$ can't achieve any maximum(or minimum) value on \Gammawedgeminus,
}\label{uwedgeminus}\statementends
\statementstarts{ and when $\sigma=\delta$, $u=0$ on \Gammawedgeminus}.\label{deltaequalssigma}\statementends
On \Gammashock, $\phi-\phi_1=0$ (\ref{nonvorticityassumptiononshock}); in $\Omega$, $\phi-\phi_1<0$ (\ref{philessthanphioneinOmega}),
\[a_{ij}(\phi-\phi_{1})_{ij}=0\ \text{(because $\phi_1$ is a linear function),}\]
 and since we know now \Gammashock\ is the graph of a $C^1$ function of $\eta$ (\ref{regularityassumptionofshock}), so we can apply Hopf maximum principle, to $\phi-\phi_1$, and get
 \begin{equation}u=\phi_\xi<\phi_{1,\xi}=u_1, \ \text{on } \gammashock-\overline{\gammasonic^\pm}\label{ushock}.\end{equation}
Put (\ref{cornerspeed}) (\ref{uinterior}) (\ref{uwedgeplus}) (\ref{uwedgeminus}) (\ref{deltaequalssigma}) (\ref{ushock})  together, and use the assumption, that regular solution is $C^1$ on sonic circle, we have : when $\sigma,\delta$ small enough s.t. $u_2^\pm<u_1$, 
\begin{equation}u<u_1, \ \text{in}\ \overline\Omega\label{upperboundofu}.\end{equation}

Now with (\ref{upperboundofu}) we can define $S=\FS$ in $\overline\Omega$, and since $\phiisregular$ and $|u|=|u_2^\pm|<u_1$ on $\gammasonic^\pm$, $S\in C^0(\overline\Omega)\cap C^2(\overline\Omega\setminus(\overline\gammasonic\cup\mathcal{C})) $.

In $\overline\Omega\setminus(\overline\gammasonic\cup\mathcal{C})$, we have
\[
\left\{\begin{array}{c}
        S_\xi=S_u u_\xi+S_v v_\xi\\
        S_\eta=S_u u_\eta+S_v v_\eta\\
        \linearequationofphi
        \end{array}\right.
\]
Solve above equation of $u_\xi, v_\xi, v_\eta$, we get (we know $S_v=\frac{1}{u_1-u}$ is also well defined and $S_v\neq0$)
\begin{align}\label{derivativeofS}
\left(\begin{array}{c}u_\xi \\v_\xi \\ v_\eta
        \end{array}                                    \right)={1 \over a_{11}S_v^2-2a_{12}S_u S_v+a_{22}S_u^2}\left(\begin{array}{cc}S_u a_{22}-2S_va_{12}&-S_va_{22}\\
																											a_{11}S_v                        &S_ua_{22}\\
																											-a_{11}S_u 			  &-2S_ua_{12}+S_va_{11}
																							\end{array}
                                                                                                                                                                                \right)
\left(\begin{array}{c}
        S_\xi\\S_\eta
        \end{array}
\right)
.\end{align}
Plug
\[\equationofu\]
\[\equationofv,\] into:
\begin{align*}
  &a_{11}S_{\xi\xi}+2a_{12}S_{\xi\eta}+a_{22}S_{\eta\eta}\\
=&a_{11}(S_u u_{\xi\xi}+S_v v_{\xi\xi}+S_{uu}u_\xi^2+2S_{uv}u_\xi v_\xi+S_{vv} v_\xi^2)\\
  &+2a_{12}(S_u u_{\xi\eta}+S_v v_{\xi\eta}+S_{uu}u_\xi u_\eta+S_{uv}u_\xi v_\eta+S_{uv}u_\eta v_{\xi}+S_{vv} v_\xi v_\eta)\\
  &+a_{22}(S_u u_{\eta\eta}+S_v v_{\eta\eta}+S_{uu}u_\eta^2+2S_{uv}u_\eta v_\eta+S_{vv} v_\eta^2)
\end{align*}
and replace $u_\xi, v_\xi=u_\eta, v_\eta$ by linear combination of  $S_\xi, S_\eta$, using (\ref{derivativeofS}). We get there exists $ \hat b_i's$, which are $C^0$ functions in $\overline\Omega\setminus(\overline\gammasonic\cup\mathcal{C})$, s.t.
\begin{equation}a_{ij}S_{ij}+\hat b_i S_i=0.\label{equationofS}
\end{equation}
\statementstarts{So $S$ can't achieve maximum or minimum at interior point of $\Omega$.}\label{Sinterior}\statementends
 
To show that $S$ can not achieve maximum or minimum on \Gammashock\   either,  we define 
\begin{equation}
RH=[\rho(u-\xi,v-\eta)-\rho_1(u_1-\xi,-\eta)](u-u_1,v), \text{ in $\overline{\Omega}$}
 \label{RH}
\end{equation} in the expression above, 
\begin{equation} \rho=\left[\rho_0^{\gamma-1}-\phi+\phi_\xi\xi+\phi_\eta\eta-\frac{\phi_\xi^2+\phi_\eta^2}{2}\right]^{1 \over \gamma-1}\label{rho}\end{equation}
$RH$ can be considered either as a function of $\xi, \eta$ or a function of five variables, $\phi, u, v, \xi, \eta$, as expressed in (\ref{RH}) and (\ref{rho}).  To avoid confusion, when RH is considered as a function of five variables,  $\xi(\eta\text{ and }T)$-derivative of RH will be denoted as $RH_{(\xi)}(RH_{(\eta)}\text{ and } RH_{(T)})$ respectively.  Taking tangential derivative of  RH on \Gammashock\ gives: 
\begin{equation}
RH_T=RH_u u_T+RH_v v_T+RH_\phi \phi_T+RH_{(T)}\ \ (\text{here }T=(v, u_1-u)),\label{tangentialderivativeofRHonshock}
\end{equation} 
\begin{itemize}
\item
on left-hand side of (\ref{tangentialderivativeofRHonshock}), RH is considered as a pure function of $\xi, \eta$, 
\item
on right-hand side of  (\ref{tangentialderivativeofRHonshock}), RH is considered as a function of $\phi, u, v, \xi, \eta$. 
\end{itemize}
Then algebraic computation gives, when RH is considered as a function of $\phi, u, v, \xi, \eta$:
\[RH_\phi=-[(u-\xi)(u-u_1)+(v-\eta)v]\frac{\rho^{2-\gamma}}{\gamma-1},\]
\[RH_{(\xi)}=(\rho_1-\rho)(u-u_1)+[(u-\xi)(u-u_1)+(v-\eta)v]\frac{\rho^{2-\gamma}}{\gamma-1}u,\]
\[RH_{(\eta)}=v(\rho_1-\rho)+[(u-\xi)(u-u_1)+(v-\eta)v]\frac{\rho^{2-\gamma}}{\gamma-1}v.\]
Combine these with 
\[\phi_T=uv+(u_1-u)v=u_1v, \] 
and RH condition, we know on right-hand side of  (\ref{tangentialderivativeofRHonshock}) 
\[RH_\phi\phi_T+RH_{(T)}=0.\]
So along \Gammashock 
\begin{equation}
0=RH_T=RH_u u_T+RH_v v_T.\label{reducedtangentialderivativeofRHonshock}
\end{equation}
If $S$ achieves maximum or minimum at some interior point of \Gammashock, at this point:
\begin{equation}D_TS=S_u u_T+S_v v_T=0,\label{tangentialderivativeofSonshock}
\end{equation}
We put (\ref{tangentialderivativeofSonshock}) (\ref{reducedtangentialderivativeofRHonshock}) and linear equation of $\phi$ (\ref{linearequationofphi}) together in matrix form:
\[
\left(\begin{array}{ccc}
       RH_u v&RH_v v+RH_u (u_1-u)&RH_v(u_1-u)\\
       S_u v   &S_v v+S_u(u_1-u)        &S_v(u_1-u)\\
       a_{11}   & 2a_{12   }                      &a_{22}
         \end{array}
\right)
\left(\begin{array}{c}
      u_\xi \\  v_\xi \\v_\eta
         \end{array}
\right)=0.
\]
Plug $S_u=\frac{v}{(u_1-u)^2}, S_v=\frac{1}{u_1-u}$ into the $3\times3$ matrix above, with computation we get its determinant is
\begin{align*}
       &-(RH_u(u-u_1)+RH_v v)(a_{11}-2a_{12}\frac{v}{u_1-u}+a_{22}\frac{v^2}{(u_1-u)^2})\\
\leq&-\rho[(u-u_1)^2+v^2](1-\frac{(u-\xi)^2+(v-\eta)^2}{c^2})(a_{11}-2a_{12}\frac{v}{u_1-u}+a_{22}\frac{v^2}{(u_1-u)^2})<0
\end{align*}
So at such minimum or maximum point, $D^2\phi=0,\  S_{\vn}=0$, which contradicts with Hopf lemma, since $S$ satisfies (\ref{equationofS}) in $\Omega$.
\statementstarts{So $S$ can not achieve a minimum or maximum at any interior point of \Gammashock}\label{maxminofSonshock}.
\statementends
Then if $S$ achieves maximum or minimum at some interior point of \Gammawedgeplus, at this point we have:
\[ D_\vt\phi_{\vn}=0, \  D_\vt S=0,\   \linearequationofphi,\] in above $\vt=(\sin(\sigma+\delta),\cos(\sigma+\delta))$, again we write them  in matrix form:
\[
\left(\begin{array}{ccc}
       \sin\spd\cos\spd& 	  \cos^2\spd-\sin^2\spd&		-\sin\spd\cos\spd\\
       S_u\sin\spd 		  &S_v \sin\spd+S_u\cos\spd       &S_v\cos\spd\\
       a_{11}   				 & 2a_{12   }              			        &a_{22}
         \end{array}
\right)
\left(\begin{array}{c}
      u_\xi \\  v_\xi \\v_\eta
         \end{array}
\right)=0,
\]
with computation we get the determinant of  the $3\times 3$ matrix above is:
\begin{align*}
   &(a_{11}\cos^2\spd-2a_{12}\sin\spd\cos\spd+a_{22}\sin^2\spd)(S_v \cos\spd+S_u\sin\spd)\\
=&(a_{11}\cos^2\spd-2a_{12}\sin\spd\cos\spd+a_{22}\sin^2\spd)\frac{u_1\cos\spd}{(u_1-u)^2}>0.
\end{align*}
In above we used $D_{\vn}\phi=v\sin\spd-u\cos\spd=0$.
Again we get at such a maximum(or minimum) point $D^2\phi=0, $ so $ D_{\vn}S=0$, which is a contradiction to Hopf lemma.
With same method we can show $S$ cannot achieve maximum or minimum on $(\gammawedgeminus)^{\circ}$

Combine with (\ref{maxminofSonshock}), we know $S$ can only achieve maximum or minimum on $\overline\gammasonic\cup \MC$, so
 \begin{align} -C_p\sigma\leq\frac{v_2^-}{u_1-u_2^-}<S=\frac{v}{u_1-u}<\frac{v_2^+}{u_1-u_2^+}\leq C_p \sigma, \text{ in }\Omega.
\label{estimateofS}\end{align}
(\ref{estimateofS}) implies, providing $\sigma$ small enough, if $\vn=(n_1,n_2)$ is the unit normal direction of \Gammashock, pointing rightward, then 
\begin{align}n_1>{1\over 2}, |n_2|\leq C_p\sigma.\label{estimateofnormaldirection}\end{align}
This implies,
\begin{equation}
{1\over\left|\text{slope of }\gammashock\right|}\leq C_p\sigma  \label{estimateofslopeofshock},
\end{equation}
so,
\statementstarts{ \Gammashock\ stays in a $C_p\sigma$ neighborhood of normal reflected shock.}\label{positionofshock}
\statementends
And (\ref{estimateofslopeofshock}) (\ref{regularityassumptionofshock}) also implies
\begin{equation}
\text{ {$\gammashock$ is a $C^1$ function of $\eta$,}}
\label{gammashockisaconefunctionofeta}
\end{equation}
 in any of 4 positions shown in Figure \ref{fig:NonVerticalShockhitsSymmetricWedge}, \ref{fig:VerticalShockhitsNonSymmetricWedge}, \ref{fig:Gammawedgepluscoincideswith+eta-Axis} and \ref{fig:easytocomputeRHconditionbecausenownormalishorizontal}.
Plug (\ref{estimateofnormaldirection}) into $RH$ condition:
\[[\rho(u-\xi,v-\eta)-\rho_1(u_1-\xi,-\eta)](n_1,n_2)=0,\]
we get,
\[|\rho(u-\xi)-\rho_1(u_1-\xi)|\leq C_p\sigma\rho|v-\eta|+C_p\\
\leq C_p\sigma\rho(u_1-u)+C_p\sigma\rho+C_p,\]
so,
\[(1-C_p\sigma)\rho u\geq \rho(\xi-C_p)-C_p,\]
which implies:
\begin{align} u\geq {-C_p}\text{, on } \gammashock ,\label{lowerboundofu}\end{align}
since $\rho\geq\rho_1$.
Combining (\ref{lowerboundofu}) with $|\frac{v}{u_1-u}|<C_p\sigma$ (\ref{estimateofS}), and $u<u_1$ (\ref{upperboundofu}), gives
\begin{equation}|v|<C_p\sigma\label{estimateofv}.\end{equation}

\begin{figure}
       \centering\label{graphofR}
    \includegraphics[height=3.5cm]{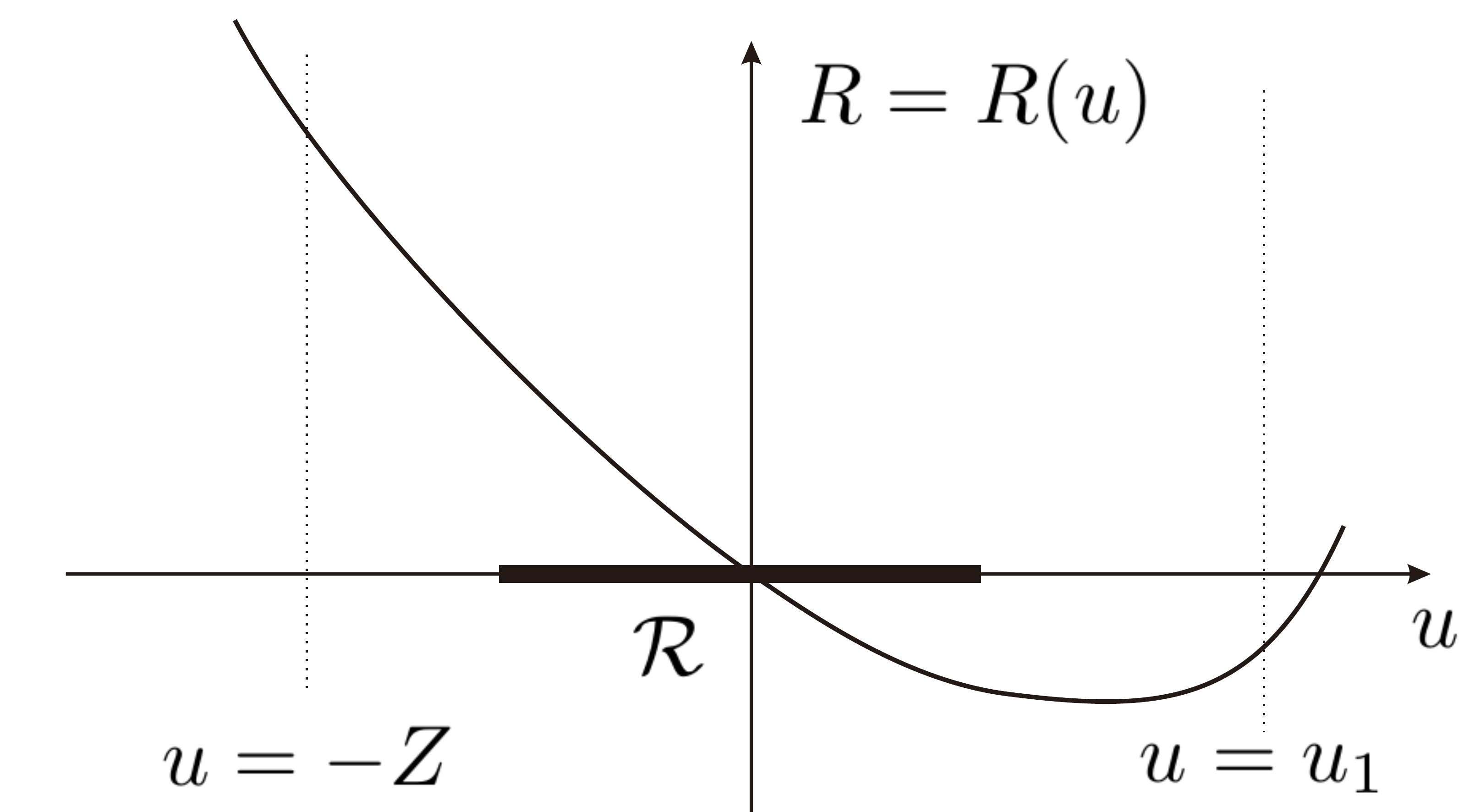}
    \caption{Graph of $R$}
\end{figure}
Now with estimate of $v$ in hand we want to get a same level estimate of $u$, to do so, we define: 
\begin{align}R(u)=\overline\rho_2 Z-(\overline\rho_2^{\gamma-1}-u Z-\frac{u^2}{2})^{1\over \gamma-1}(u+ Z),\label{definitionofR}\end{align}
here, $-Z$ is the $\xi$-coordinate of normal reflected shock, $\overline\rho_2$ is the density of normal shock reflection.
\begin{align}
\overline\rho_2Z=\rho_1(u_1+Z)  \label{RHfornormalshockreflection}
\end{align}
is the RH condition for normal shock reflection.
\begin{align}
\rho_0^{\gamma-1}-\overline\phi_2=\overline\rho_2^{\gamma-1} \label{Bernoullilawfornormalshockreflection}
\end{align}
is Bernoulli law for normal reflection.\\
Now in (\ref{rho}), replacing $\rho_0$ by $\overline\rho_2$ with (\ref{Bernoullilawfornormalshockreflection}) gives on \Gammashock
\begin{equation}\left(\overline\rho_2^{\gamma-1}+\overline\phi_2-\phi_1+u\xi+v\eta-\frac{u^2+v^2}{2}\right)^{1\over\gamma-1}=\rho\label{rhoexpressedbyrhotwo}
\end{equation}
Since we know , $|u|\leq C_p$ (\ref{upperboundofu}) (\ref{lowerboundofu}),  $|v|\leq C_p\sigma$ (\ref{estimateofv}), $\phi=\phi_1$ on \Gammashock (\ref{nonvorticityassumptiononshock}) and \Gammashock\ is close to normal reflected shock (\ref{positionofshock}), we can estimate
\begin{equation}
\overline\rho_2^{\gamma-1}-u Z-\frac{u^2}{2}=\rho^{\gamma-1} +O(\sigma) \label{rhorelatedtonormalreflectionbyBernoulli}.
\end{equation}

(\ref{rhorelatedtonormalreflectionbyBernoulli}) implies:
first, that 
\begin{equation}\overline\rho_2^{\gamma-1}-u Z-\frac{u^2}{2}>\rho_1^{\gamma-1}-C_p\sigma>\frac{\rho_1^{\gamma-1}}{2}
 \ \ (\text{providing $\sigma$ small enough})\label{lowerboundofm}\end{equation}so $R(u)$ is well defined; second,
\begin{equation}\rho\leq C_p\text{, on \Gammashock.}\label{boundofrhoonshock}\end{equation}

With (\ref{estimateofnormaldirection}) and (\ref{estimateofv}), we can reduce RH condition on \Gammashock
\[[\rho(u-\xi,v-\eta)-\rho_1(u_1-\xi,-\eta)](n_1,n_2)=0,\]
 to
\[|\rho(u-\xi)-\rho_1(u_1-\xi)|\leq C_p\sigma.\]
Plug in (\ref{positionofshock}) (\ref{rhorelatedtonormalreflectionbyBernoulli}) and (\ref{RHfornormalshockreflection}), and use (\ref{boundofrhoonshock}), RH condition can be further reduced to
\begin{equation}
|R(u)|=\left|\left(\overline\rho_2^{\gamma-1}-uZ-\frac{u^2}{2}\right)^{1\over \gamma-1}(u+Z)-\overline\rho_2 Z\right|\leq C_p\sigma\label{estimateofR}
\end{equation} 
When $u<-Z$, $R(u)\geq \overline\rho_2 Z$. So with estimate (\ref{estimateofR}), we know on \Gammashock, $u>-Z$. And we also know $u<u_1$ (\ref{upperboundofu}), so on \Gammashock, 
\begin{equation}-Z<u<u_1.\label{roughboundofu}\end{equation}

Then for convenience, define
\[m(u)=\overline\rho_2^{\gamma-1}-u Z-\frac{u^2}{2}.\]

And denote the range of $u$ on \Gammashock\ by $\mathcal{R}$, precisely let $\MR$ be the minimal subset of $\mathbb{R}$ s.t. for every point $p\in$\Gammashock, $u(p)\in \mathcal{R}$. Since $u$ is continuous on \Gammashock, $\mathcal{R}$ is a connected interval. And (\ref{roughboundofu}) equals $\mathcal{R}\subset[-Z,u_1]$.

\begin{align*}
\frac{dR}{du}&=\frac{m^{2-\gamma\over\gamma-1}}{\gamma-1}(u+ Z)^2-m^{1\over \gamma-1}\\
                        &=\frac{m^{2-\gamma\over\gamma-1}}{\gamma-1}((u+ Z)^2-(\gamma-1)m)
\end{align*}
so \[(\gamma-1)\frac{dR}{du}m^{\gamma-2\over\gamma-1}=\frac{\gamma+1}{2}(u+Z)^2-\frac{(\gamma-1)Z^2}{2}-(\gamma-1)\overline\rho_2^{\gamma-1}\] is a monotone function of $u$ on $\mathcal{R}$, and
\[
\frac{dR}{du}(0)=\frac{\overline\rho^{\gamma-1}_{2}}{\gamma-1}(Z^2-\overline\rho_2^{\gamma-1}(\gamma-1))\underset{(\ref{subsonicityofnormallyreflectedshock})}{<}-{1\over C_p}.
\]
So on $[- Z, 0], \ R'<-{1\over C_p}$.
Since $R$ is a smooth function of $u$ and $m(u)>{1\over C_p}$ (\ref{lowerboundofm}), 
\[\left|\frac{d^2 R}{du^2}\right|\leq C_p,\] 
so there exists $ \delta_0\geq{1\over C_p},$ s.t. 
\[R'<{-1\over C_p} \text{ on $[0,{\delta_0} ]$ }\Rightarrow  R(\delta_0)<{-1 \over C_p}.\]
And by (\ref{subsonicityofnormallyreflectedshock}) and (\ref{rhorelatedtonormalreflectionbyBernoulli}):
\[(\gamma-1)m(u)-(u+ Z)^2-(\eta-v)^2\geq -C_p\sigma\]
\[\Rightarrow \frac{dR}{du}=\left[(u+ Z)^2-(\gamma-1)m\right]\frac{m^{2-\gamma\over \gamma-1}}{\gamma-1}\leq C_p\sigma.\]
This means if we choose $\sigma$ small enough, then $R<-{1\over C_p}$ on $({\delta_0}, u_1)$.

And so $u\in [-Z,\delta_0]$, and in this interval, $R_u\leq {-1\over C_p}$. Since $|R|\leq C_p\sigma$ (\ref{estimateofR}) and $R(0)=0$,
\begin{equation}
|u|<C_p\sigma \text{, in } \Omega\label{estimateofu}
.\end{equation}
Then plug (\ref{estimateofu}) into (\ref{rhorelatedtonormalreflectionbyBernoulli}) we get
\begin{equation} |\rho-\overline\rho_2|\leq C_p\sigma \text{, in } \Omega\label{rhocomparewithnormalreflection}
.\end{equation}
So, we have reached conclusion of this subsection:
\begin{proposition}\label{propprop:smallnessestimate}
If there is a regular solution $\phi$, as defined in Definition \ref{def:definitionofregularsolution}, then the physical state of gas, described by this potential function, is close to that of normal shock reflection, precisely
\[|u|+|v|+|\rho-\overline{\rho}_2|\leq C_p\sigma, \ \text{in }\Omega,\]
and $\gammashock$ stays in a $C_p\sigma$ neighborhood of normal reflected shock.
\end{proposition}

\subsection{Monotonicity of the Potential Function along Some Direction}
\label{monotonicityofpotentialfunction}

{\em In this subsection we first rotate wedge together with flow, such that upper part of \Gammawedge\ is $+\eta$-axis, as illustrated in Figure {\ref{fig:Gammawedgepluscoincideswith+eta-Axis}}. Later we need to change direction of wedge and flow. }

We denote:
\[\huone=u_1\cos(\sigma+\delta), \ \ \hvone=u_1\sin(\sigma+\delta)\]
 So the direction of upstream flow is $(\huone,\ \hvone)$.

First we want to show  $u$ cannot achieve minimum at any interior point of \Gammashock.

If $u$ achieves minimum at some interior point on \Gammashock, then at this point (denoted by $M$),
\begin{equation}{u_\xi>0}\label{uxigreaterthanzero}.
\end{equation}
This is because $u$ satisfies a linear elliptic equation (\ref{linearequationofu}), \Gammashock\ is a $C^1$ curve (Assumption (\ref{regularityassumptionofshock})), and $\left|\frac{1}{\text {slope of shock}}\right|\leq C_p\sigma$ (\ref{estimateofslopeofshock}), so we can apply Hopf maximum principle to $u$ at $M$.

At this minimum point of $u$ we should also have  
\begin{equation}u_T=0, \text{\   here }T=(v-\hvone,\huone-u),\label{uTequalszero}
\end{equation}

We put $u_\xi$, (\ref{uTequalszero}) (\ref{linearequationofphi}) together in matrix form:
\[\left(\begin{array}{ccc}
          1                &0		&0\\
	  v-\hat v_1&\hat u_1-u&0\\
           a_{11}      & 2a_{12}  &a_{22}
                     \end{array}\right)
\left(\begin{array}{c}
    u_{\xi}\\u_{\eta}\\v_{\eta}
   \end{array}\right)
                                                  =       \left(\begin{array}{c}
                                                            u_\xi\\0\\0
                                                             \end{array}\right)
\]
Solve above linear equation  of $(u_\xi,u_\eta,v_\eta)$, we get
\begin{equation}\left(u_\xi , u_\eta , v_\eta\right)
=u_\xi\left(1,\frac{v-\hat v_1}{u-\hat u_1},-\frac{a_{11}}{a_{22}}-2\frac{a_{12}}{a_{22}}\frac{v-\hat v_1}{u-\hat u_1}
\right)\label{secondorderderivativeatminimumpointofuonshock}
\end{equation}

Then we can compute the sign of 
\[D_T\hat S=D_T(\frac{v-\hat v_1}{\hat u_1-u}), \ (\text{ here, we define } \hat S=\frac{v-\hat v_1}{\hat u_1-u})\]
Plug (\ref{secondorderderivativeatminimumpointofuonshock}) and
 \[\text{ }\hat S_u=\frac{v-\hvone}{(\huone-u)^2},\ \hat S_v=\frac{1}{\huone-u} ,\]
into:
\[D_T\hat S=\hat S_u u_\xi(v-\hvone)+\hat S_v v_\xi(v-\hvone)+\hat S_u u_\eta(\huone-u)+\hat S_v v_\eta(\huone-u)\]
we get
\[ D_T\hat S=u_\xi\left(-\hat S^2+2\frac{a_{12}}{a_{22}}\hat S-\frac{a_{11}}{a_{22}}\right)<0\]



This means at this minimum point of $u$, shock bends against upstream flow direction, as shown in Figure \ref{fig:sliding}.
So the tangent line $l_T$ of \Gammashock\ at M,  must separate domain $\Omega$ into three parts, and when we slide this line leftwards it must contact domain $\Omega$ at some interior point of \Gammashock\ (denote this point by $N$).

\begin{figure}

 \centering%
 \includegraphics[height=5.5cm]{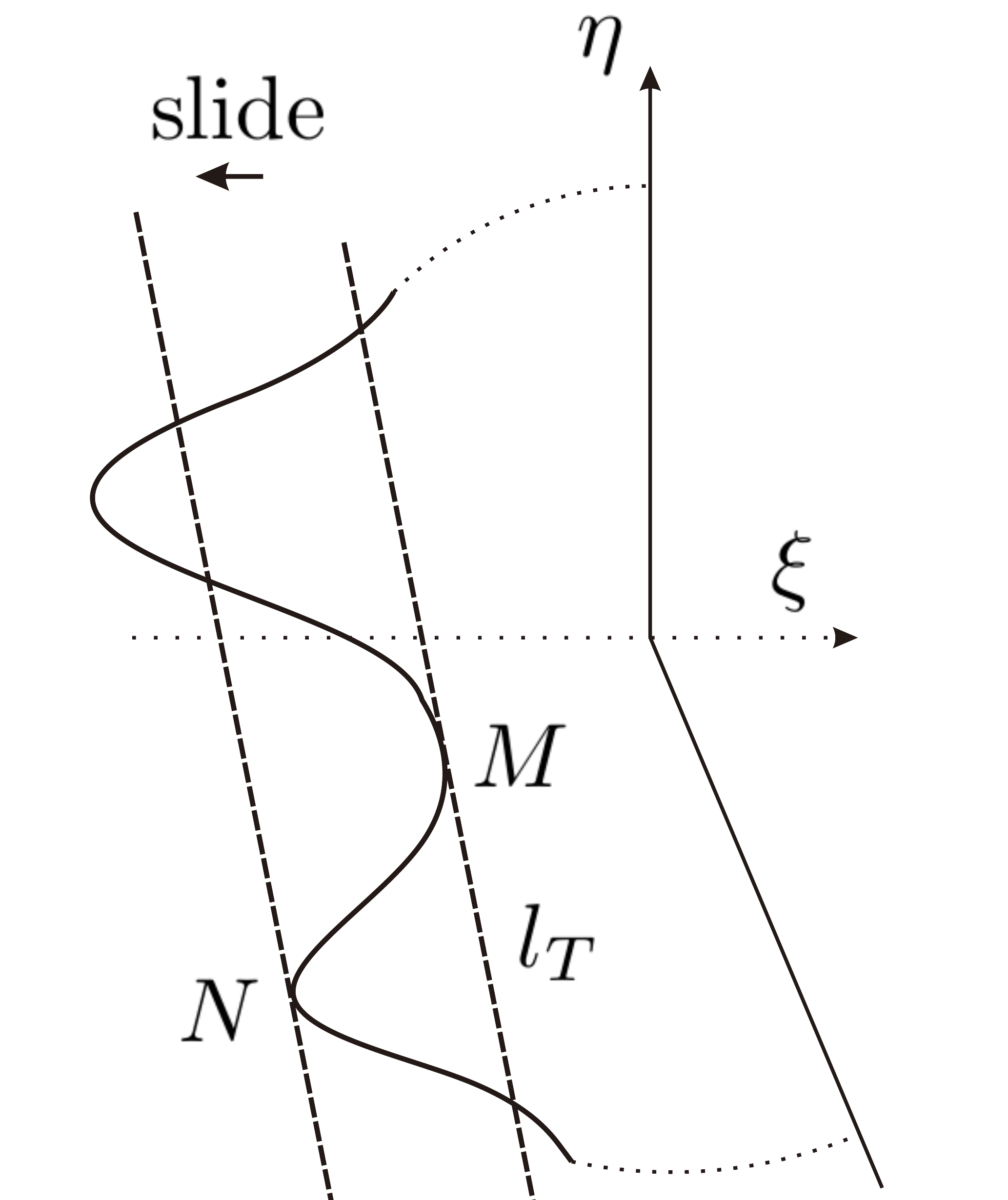}
   \caption{Sliding Tangent Line}
  	 \label{fig:sliding}
\end{figure}
Then we compare these two points $M$ and $N$. For convenience, denote $\partial_\xi \phi=u,\partial_\eta \phi=v$ and $\rho,\  \xi,\ \eta$ at $M, N$ will be denoted as $u_M,v_M,\rho_M,\xi_M,\eta_M$ and $u_N,v_N,\rho_N,\xi_N,\eta_N$ respectively.

We define \[k\triangleq\frac{\hvone-v_M}{\huone-u_M}=\frac{\hvone-v_N}{\huone-u_N}.\]
And direct computation shows
\begin{align*}
  &u_M+kv_M\\
=&u_M+k(-k\huone+ku_M+\hvone)\\
\leq&u_N+k(-k\huone+ku_N+\hvone)\\
=&u_N+kv_N.
\end{align*}

{\em Now we rotate wedge and flow s.t. $l_T$ parallels to $\eta$-axis. }

For simplicity, after this rotation we still use notation in previous part of this subsection. So we still have \begin{equation}\partial_\xi\phi(M) =u_M\leq\partial_\xi\phi(N)=u_N\label{uMleqthanuN}.
\end{equation}

Since, now, the tangent lines at $M, N$ are vertical, RH condition becomes
\begin{equation}
\rho_1(\hat u_1-\xi_M)=\rho_M(u_M-\xi_M);\ \ \rho_1(\hat u_1-\xi_N)=\rho_N(u_N-\xi_N)\label{RHatMN}.
\end{equation}
Plugging state parameters of State(\RN{1}) into (\ref{sonicspeed}), gives at $M$ and $N$
\begin{align}
\rho_0^{\gamma-1}=\rho_1^{\gamma-1}+\phi_1-\hat u_1\xi-\hat v_1\eta+{\hat u_1^2+\hat v_1^2\over 2}
				   =\rho^{\gamma-1}+\phi- u\xi- v\eta+{ u^2+v^2\over 2}\label{BernoullilawatMN}.
\end{align}
Since at $M, N$, $\hvone=v$ and $\phi_1=\phi$, (\ref{BernoullilawatMN}) can be reduced to 
\begin{align}
\rho^{\gamma-1}=\rho_1^{\gamma-1}+\frac{\huone^2}{2}-(\huone-u)\xi-\frac{u^2}{2}\label{reducedBernoullilawatMN}.
\end{align}

Then plug (\ref{reducedBernoullilawatMN}) into (\ref{RHatMN}), we get at $M, N$
\[-(\rho_1^{\gamma-1}+\frac{\huone^2}{2}-\huone\xi+u\xi-{u^2 \over 2})^{1\over \gamma-1}(u-\xi)+\rho_1(\huone-\xi)=0.\]

Define:
\[R(u,\xi)=-(\rho_1^{\gamma-1}+\frac{\huone^2}{2}-\huone\xi+u\xi-{u^2 \over 2})^{1\over \gamma-1}(u-\xi)+\rho_1(\huone-\xi),\]
then
 $R(u_M,\xi_M)=R(u_N,\xi_N)=0$.
\[R_u=-\rho\left[1-\frac{(\xi-u)^2}{(\gamma-1)\rho^{\gamma-1}}\right]=-\rho\left[1-\frac{(\xi-u)^2}{c^2}\right]<0,\]
Inequality above follows from  (\ref{subsonicityofnormallyreflectedshock}). Note that the rotation we did in later part of this section didn't move shock much, since we have estimated $|u|,|v|\leq C_p\sigma$ (\ref{estimateofu}) (\ref{estimateofv}), so \Gammashock\ still stays in a $C_p\sigma$ neighborhood of \Gammashock, after the rotation.  

So when $u$ closes to $0$ and $\xi$ closes to $- Z$ we can consider $u$ as a function of $\xi$, s.t. $R(u(\xi),\xi)=0$. And since
\[ R_\xi=-\frac{\rho^{2-\gamma}}{\gamma-1}(-\huone+u)(u-\xi)+\rho-\rho_1>0,\]
we get $\partial_\xi u>0$,  which implies $u_M>u_N$, and contradicts with $u_M\leq u_N$(\ref{uMleqthanuN}).

This means the minimum of $u$ cannot be achieved at any interior point of \Gammashock. 

{\em Now we rotate back to the position used at the beginning of this subsection}.

  Argument in section \ref{gradientestimate} shows $u$ cannot achieve its minimum in $\Omega$ and at interior point of $\gammawedgeminus$. So
\begin{equation}
u\geq\underset{\gammasonicplus\cup\gammasonicminus\cup \gammawedgeplus}{\min }u=\min\{u_2^+=0, u_2^-, 0\}=0
\label{ugeqthancero}
\end{equation}

Similar argument works when we rotate wedge, s.t. \Gammawedgeminus\ coincides with $\eta$ axis, so we get the following conclusion of this subsection:
\begin{proposition}
Let $T_\pm$ denote the tangent vector of $\gammawedgepm$ respectively, with $T_+$ pointing upwards and $T_-$ pointing downwards. If a vector $\vgamma$ satisfies
\[\vgamma\cdot T_\pm\geq 0,\]
then,
\[\phi_{\vgamma}\geq0,\ \text{in} \  \overline\Omega,\]
for any regular solution $\phi$.
\end{proposition}

\subsection{Estimates of the RH function}\label{estimateofRHfunction}

 {\em In this subsection, at first, the direction is chosen s.t. the upstream flow is $(u_1,0)$, as illustrated in Figure {\ref{fig:NonVerticalShockhitsSymmetricWedge}}, later for convenience we will need to change the direction of  wedge and flow. }\\
Consider 
\begin{equation}G=RH-(\rho-\rho_1)(\phi_1-\phi)\label{G}
\end{equation}
In subsection \ref{gradientestimate}, RH is defined (\ref{RH}) (\ref{rho}). And we stated there, that RH can be considered either as a function of $\xi,\eta$ or a function of five variables $u, v, \phi, \xi, \eta$, and when considered as a 5-variable- function, $\xi(\eta)$-derivative of RH is denoted as $RH_{(\xi)}(RH_{(\eta)})$ respectively. Here we adopt the idea and notation.

Pure algebraic computation gives:
\begin{equation}
G_{\phi}=RH_\phi+(\rho-\rho_1)+(\phi_1-\phi)\frac{\rho^{2-\gamma}}{\gamma-1}\label{algebraicphiderivativeofG}
\end{equation}
\begin{equation}
G_{(\xi)}=RH_{(\xi)}-(\rho-\rho_1)u_1-(\phi_1-\phi)\frac{\rho^{2-\gamma}}{\gamma-1}u\label{algebraicxiderivativeofG}
\end{equation}

Plug (\ref{algebraicphiderivativeofG}) and (\ref{algebraicxiderivativeofG}) into 
\begin{equation}
G_{\xi}=G_uu_{\xi}+G_vv_\xi+G_\phi u+G_{(\xi)} \label{xiderivativeofG},
\end{equation}
we find, on right-hand side of (\ref{xiderivativeofG}), $G_\phi u+G_{(\xi)}$ disappeared. So (\ref{xiderivativeofG}) can be reduced to 
\begin{equation}
G_\xi=G_uu_\xi+G_vv_{\xi} \label{reducedxiderivativeofG}.
\end{equation}
And replace $\xi$ by $\eta$, we get
\begin{equation}
G_\eta=G_uu_\eta+G_vv_{\eta} \label{reducedetaderivativeofG}.
\end{equation}

Given $\sigma$ small enough
\begin{align}
G_u&=RH_u-(\phi_1-\phi)\rho_u\nonumber\\
        &<\rho u_1\left(\frac{(\xi-u)^2}{c^2}-1\right)+(\rho_1-\rho)(\xi+ Z)-u_1(\xi+ Z)\frac{\rho}{c^2}(\xi-u)+C_p\sigma\nonumber\\
        &=\rho u_1\left(\frac{(- Z-u)(\xi -u)}{c^2}-1\right)+(\rho_1-\rho)(\xi+ Z)+C_p\sigma\nonumber\\
	&\leq\overline\rho_2u_1\left(\frac{Z^2}{\overline c_2^2}-1\right)+C_p\sigma\nonumber\\
        &<-C_p\label{negativeupperboundofGu}.
\end{align}
Now, to G, we can copy what we did to $S$, in subsection \ref{gradientestimate} (from (\ref{derivativeofS}) to (\ref{equationofS})), simply replace $S$ by $G$, and get there exist $\hat b_i$'s $(\hat b_i\in C^0(\overline\Omega\setminus(\overline{\gammasonic}\cup\overline{\MC}))$, s.t. 
\begin{equation}
a_{ij}G_{ij}+\hat b_i G_i=0,\label{equationofG}
\end{equation}
\statementstarts{so $G$ cannot achieve its minimum in $\Omega$.}\label{Gnointeriorminimum}
\statementends

We want to show $G$ can not achieve its minimum on \Gammawedge. To do so, we compute $G_{\vn}$ on \Gammawedgeplus
, here $\vn$ is the outer normal direction of $\Omega$ on \Gammawedgeplus. 

If we do the computation directly, it will be time consuming, and the idea is not clear. {\em So, we rotate wedge and flow s.t.  \Gammawedgeplus\ is $+\eta$ axis.} Now $G_{\vn}$, in original coordinate, becomes $G_\xi$. And in new coordinate, 
\begin{align}
G=[\rho(u-\xi,v-\eta)-\rho_1(u_1\cos(\sigma+\delta)-\xi, u_1\sin(\sigma+\delta)-\eta)]\nonumber\\
\cdot(u-u_1\cos(\sigma+\delta), v-u_1\sin(\sigma+\delta))-(\rho-\rho_1)(\phi_1-\phi),\label{Ginnewcoordinate}
\end{align}
\begin{align}\text{on \Gammawedgeplus},\ \ \  u=v_\xi=0\ (\text{Neumann Condition}), \ \ \xi=0 \ (\text{\Gammawedgeplus\  is $+\eta$ axis})\label{conditionsongammawedgeplus}.
\end{align}
Plug (\ref{conditionsongammawedgeplus}) into  $\xi$ derivative of (\ref{Ginnewcoordinate}), we get on \Gammawedgeplus
\[G_\xi=-(\rho+\rho_1)u_1\cos(\sigma+\delta)u_\xi.\]
In equality above, we know $u_\xi<0$, since in last subsection {\ref{monotonicityofpotentialfunction}}, we found $u>0$ in $\Omega$, and $u=0$ on \Gammawedgeplus. 
So $G$ cannot achieve its minimum on \Gammawedgeplus, with similar method we can show $G$ cannot achieve minimum on \Gammawedgeminus. And by definition $G=0$ on \Gammashock,  by computation $G=0$ on $\gammasonic$. 
Combine these with (\ref{Gnointeriorminimum}), we get
\[ G>0 \text{, in }  \Omega.\]
So
\begin{equation} RH>(\phi_1-\phi)(\rho-\rho_1) \text{, in }  \Omega.\label{lowerboundofRH}
\end{equation}

We consider (\ref{lowerboundofRH}) as a ``coercive" estimate of RH boundary condition, and it will play an essential role in Section {\ref{symmetricestimate}} and \ref{antisymmetricestimate}.

\subsection{Convexity of the Shock}\label{convexityofshock}
{\em In this section, direction of upstream flow is $(u_1,0)$.}\\
Now we know
\[G>0 \text{, in } \Omega \text{ ,  }\]\[ G=0 \text{, on } \gammashock,\]
 $G$ satisfies a second order elliptic PDE (\ref{equationofG}) and \Gammashock\ is the graph of a $C^1$ function of $\eta$ (\ref{gammashockisaconefunctionofeta}), so by Hopf Lemma $G_\xi>0$ on \Gammashock .
So along \Gammashock, we have the following:
\[G_\xi>0, \ G_T=0,\  \linearequationofphi.\]
Again we can write them in matrix form:
\[
\left(\begin{array}{ccc}
         G_u   &   G_v                          &0                   \\
         G_u v&G_v v+G_u (u_1-u)&G_v (u_1-u)\\
         a_{11}&2a_{12}                       &a_{22}
        \end{array}
        \right)
                           \left(\begin{array}{c}
                           u_\xi              \\
                            u_\eta            \\
    	   	           v_\eta
 		           \end{array}
                         \right)
                                          =  \left(\begin{array}{c}
                    			      G_\xi          \\
                			      0          \\
    	   				       0
 		 			          \end{array}
                  				       \right)
\]
Note, that the first line is (\ref{reducedxiderivativeofG}). 
The determinant of this $3\times3$ matrix is
 \[D\triangleq(u_1-u)(a_{11}G_v^2-2a_{12}G_uG_v+a_{22}G_u^2).\] Solve above linear equation  of $u_\xi,\ u_\eta,\ v_\eta$, we get

\[
                           \left(\begin{array}{c}
                           u_\xi              \\
                            u_\eta            \\
    	   	           v_\eta
 		           \end{array}
                         \right)
                                          = \frac{G_\xi}{D}\left(\begin{array}{cc}
       						G_u(u_1-u)a_{22}-2a_{12}G_v(u_1-u)&-G_va_{22}\\
                                              G_v(u_1-u)a_{11}                                     &  a_{22}G_u\\
          					-G_u(u_1-u)a_{11}                                    &G_va_{11}-2G_u a_{12}
   					     \end{array}
  					      \right)
                             						              \left(\begin{array}{c}
                    									       1         \\
                									     -v          
 		 			     						     \end{array}
                  				   						    \right)
\]
With this result we can compute the sign of $D_T S$, here   $ S=\frac{v}{u_1-u},T=(v,u_1-u)$.

\begin{align*}
D_T S&=S_u u_\xi v+S_v v_\xi v+S_u u_\eta (u_1-u)+S_vv_\eta(u_1-u)\\
            &=S^2 u_\xi +2S u_\eta +v_\eta\\
            &=\frac{G_\xi}{D}[G_v v+G_u(u-u_1)](a_{11}-2a_{12}S+a_{22}S^2)
\end{align*}
Earlier results show $|v|\leq C_p\sigma$ (\ref{estimateofv}), $|u|\leq C_p\sigma$ (\ref{estimateofu}) and $G_u<-C_p$ (\ref{negativeupperboundofGu}).
So when $\sigma$ is small enough $D_T S>0$,  shock is convex.
\subsection{Comparison of $\phi$ and $\phi_2^{\pm}$}
 In this subsection, we prove $\phi\geq \phi_2^{\pm}$ in $\Omega$. To do so, we compare
\begin{itemize}
\item value of $\phi$ and $\phi_2^{\pm}$, on $\gammashock\cup\gammasonicplus$
\item derivative of $\phi $ and $\phi_2^{\pm}$, in $\Omega$.
\end{itemize}
And we only prove $\phi\geq\phi_2^{+}$, since $\phi\geq\phi_2^-$ follows symmetrically.

{\em In this subsection, direction of wedge and flow is as illustrated in Figure \ref{fig:picofzigzagwalk}, s.t. \Gammawedgeplus\ coincides with $+\eta-$axis}. We denote velocity of \Stateone\ as $(\huone, \hvone)=(u_1\cos(\sigma+\delta), u_1\sin(\sigma+\delta))$, and velocity of \Statetwoplus\ as $(u_2^+,v_2^+)$. 

First we compare $\phi$ and $\phi_2^+$ on $\gammashock\cup \gammasonicplus$.

\begin{figure}
       \centering
    \includegraphics[height=6cm]{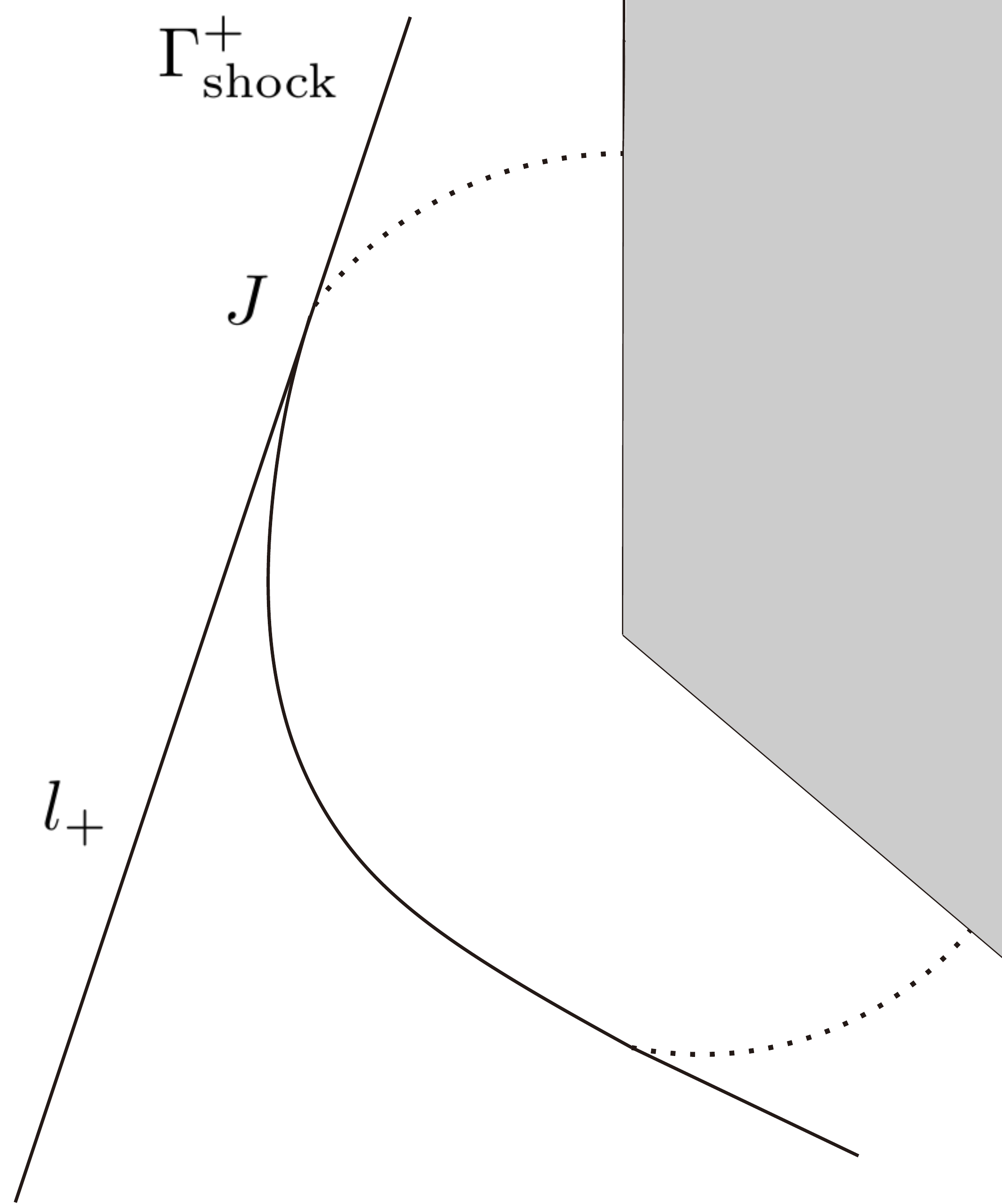}
\caption{Comparison of $\phi$ and $\phi_2^+$}\label{fig:picofzigzagwalk}
\end{figure}

Assumption (\ref{continuityofvelocityonsoniccircle}) implies, \Gammashock\ is tangent to \Gammashockplus\ at $J$ (here, $J\triangleq\overline{\gammashock}\cap\overline{\gammasonicplus}$). Here, recall that, \Gammashockplus\ is the regular reflected shock above $\xi$-axis. And in section \ref{convexityofshock} we proved that shock is convex.  So if we denote that  \Gammashockplus\ lay on a straight line $l_+$, then $l_+$ should be tangent to \Gammashock\ at $J$, and \Gammashock\ should stay right to $l_+$, as illustrated in Figure \ref{fig:picofzigzagwalk}. 

Since on \Gammashockplus, $\phi^+_2=\phi_1$, and $\phi_1, \phi_2^+$ are both linear functions, we have $\phi_1=\phi_2^+$
 on $l_+$. And since \[\phi_{1,\xi}=u_1\cos(\sigma+\delta)>0=u_2^+=\phi_{2,\xi}^+\]
$\phi_1\geq\phi_2^+$ on \Gammashock. Free boundary condition requires $\phi=\phi_1$ on \Gammashock\ (\ref{nonvorticityassumptiononshock}), so $\phi\geq\phi_2^+$ on \Gammashock.

By assumption $\phi=\phi_2^+$ on \Gammasonicplus\ (\ref{continuityofpotentialfunctiononsoniccircle}). So 
\begin{equation}
\phi\geq\phi_2^+ \text{\ \ on $\gammasonicplus\cup\gammashock$}
\end{equation}

Then, we compare velocity.

We know $u_2^+=0$, and analysis in section \ref{monotonicityofpotentialfunction} shows $u\geq0$ in $\Omega$ (\ref{ugeqthancero}), so $u\geq u_2^+$.

To $v$, we can apply analysis in section \ref{gradientestimate}(from (\ref{linearequationofphi}) to ({\ref{uwedgeplus}})),  only replacing $\xi$-derivative of (\ref{linearequationofphi}) by $\eta$-derivative of (\ref{linearequationofphi}), to assert that $v$ cannot achieve minimum or maximum in $\Omega$ or at interior point of \Gammawedgeplus\ and \Gammawedgeminus. 

Then, if $v$ achieves minimum or maximum at some interior point of $\gammashock$, at this point, 
\[v_T=0, (\text{here, }T=(\hvone-v, u-\huone))\]
Put this, (\ref{reducedtangentialderivativeofRHonshock}) and (\ref{linearequationofphi}) together:
\[
\left(\begin{array}{ccc}
0				&\hvone-v						&u-\huone\\
RH_u(\hvone-v)&RH_v(\hvone-v)+RH_u(u-\huone)&RH_v(u-\huone)\\
a_{11}			&2a_{12}						&a_{22}\end{array}\right)
\left(\begin{array}{c}  u_\xi\\v_{\xi}\\v_{\eta}   \end{array}\right)
=0\]
 Determinant of this 3$\times$3 matrix equals
\[-RH_u(a_{11}(u-\huone)^2-2a_{12}(\hvone-v)(u-\huone)+a_{22}(\hvone-v)^2)\]
In this subsection
\begin{equation}
RH=[\rho(u-\xi,v-\eta)-\rho_1(\huone-\xi,\hvone-\eta)](u-\huone,v-\hvone), \text{ on \Gammashock},
\end{equation}
 while $\rho$ is still expressed as (\ref{rho}).
Computation shows
\begin{align}
RH_u&=\frac{\rho^{2-\gamma}}{\gamma-1}(\xi-u)[(u-\xi)(u-\huone)+(v-\eta)(v-\hvone)]+\rho(2u-\huone-\xi)-\rho_1(\huone-\xi)\\
 	  &\leq C_p\sigma+u_1\overline{\rho}_2\underbrace{\left(\frac{Z^2}{\overline c_2^2}-1\right)}_{<-{1\over C_p}\text{(\ref{subsonicityofnormallyreflectedshock})}}+\underbrace{\overline\rho_2 Z-\rho_1(u_1+Z)}_{=0, \text{RH condition for normal shock reflection}}.
\end{align}

So when $\sigma$ small enough, if  the minimum or maximum of $v$ is  achieved at some point on $\gammashock$, then at this point, $D^2\phi=0$, $v_n=0$. This contradicts with Hopf Lemma, since $v$ would satisfy a linear elliptic equation without zero order term, similar to (\ref{linearequationofu}). 

So $v$ can only achieve its minimum on \Gammasonicplus\ and \Gammasonicminus. In section \ref{incidentshocknormalandregularreflection}, we have shown,  in current position, $v_2^-$ should be $\leq0$, while $v_2^+$ should be $\geq0$. So $v\leq v_2^+$ in $\overline \Omega$.

Now, we have proved, 
\begin{itemize}
\item $\phi_\xi=u\geq0=u_2^+=\phi_{2,\xi}^+,\ \phi_{\eta}=v\leq v_2^+=\phi_{2,\eta}^+ $   in $\overline\Omega$
\item $\phi\geq\phi_2^+$ on $\gammashock \cup \gammasonicplus$.
\end{itemize}

With above estimates, we can tell $\phi\geq\phi_2^+$ along $\overline\gammawedge$.  And computation gives
\[\phi^+_2<\phi^-_2, \  \text{on $\gammasonicminus$}
,\]
so $\phi_2^+<\phi$, on $\gammasonicminus$. So $\phi\geq \phi_2^+$, on $\partial \Omega$. 


Since $\phi_2^+-\phi$ satisfies a second order elliptic equation in $\Omega$, we can conclude 
\begin{equation}
\phi\geq\phi_2^+,\ \text{in } \Omega\label{comparisonofphiandphitwoplus}.
\end{equation}

\subsection{Elliptic Estimates away from the Sonic Circle}
In $\{|\eta|\leq {1\over 2}\sqrt{\overline c_2^2-Z^2}\}$ when $\sigma$ small enough,
\[\frac{|\nabla\varphi|^2}{c^2}\leq\frac{{1\over 2}(\overline c_2^2-Z^2)+Z^2+C_p\sigma}{\overline c_2^2-C_p\sigma}
                                                           =\frac{{1\over 2}(\overline c_2^2+Z^2)+C_p\sigma}{\overline c_2^2-C_p\sigma}<1-\frac{1}{C_p}.\]

\subsection{Derivative Estimate at Corner of Wedge}\label{holdergradientestimateatcornerofwedge}
Now with Proposition \ref{propprop:smallnessestimate}, we can control $|\nabla\phi|$ by $C_p$, so according to Lemma A.1 of \cite{CFHX}, there exists $0<\alpha<1$, which does not depend on angle of wedge s.t. we can take 
\[\omega(r)=C_pr^\alpha,\]
in Lemma 4.3 in \cite{CFHX}, and get:
\begin{proposition}\label{prop:gradientestimateatcornerofwedge}
Let $\phi$ be a regular solution in the sense of Definition \ref{def:definitionofregularsolution}, we can find $C_p$ and $\alpha$ $(0<\alpha<1)$, which do not depend on $\sigma$, s.t.
\begin{align}
|\nabla\phi|_{C^0}\leq C_pr^\alpha,\ \text{in }B_r(\MC)\cap\Omega;\\
|D^2\phi|_{C^0}\leq C_pr^{\alpha-1},\ \text{in }B_r(\MC)\cap\Omega.
\end{align}
for $r<\min\{{Z\over 2}, {Y\over 2}\}$.
\end{proposition}  
\subsection{H\"older Gradient Estimates away from the Sonic Circle and the Corner of the Wedge }
\label{holdergradientestimateawayfromsoniccircleandcornerofwedge}
In this section we estimate H\"older norm of $\nabla\phi$ through quasiconformal mapping. Our method is a modification of the method in \cite{Ni} and Chap 12 of \cite{GT}.

For $ q\in\overline{\Omega}\cap\left\{|\eta|<\frac{Y}{4}\right\}\cap\left\{\xi+Z<\frac{Z}{4}\right\}$, 
consider
\[\Drq\triangleq\int_{B_r(q)\cap\Omega}u_\xi^2+v_\xi^2+u_\eta^2+v_\eta^2,\ \left(0<r<r_0\triangleq\min\left\{\frac{Y}{4},\frac{Z}{4}\right\}\right).
\]

\begin{align*}
\Drq\leq &2K\int_{B_r(q)\cap \Omega}(v_\xi^2-v_\eta u_\xi)\ (K\triangleq\frac{\Lambda}{\lambda}\leq C_p)\\
            = &2K\int_{B_r(q)\cap \Omega}\frac{v_\xi G_\eta-v_\eta G_\xi}{G_u}\ (\text{in denominator we consider $G$}\text{
                                             as a function of $u, v, \xi,\eta$,}\\ &\text{ while in numerator we consider $G$ as a function of $\xi,\eta$} )\\
          \leq&2C_p\int_{B_r(q)\cap \Omega}\nabla\cdot(v_\eta G,-v_\xi G)\ (\text{because $v_\xi^2-v_\eta u_\xi>0$, and  $G_u\leq-\frac{1}{C_p}$} (\ref{negativeupperboundofGu}))\\
          =&2 C_p\int_{\partial(B_r(q)\cap \Omega)}Gv_{\vt} \ \  \left((v_\eta,-v_\xi)\cdot \vn\triangleq v_{\vt}\right)\\
         = &2C_p\int_{(\partial B_r(q))\cap \Omega}Gv_{\vt}\ \ \left(\text{because }G\mid_{\gammashock}=0\right)\\
      \leq&C_p\left(2\pi r\int_{(\partial B_r(q))\cap\Omega}|\nabla v|^2 \right)^{1\over 2}.
\end{align*}
Above means \[\Drq^2\leq C_pr\frac{d\Drq}{d r}\ \Rightarrow \ \frac{d}{dr}\left[\frac{1}{\Drq}+\frac{\log r}{C_p}\right]\leq 0, \]
and after integral we get \[\Drq\leq \frac{C_p}{-\log{r\over r_0}}.\]

Then we estimate growth of $\Drq$ more precisely with\[\Drq\leq 2C_p\int_{(\partial B_r(q))\cap \Omega}Gv_{\vt}.\]
1) if $(\partial B_r)\cap\gammashock\neq\emptyset$
\begin{align*}
\Drq\leq &C_p\left( \frac{1}{\mu}\int_{(\partial B_r)\cap\Omega}|\nabla v|^2 +\mu\int_{(\partial B_r)\cap\Omega} G^2\right)\\
         \leq&C_p\left( \frac{1}{\mu}\int_{(\partial B_r)\cap\Omega}|\nabla v|^2 +\mu\int_{(\partial B_r)\cap\Omega} |\nabla G|^2 4\pi^2r^2\right)
             \ \ (G\mid_{\gammashock}=0)\\
         =     &C_pr\int_{(\partial B_r)\cap\Omega}|\nabla u|^2+|\nabla v|^2 \ \ (\text{take }\mu={1\over 2\pi r})\\
        =     &C_p r\frac{d\Drq}{dr};
\end{align*}
2) if $(\partial B_r)\cap\gammashock=\emptyset$
\[ \Drq\leq2 C_p\int_{(\partial B_r)\cap\Omega}(G-\overline G) v_{\vt},\]
again with method above we get\[ \Drq\leq C_p r\frac{d\Drq}{dr}.\]
Above result implies\[\Drq\leq\mathcal{D}(r_0/2,q)\left(\frac{r}{r_0/2}\right)^{1\over C_p}\leq C_p r^{1\over C_p}, \text{ for }r<r_0/2.
\]
Then with Lemma 7.16, Lemma 7.18 of \cite{GT} and the property that \Gammashock\ is convex we get:

\begin{proposition}
\label{prop:holdergradientestimateawayfromsoniccircleandcornerofwedge}For $\phi$, a regular solution in the sense of Definition \ref{def:definitionofregularsolution}, there exists $0<\chi<1$ and $C_p$, both do not depend on angle of wedge, s.t.
 \begin{equation}|\nabla\phi|_{\chi;\mathcal{F}}\leq C_p,
\label{equationofholdergradientestimateawayfromcornerofwedgeandsoniccircle}
\end{equation}
where we define 
\[\mathcal{F}\triangleq\left(\overline{\Omega}\cap\left\{|\eta|<\frac{Y}{4}\right\}\cap\left\{\xi+Z<\frac{Z}{4}\right\}\right).\]
\end{proposition}
\subsection{$C^{2,\chi}$ Estimate away from Sonic Circle and Corner of Wedge}
\label{C2chiestimateawayfromsoniccircleandcornerofwedge}
Now we consider divergence form equation of $G$.

Since $a_{ij}\phi_{ij}=0$, given any $h\in C_0^\infty(\Omega)$, we have
\begin{align*}
0=&\int_\Omega\left({a_{11}\over a_{22}}\phi_{\xi\xi}+{2a_{12}\over a_{22}}\phi_{\xi\eta}+\phi_{\eta\eta}\right)h_\xi\\
  =&\int_{\Omega}{a_{11}\over a_{22}}u_{\xi}h_{\xi}+{2a_{12}\over a_{22}}v_\xi h_\xi+u_\eta h_{\eta}.
\end{align*}
Then make use of first two lines of
\begin{equation}
\left(\begin{array}{c}u_\xi\\u_\eta\\v_\eta\end{array}\right)
={1\over D}\left(\begin{array}{cc}a_{22}G_u-2a_{12}G_v&-a_{22}G_v\\
								a_{11}G_v&a_{22}G_u\\
								-a_{11}G_u&a_{11}G_v-2a_{12}G_u\end{array}\right)
\left(\begin{array}{c}G_\xi\\G_\eta\end{array}\right),
\label{eq:relationofderivativeofuandderivativeofG}
\end{equation}
where $D=a_{11}G_v^2-2a_{12}G_uG_v+a_{22}G_u^2$, we get
\[\int_{\Omega}{a_{11}G_u\over D}G_\xi h_\xi+{a_{11}G_v\over D}G_\xi h_\eta+{-a_{11}G_v+2a_{12}G_u\over D}G_\eta h_{\xi}+{a_{22}G_u\over D}G_\eta h_\eta=0.\]
With computation we can tell above  equation is elliptic in $\Omega$, when $\sigma$ small($G_u\neq 0$).
And we know $G\mid_{\gammashock}=0$ and $|\gammashock|_{1,\chi}\leq C_p$
 (\ref{equationofholdergradientestimateawayfromcornerofwedgeandsoniccircle}), with Theorem 8.33 of \cite{GT}, we get
 \[|G|_{1,\chi;\mathcal{F}}\leq C_p,\]
where $\MF$ is defined in last subsection.

Since with (\ref{eq:relationofderivativeofuandderivativeofG}) we can represent $D^2\phi$ by derivatives of $G$, derivatives of $\phi$ and $\phi$, we can conclude:
\begin{proposition}\label{prop:C2chiestimateawayfromsoniccircleandcornerofwedge}
For $\phi$, a regular solution in the sense of Definition \ref{def:definitionofregularsolution}, and for $0<\chi<1$, whose existence is argued in last subsection, there exists $C_p$, s.t.
\begin{equation}
|\phi|_{2,\chi;\MF}\leq C_p.
\end{equation}
\end{proposition}



\section{Perturbation of \Statetwo}\label{asymptotic}

In this section we show that some parameters of \Statetwopm\ are analytic functions of $\sigma$, and compute the derivative of these parameters w.r.t. $\sigma$ and $\delta$.

\subsection{Symmetric: Derivatives w.r.t. $\sigma$}
{\em In this subsection the position of flow and wedge is as in Figure \ref{fig:SymmetricCase}.} And we compute derivative of parameters of \Statetwo\ w.r.t.  $\sigma$.  If we change $\sigma$ to $\sigma+\delta$, results here apply to \Statetwoplus; and if we change $\sigma$ to $\sigma-\delta$, then results here apply to $\statetwominus$, after a reflection.
\subsubsection{Parameters of \Statetwo}

\begin{figure}
       \centering
    \includegraphics[height=7cm]{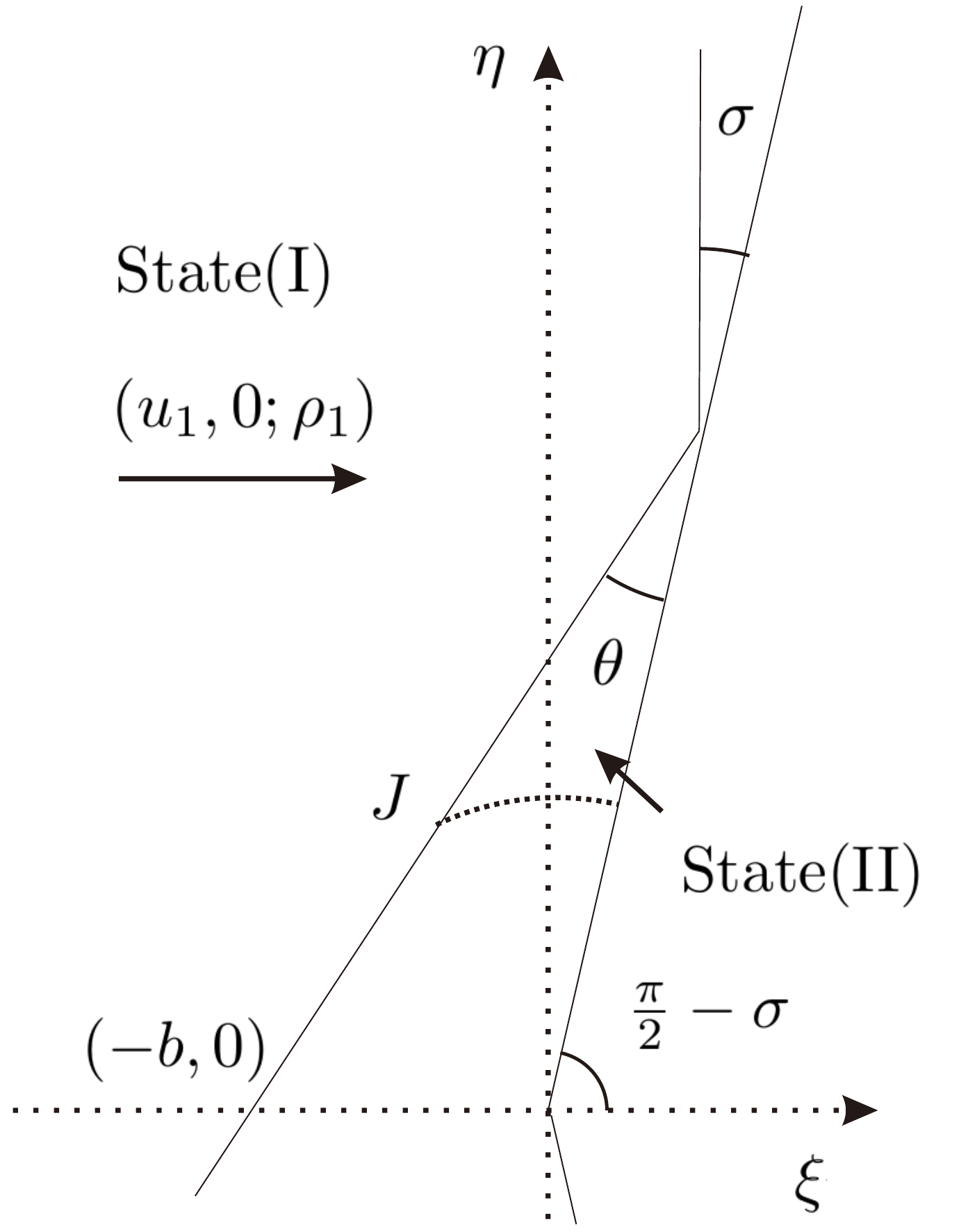}
\caption{Parameters of \Statetwo}
\label{fig:ParametersofStatetwo}
\end{figure}

We denote velocity and density of gas in  \Statetwo\ as $(u_2,v_2)$ and $\rho_2$; the point where  regular reflected shock intersects with sonic circle as $J$; the angle between regular reflected shock and $\gammawedge$ as $\theta$; the point where extension line of regular reflected shock intersects with symmetry axis of wedge as $(-b,0)$, as illustrated in Figure \ref{fig:ParametersofStatetwo}.

We have the following relations for $(v_2,\rho_2,b,\theta)$ and $\sigma$. 

\begin{gather}
\rho_2^{\gamma-1}+\frac{\tan^2\sigma+1}{2}v_2^2+v_2b\tan\sigma=\rho_1^{\gamma-1}+\frac{u_1^2}{2}+u_1 b
\label{symmetricBernoullilawatb}\\
{\centering v_2 \frac{\cos\theta}{\cos\sigma}=u_1\sin(\sigma+\theta)}
\label{nonvorticityalongregularreflectedshock}\\
(\rho_1(u_1+b)-\rho_2(v_2\tan\sigma+b))(u_1-v_2\tan\sigma)\nonumber\\
+\rho_2v_2^2=0
\label{massconservationatb}\\
X\tan(\theta+\sigma)=\tan\sigma(b+X)
\label{geometry}
\end{gather}

(\ref{symmetricBernoullilawatb}) is Bernoulli Law computed at $(-b,0)$, (\ref{nonvorticityalongregularreflectedshock}) means gas in \Stateone\ and \Statetwo\ have same tangential velocity along regularly reflected shock. (\ref{massconservationatb}) is Mass Conservation computed at $(-b,0)$, with $(u_1-u_2,-v_2)$ taken as normal direction of shock. (\ref{geometry}) is a fundamental geometric condition.

Computation shows for $\sigma$ small enough,  system (\ref{symmetricBernoullilawatb}) (\ref{nonvorticityalongregularreflectedshock}) (\ref{massconservationatb}) (\ref{geometry}) has a unique solution, which satisfies 
\[|b-Z|\leq C_p\sigma, |u_2|+|v_2|\leq C_p\sigma, |\rho_2-\overline\rho_2|\leq C_p\sigma, |\theta|\leq C_p\sigma,\] so
$ b,\  \theta,\  v_2,\ \rho_2 \text{ are analytic functions of } \sigma$. And at $\sigma=0$ \[\frac{d \theta}{d \sigma}=\frac{ Z}{X}, \ \frac{db}{d\sigma}=\frac{d\rho}{d\sigma}=0,\ \frac{d v_2}{d \sigma}=u_1(\frac{Z}{X}+1).\]

\subsubsection{Coordinates of the Intersection Point}
Then we compute the coordinate of $J$ where shock intersects with sonic circle, we have:
\begin{align}
(\xi_{J}-u_2)^2+(\eta_{J}-v_2)^2=c_2^2=(\gamma-1)\rho_2^{\gamma-1}\label{circle},\\
(-u_1+u_2)\xi_J+v_2\eta_J+\frac{u_1^2}{2}-\frac{u_2^2+v_2^2}{2}+\rho_1^{\gamma-1}-\rho_2^{\gamma-1}=0\label{section3.2RH}.
\end{align}

From (\ref{circle})(\ref{section3.2RH}), we get, at $\sigma=0$,
\begin{align}
\frac{d\xi_J}{d \sigma}=(\frac{ Z}{ X}+1)\sqrt{\overline c_2^2-{Z}^2},\\
\frac{d\eta_J}{d \sigma}=(\frac{ Z}{ X}+1)(u_1+Z).
\end{align}

\subsection{Non-symmetric: Derivatives w.r.t. $\delta$}
\label{nonsymmetricderivativewrtdelta}
In this subsection we fix the angle between wedge and $\eta$-axis as $\sigma_0$, consider the derivatives of $(u_2^+,v_2^+,\rho_2^+)$ w.r.t. $\delta$ for $0\leq\delta\leq\sigma_0$(here consider $u_2^+,v_2^+,\rho_2^+$ as functions of $\delta$).

\begin{figure} \centering
    \includegraphics[height=5cm ]      
 {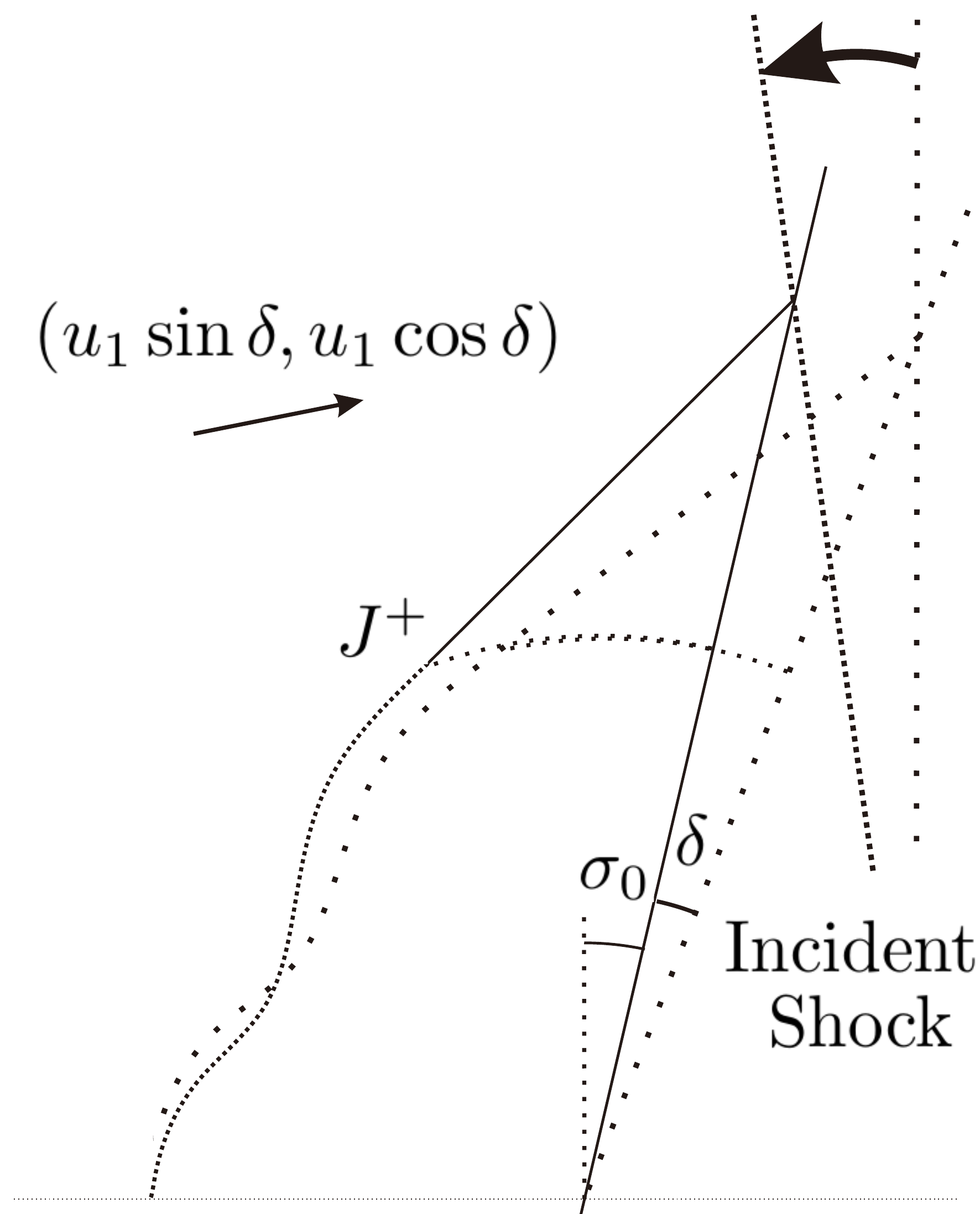}
\caption{Rotation from Symmetric Flow to Symmetric Wedge}
\label{fig:RotationfromSymmetricFlowtoSymmetricWedge}
\end{figure}
With result of last subsection, we get (as illustrated in Figure \ref{fig:RotationfromSymmetricFlowtoSymmetricWedge}):
\[(u_2^+,v_2^+)=(u_2(\sigma_0+\delta),v_2(\sigma_0+\delta))
\left(
\begin{array}{cc}
\cos\delta&\sin\delta\\
-\sin\delta&\cos\delta
\end{array}
\right)\]
\[\Rightarrow \frac{du_2^+}{d\delta}=O(\sigma_0),\ \ \frac{dv_2^+}{d\delta}=u_1({Z\over X}+1)+O(\sigma_0)\]
\[(\xi_J^+,\eta_J^+)=(\xi_J(\sigma_0+\delta),\eta_J(\sigma_0+\delta))
\left(
\begin{array}{cc}
\cos\delta&\sin\delta\\
-\sin\delta&\cos\delta
\end{array}
\right)\]
\[\Rightarrow \frac{d\xi_J^+}{d\delta}=\frac{Z}{X}\sqrt{\overline c_2^2-Z^2}+O(\sigma_0),\ \                     \frac{d\eta_J^+}{d\delta}=u_1(\frac{Z}{X}+1)+\frac{Z^2}{X}+O(\sigma_0)\]

Then we analyze motion of Sonic Circle.

Any point $\vecp$ on \Gammasonicplus\ takes the following form:
\[\vecp=c_2^+(\cos\alpha,\sin\alpha)+(u_2^+,v_2^+)\ ({\pi\over 2}-\sigma_0<\alpha<{\pi\over 2}-\sigma_0+\alpha_0,
\ \alpha_0={\pi\over 2}-{1\over 2}\arccos\frac{Z}{c_2})\]
and \[\frac{d\vecp}{d\delta}=(O(\sigma_0),u_1(\frac{Z}{X}+1)+O(\sigma_0)).\]
So  if we rotate upstream flow counterclockwise($\delta>0$), and $\sigma_0$ small enough s.t. \[\frac{O(\sigma_0)}{u_1(\frac{Z}{X}+1)+O(\sigma_0)}\leq\frac{\sqrt{\overline c_2^2-Z^2}}{2Z},\]
then every point on \Gammasonicplus\ moves up.
If we rotate upstream flow clockwise(corresponds to $\delta<0$), then sonic circle moves down.
\subsection{Comparison of \Statetwoplus\ and \Statetwominus}

Now with computation above we compare physical quantity and geometric structure of \Statetwoplus\ and $\statetwominus$.

\subsubsection{Comparison of Geometric Structures}
\begin{figure}
\centering
    \includegraphics[height=3.5cm ]      
 {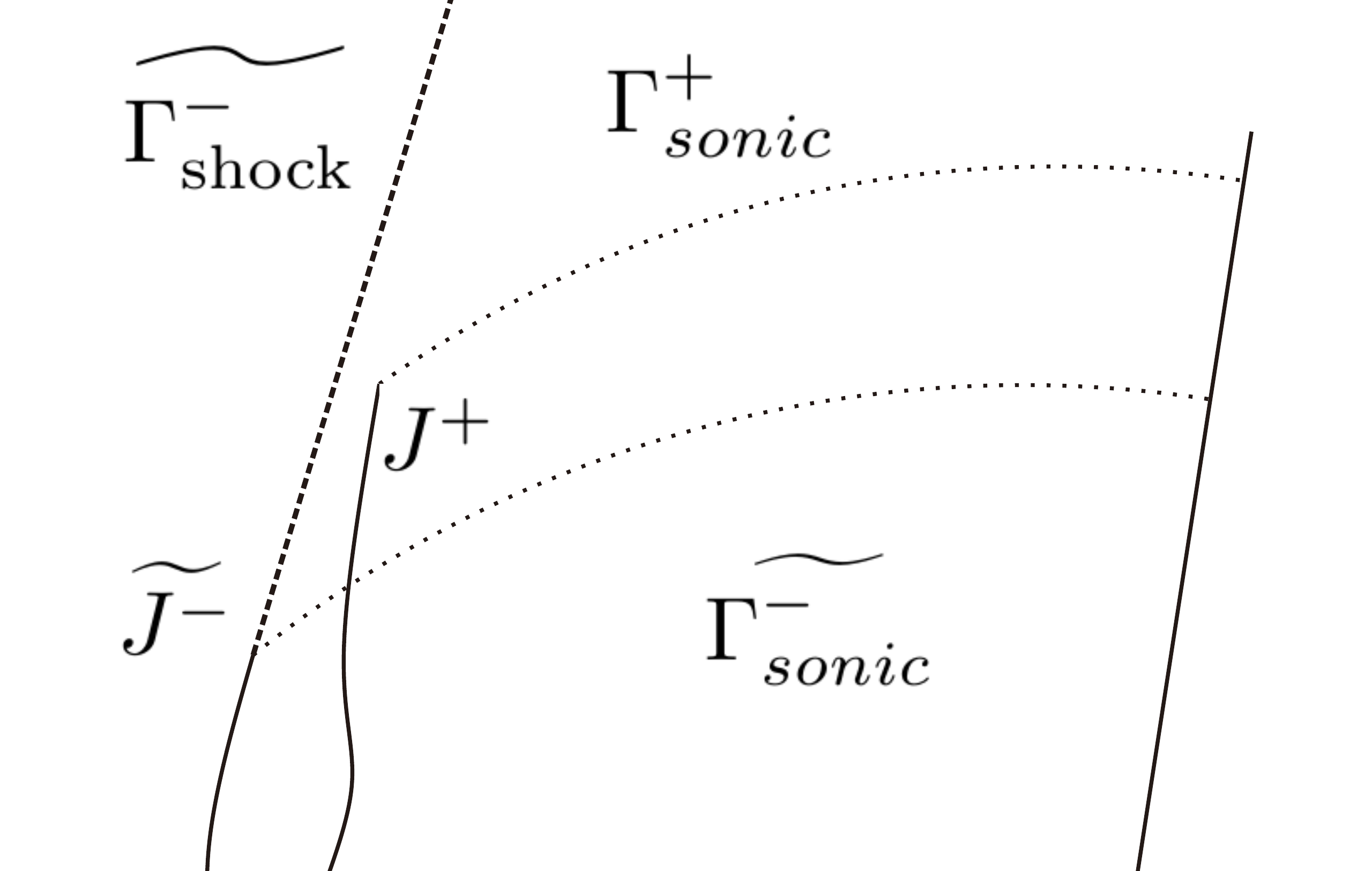}
\caption{Relative Positions of $\gammasonicplus,\ \gammashock,\ \widetilde{\gammasonicminus}$}
\label{fig:relativepositionofboundaryunderperturbation}
\end{figure}

We compare \Statetwoplus\ and State$(\RN{2}_-)$, by reflecting State$(\RN{2}_-)$ across symmetry axis of wedge.
And for any point $p$ (or set $A$), $\tilde{p}$ (and $\tilde{A}$) denotes the reflection of $p$ (and $A$) into symmetry axis of wedge.
So $\widetilde{J^-},\ \widetilde{\gammasonicminus},\ \widetilde{\gammashock}$ are the reflection of $J^-,\ \gammasonicminus,\ \gammashock$, and if we consider them as set valued function of $\delta$, then
\[\widetilde{J^-(\delta)}=J^+(-\delta),\ \widetilde{\gammasonicminus(\delta)}=\gammasonicplus(-\delta).\]
With computation result we have, we know
when $\sigma_0$ small enough and $\delta>0$, $\gammasonic^{+}$ should stay above $\widetilde{\gammasonicminus}$, and $ \widetilde{\Gamma_{\text{shock}}^{-}}$ should not intersect with \Gammasonicplus\ and $\gammashock$. So relative position of $\gammasonicplus,\ \gammashock,\ \widetilde{\gammasonicminus},\ \widetilde{\Gamma_{\text{shock}}^{-}}$ should be as illustrated in  Figure \ref{fig:relativepositionofboundaryunderperturbation}.
\subsubsection{Comparison of $\phi_2^+$ and $\phi_2^-$}\label{comparisonofphi2plusandphi2minus}
With (\ref{sonicspeed}), 
\[\phi_2^+=-{\rho_2^+}^{\gamma-1}+{\rho_0^{\gamma-1}}+u_2^+\xi+v^+_2\eta-\frac{{u_2^+}^2+{v_2^+}^2}{2},\]
\[\phi_2^-=-{\rho_2^-}^{\gamma-1}+{\rho_0^{\gamma-1}}+u_2^-\xi+v^-_2\eta-\frac{{u_2^-}^2+{v_2^-}^2}{2},\]
so on $\{\eta=0\}$, which is the symmetry axis of wedge,\[|\phi_2^+-\phi_2^-|\leq C_p\sigma_0\delta.\]
And $v_2^+-v_2^->\frac{1}{C_p}\delta$, so when $\sigma_0$ small enough
\begin{equation} \phi_2^+>\widetilde{\phi_2^-}+{\delta\over C_p}, \ \ \text{on  } \{\eta>{\sqrt{\overline c_2^2-Z^2}\over 2}\} ,\label{antisymmetricestimateofphi2}
\end{equation}
where, $\widetilde{\phi_2^-}$ is the reflection of $\phi_2^-$, 
\[\widetilde{\phi_2^-}(\xi,\eta)=\phi_2^-(\xi,-\eta).\]
\subsection{Stronger Comparison}\label{strongercomparisonbasedonderivative}
In section 5 we need a more precise estimate, we need to shift $J^{\pm},\ \gammasonic^{\pm}$ by $(\mp u_1\sin\delta\tan\sigma_0,
-u_1\sin\delta)$.
We denote\[\hat J^\pm=J^\pm-(\pm u_1\sin\delta\tan\sigma_0,u_1\sin\delta),\] 
\[ \hat\Gamma_{\text{sonic}}^\pm=\gammasonic^\pm-(\pm u_1\sin\delta\tan\sigma_0,u_1\sin\delta).\]

With previous results, in section \ref{nonsymmetricderivativewrtdelta}, we get
 \[\frac{d\hat J^+}{d\delta}=\frac{Z}{X}(\sqrt{\overline{c}_2^2-Z^2}+O(\sigma_0),u_1+Z+O(\sigma_0)).\]
And  points $\vecq$ on $\hgammasonicplus$ can be parametrized by
\[\vecq=c_2^+(\cos\alpha,\sin\alpha)+(u_2^+,v_2^+)-(u_1\sin\delta\tan\sigma_0,u_1\sin\delta),\]
for
\[ {\pi\over 2}-\sigma_0<\alpha<{\pi\over 2}-\sigma_0+\alpha_0,\ \ \ \alpha_0={\pi\over 2}-{1\over 2}\arccos\frac{Z}{\overline c_2},\]
then,
\[\frac{d\vecq}{d\delta}=(O(\sigma_0),u_1\frac{Z}{X}+O(\sigma_0)).\]
And this shift does not change tangential direction of shock, so after this shift, relative position of $\gammasonicplus,\ \gammashock,\ \widetilde{\gammasonicminus}$ remains, as illustrated in Figure \ref{fig:relativepositionofboundaryunderperturbation}.


\section{Symmetric Estimates}\label{symmetricestimate}

\begin{figure}
       \centering
    \includegraphics[height=9cm]{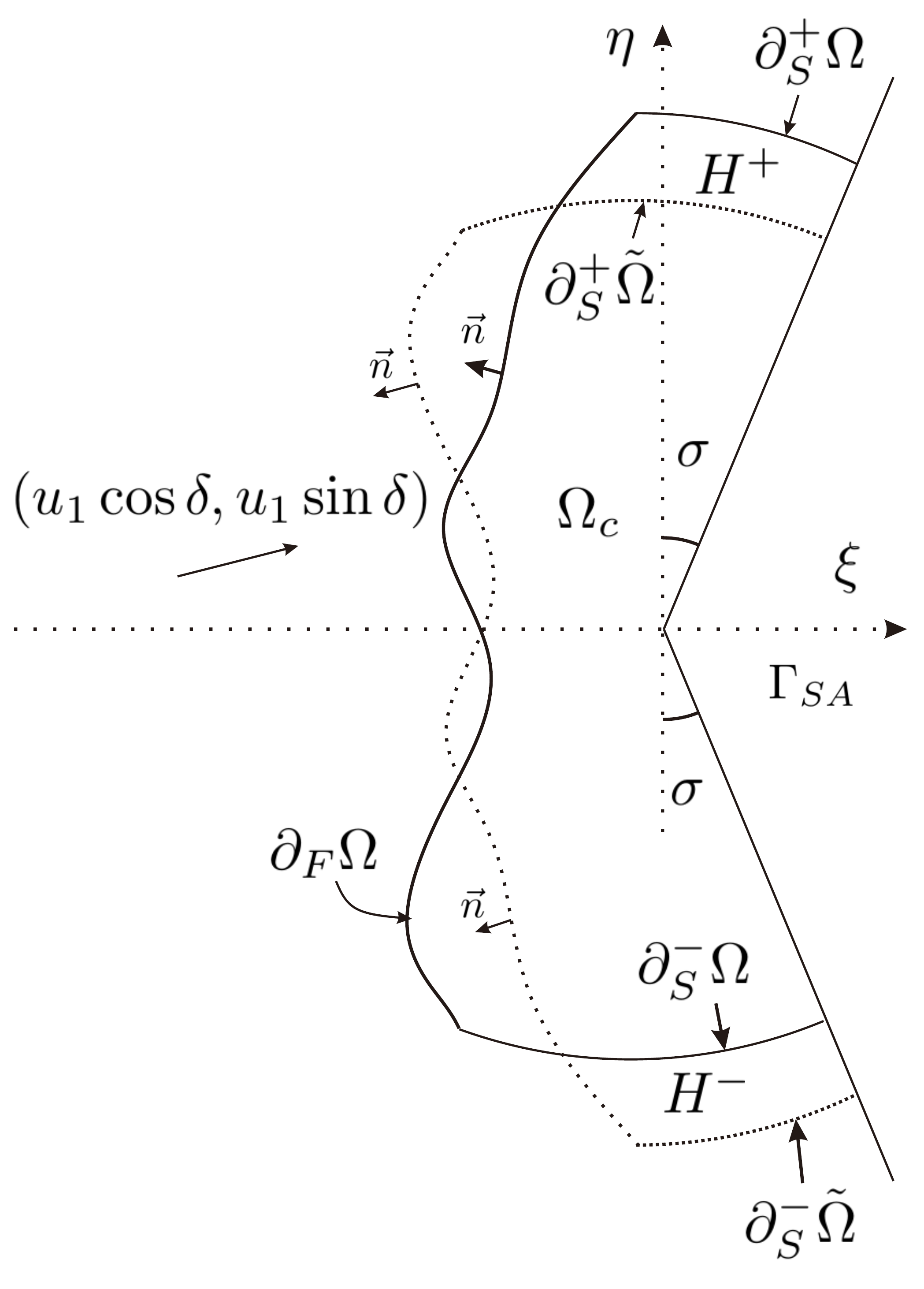}
\caption{Reflective Comparison}
\end{figure}

In this part we show that, given a regular solution to the potential flow equation, with velocity of upstream flow being $(u_1\cos \delta,u_1\sin\delta)$, the solution should be ``$\delta$-symmetric'' w.r.t. $\eta$. Before more precisely present the result of this section, we introduce the following notations:
\begin{notation}
In this section we define:

$\Gamma_{SA}$=Symmetry Axis of wedge.

For any function $f$, $\tilde f(\xi,\eta)=f(\xi,-\eta)$. And for any set $\mathcal{S}$, 
$\tilde {\mathcal{S}}$ is the reflection of $\mathcal{S}$ along $\Gamma_{SA}$,

  $\Omega_{c}=\Omega\cap \tilde\Omega$ (means ``common domain'').

And when $\sigma$ is small, based on estimate of section \ref{gradientestimate}, $\nabla\phi$ is small, so it's clear that we can divide $\partial \Omega$ into``free boundary part'',``sonic circle part" and``upper(or lower) sonic circle part". We denote them by $\partial_F \Omega$, $\partial_S \Omega$ and  $\partial_S^+\Omega(\text{or } \partial_S^-\Omega) $ respectively.
Similarly $\partial_F\Omega_c$ means the free boundary part of  the boundary of ``common domain".

 And the part of $\Omega\backslash\tilde\Omega$ above $\partial_S^+ \tilde\Omega$ is denoted by $H^+$, similarly  the part of  $\tilde\Omega\backslash\Omega$ below $\partial_S^- \Omega$ is denoted by $H^-$.

On $\partial \Omega$ and $\partial\tilde \Omega$,   $\vn$ is unit outer normal direction of $\Omega$ and $\tilde\Omega$.

And in $\Omega\cup\widetilde\Omega$ define:

\begin{minipage}{0.5\linewidth}
\[h=\left\{\begin{array}{cc}
           \phi-\tilde\phi,     &\text{ in } \Omega_c\\
            \phi_1-\tilde\phi, & \text{ in } (\tilde\Omega\backslash\Omega)\backslash H^-\\
              \phi-\widetilde{\phi_1}, &\text{ in } (\Omega\backslash\tilde\Omega)\backslash H^+\\
             \phi-\widetilde{\phi_2^-},&\text{ in } H^+\\
             \phi_2^--\tilde\phi, &   \text{ in } H^-
             \end{array}
     \right. 
\]
\end{minipage}
\begin{minipage}{0.5\linewidth}
\[h^+=\max\{h,0\},\]
\[ w=(h^+)^2.\]
\end{minipage}
\end{notation}

In this section we prove:

\begin{description}
\item [Section  \ref{integralsymmetricestimate}]\ \ \ $\int_{\Omega_c}|\phi-\tilde\phi|^3\leq C_p\cdot\delta^3             $
\item  [Section  \ref{linftyestimatenearcornerofwedge}]\ $ |\phi-\tilde\phi|\leq C_p\delta $,\ \ \ \  in $ B_{Z\over 8}(\MC)$
\item [Section \ref{gradientestimatenearcornerofwedge}]
					$\  |\phi_\eta|\leq C_p\delta r^{{\alpha\over 2}-1},\ \ \   \text{along $\Gamma_{SA}$ near corner of wedge }    $                          
\item [Section \ref{linftyestimatenearfreeboundary}]
					$\ |\phi-\tilde\phi|\leq C_p\delta \text{,\ \ \  in a neighborhood of $\gammashock\cap\Gamma_{SA}$}     $                                                                   
\item [Section \ref{gradientestimatenearfreeboundary}]
                                \ \   $|\phi_\eta|\leq C_p\delta,\ \ \ \  \text{along $\Gamma_{SA}$ near \Gammashock}        $
\end{description}

\subsection{Integral Symmetric Estimates}\label{integralsymmetricestimate}
Now we  consider the following integral:

\begin{align*}
I&=\int_{\Omega_c}\rho\nabla\cdot\varphi\nabla w-\tilde\rho\nabla\tilde\varphi\cdot  \nabla w-2\rho w+2\tilde\rho w\\
&=\int_{\Omega_c}\nabla\cdot[\rho\nabla\varphi w-\tilde\rho\nabla\tilde\varphi w]\ (\text{we used equation }\nabla\cdot[\rho\nabla\varphi]+2\rho=0 ).
\end{align*}
Since we know \Gammashock\ is the graph of a $C^1$ function of $\eta$ (\ref{gammashockisaconefunctionofeta}), $\partial\Omega_c $ should be Lipschitz, so $\vn$ is defined a.e.. With Green formula generalized to domain with Lipschitz boundary (Lemma 14.4 of \cite{LucCarneige}), we  can reduce $I$ to boundary integral,
\begin{align*}
I&=\int_{\partial\Omega_c}\rho D_{\vn}\varphi w-\tilde\rho D_{\vn}\tilde\varphi w\\
&=\int_{\partial_S \Omega_c}\rho D_{\vn}\varphi w-\tilde\rho D_{\vn}\tilde\varphi w
  +  \int_{\partial_F \Omega_c}\rho D_{\vn}\varphi w-\tilde\rho D_{\vn}\tilde\varphi w\triangleq I_S+I_F.
\end{align*}

 Here we don't need  to consider boundary integral on boundary of wedge, because on boundary of wedge  $D_{\vn}\varphi=0$. 


Then we estimate integral on $\partial_S\Omega_c$ and $\partial_F\Omega_c$ separately, first we estimate integral on sonic circle.
\begin{align*}
I_S= &		\int_{\partial_S\Omega_c}\rho D_{\vn}\varphi w-\trho\Dn\tvphi w\\
=&		\int_{\partial_S^+\Omega_c}\rho D_{\vn}\varphi w-\trho\Dn\tvphi w
   		+  \underbrace{\int_{\partial_S^-\Omega_c}\rho\Dn\varphi w-\trho\Dn\tvphi w}_{=0(w=0, (\ref{antisymmetricestimateofphi2}), (\ref{comparisonofphiandphitwoplus}) )}.
\end{align*}
 By symmetry,  (\ref{antisymmetricestimateofphi2}) of section \ref{comparisonofphi2plusandphi2minus} implies, on $\partial_S^-\Omega_c$ $\phi_2^-<\widetilde{\phi^+_2}$; and $\widetilde{\phi^+_2}\leq \tilde\phi$ (\ref{comparisonofphiandphitwoplus}), so $\phi=\phi_2^-<\widetilde{\phi^+_2}\leq \tilde\phi$, $w=0$. And similar argument shows $h>0$ in $H^+$, so 
\begin{align*}
I_S=&         \int_{\partial_S^+\Omega_c}\rho\Dn\varphi w-\widetilde{\rho_2^-}\Dn\widetilde{\varphi_2^- }w\ 
                                (\text{because  }\ \phi=\phi_2^-,\ \nabla\phi=\nabla\phi_2^- \text{ on }\gammasonicminus)\\
=&           \int_{H^+}-\nabla\cdot[(\rho\nabla\varphi-\widetilde{\rho_2^-}\nabla\widetilde{\varphi_2^- })h^2]\\
&	  	+ \int_{\partial_S^+\Omega}\underbrace{(\rho\Dn\varphi-\widetilde{\rho_2^-}\Dn\widetilde{\varphi_2^- })h^2}_{\leq C_p \delta^3(\rho=\rho_2^+,\nabla\varphi=\nabla\varphi_2^+)}
            +\underbrace{\int_{(\partial_F\Omega)\cap\overline{H^+}}(\rho\Dn\varphi-\widetilde{\rho_2^-}\Dn\widetilde{\varphi_2^-} )h^2}_{\leq C_p\delta^3\ (h\leq C_p\delta,\text{ and length of curve}\leq C_p\delta)}\\
\leq&		\int_{H^+}-2h[(\rho\nabla\varphi-\widetilde{\rho_2^-}\nabla\widetilde{\varphi_2^-} )\cdot\nabla h-(\rho-\widetilde{\rho_2^-})h]+C_p\delta^3\\
\leq&             \int_{H^+}-2h[(\rho\nabla\varphi-\rho_2^+\nabla\varphi_2^+ )\cdot\nabla g-(\rho-\rho_2^+)g]+C_p\delta^3\ (\text{here } g\triangleq\phi-\phi_2^+).
\end{align*}
Last inequality follows from 
\begin{itemize}
\item width of $H^+\leq C_p\delta,\ h\leq C_p\sigma$;
\item $|\phi_2^+-\widetilde{\phi_2^-}|\leq C_p\delta, |\nabla\phi_2^+-\nabla\widetilde{\phi_2^-}|\leq C_p\delta,|\rho_2^+-\widetilde{\rho_2^-}|\leq C_p\delta$.
\end{itemize}
Then as in  \cite{CFcomparison}, define\[\ft=\varphi_2^++tg,\ \rt=\left(\rho_0^{\gamma-1}-\ft-\frac{{|\nabla\ft|}^2}{2}\right)^{1\over\gamma-1}.\]
\begin{align*}
   &(\rho\nabla\varphi-\rho_2^+\nabla\varphi_2^+)\cdot\nabla g-(\rho-\rho_2^+)g\\
=&\int_0^1\frac{d}{dt}[\rt\nabla\ft]\cdot\nabla g-(\frac{d}{dt}\rt)g\\
=&\int_0^1\rt|\nabla g|^2+\frac{\rt^{2-\gamma}(-g-\nabla\ft\cdot\nabla g)}{\gamma-1}\nabla\ft\cdot\nabla g+\int_0^1\frac{\rt^{2-\gamma}}{\gamma-1}(g+\nabla \ft \cdot\nabla g)g\\
=&\int_0^1[\rt\delta_{ij}-\frac{\rt^{2-\gamma}}{\gamma-1}\ft_i\ft_j]g_i g_j+\int_0^1\frac{\rt^{2-\gamma}}{\gamma-1}g^2\geq0.
\end{align*}
In  inequality above, 
\begin{equation}\left(\rt\delta_{ij}-\frac{\rt^{2-\gamma}}{\gamma-1}\ft_i\ft_j\right)>0,\ \ (\text{as matrix}).
\label{ellipticityoflinearconnection}\nonumber
\end{equation}
This follows from the fact that, 
\begin{align*}
   &\left(\delta_{ij}-\frac{{r^{(0)}}^{1-\gamma}}{\gamma-1}{f^{(0)}}_i{f^{(0)}}_j\right)
=\left(\delta_{ij}-\frac{\varphi^+_{2,i}\varphi^+_{2,j}}{(\gamma-1){\rho_2^+}^{\gamma-1}}\right)>0,\\
   &\left(\delta_{ij}-\frac{{r^{(1)}}^{1-\gamma}}{\gamma-1}{f^{(1)}}_i{f^{(1)}}_j\right)
=\left(\delta_{ij}-\frac{\varphi_{i}\varphi_{j}}{(\gamma-1){\rho}^{\gamma-1}}\right)>0,
\end{align*} because $H^+$ is the subsonic region for both $\varphi$ and $\varphi_2^+$, and 
\begin{align*}
-(\gamma-1)\rt^{\gamma-1}+|\nabla\ft|^2 \text{ is a convex function of } t.
\end{align*}
 More detail and explanation of this method is contained in \cite{CFcomparison}.
So \[\int_{\partial_S\Omega_c}\rho D_{\vn}\varphi w-\trho\Dn\tvphi w\leq C_p\delta^3.\]
Then we estimate integral on free boundary
\begin{align}
I_F=  &\int_{\partial_F\Omega_c}(\rho\nabla\varphi w-\tilde\rho\nabla\tilde\varphi w)\cdot\vn\nonumber\\
=&\int_{(\partial_F\Omega)\cap \tilde\Omega}( \rho_1\nabla\varphi_1 w-\tilde\rho\nabla\tilde\varphi w)\cdot\vn
   +\int_{(\partial_F\tilde \Omega)\cap \Omega}( \rho\nabla\varphi w-\widetilderhoone\nabla\widetilde{\varphi_1} w)\cdot\vn\nonumber\\
&+\int_{(\partial_F\Omega)\cap {(\partial_F\tilde\Omega)}}(\rho\nabla\varphi w-\tilde\rho\nabla\tilde\varphi w)\cdot\vn\label{symmetricintegralonshockmeetsshock}.
\end{align}
 In (\ref{symmetricintegralonshockmeetsshock}), the third term can be a little complex, since $(\partial_F\Omega)\cap {(\partial_F\tilde\Omega)}$ may be a set of isolated points, but with positive 1-dim Hausdorff measure. And on $(\partial_F\Omega)\cap {(\partial_F\tilde\Omega)}$, $\vn$ is not trivially related to normal of $\Omega$ and $\tilde\Omega$. 

Since we know \Gammashock\ is the graph of a $C^1$ function of $\eta$ (\ref{gammashockisaconefunctionofeta}),   so let \Gammashock\ be the graph of $f_1$, and $\widetilde\gammashock=\partial_F\tilde\Omega$ be the graph of $f_2$, with $f_1$ and $f_2$ both being $C^1$ functions of $\eta$. And let $Proj$ be the projection from ${\mathbb{R}}^2$ to ${\mathbb{R}}=\{\eta\}$. So on $\mathcal{F}\triangleq Proj(\partial_F\Omega_c)\subset \mathbb{R}$, $f_1$ and $f_2$ are both defined, and actually, in the following, we only need to consider $f_1$ and $f_2$ as functions defined  on $\mathcal{F}$
, since $\gammashock\cap\widetilde{\gammashock}\subset \partial_F\Omega_c$.

So $Proj(\gammashock\cap\widetilde{\gammashock})=\{f_1=f_2\}$ is a measurable set in $\mathbb{R}$. And since  we have an estimate of $f_1^\prime$ and $f_2^\prime$ ($|u|\leq C_p\sigma$ (\ref{estimateofu}), $|v|\leq C_p\sigma$ (\ref{estimateofv})), the 1-dim Hausdorff measure on $\gammashock\cap\widetilde{\gammashock}$ is equivalent to 1-dim Lebesgue measure on $\{f_1=f_2\}$. 

On $\{f_1=f_2\}$, except a set of measure zero, any point is a Lebesgue point. For a Lebesgue point of $\{f_1=f_2\}$, say $a$, we can find a sequence $\{a_i\}_{i=0}^\infty$, s.t. 
\begin{equation}
\lim_{i\rightarrow \infty}a_i=a \ (a_i\neq a),\ \ \text{and }\ \ f_1(a_i)=f_2(a_i).
\end{equation}
Since we assumed $f_1$ and $f_2$ are $C^1$ functions of $\eta$,
\begin{equation}
f_1^\prime(a)=\lim_{i\rightarrow \infty}\frac{f_1(a_i)-f_1(a)}{a_i-a}=\lim_{i\rightarrow \infty}\frac{f_2(a_i)-f_2(a)}{a_i-a}
=f_2^\prime(a),
\end{equation}
$f_1^\prime=f_2^\prime$ a.e. on   $\{f_1=f_2\}$. So, on $\gammashock\cap\widetilde{\gammashock}$, \Gammashock\ and $\widetilde{\gammashock}$ have same normal directions a.e.,  w.r.t. the 1-dim Hausdorff measure on $\gammashock\cap\widetilde{\gammashock}$. So, 
\begin{equation}
(\rho\nabla\varphi-\tilde\rho\nabla\tilde\varphi)\cdot{\vn}=(\rho_1\nabla\varphi_1-\widetilderhoone\nabla\widetilde{\varphi_1})\cdot{\vn}\leq C_p\delta \ \  \text{a.e. on $\gammashock\cap\widetilde{\gammashock}$}.\label{boundaryestimateofvectoronshockmeetsshock}
\end{equation}

And on $\gammashock\cap\widetilde{\gammashock}$, $w\leq |h|^2=|\widetilde{\varphi_1}-\varphi_1|^2\leq C_p\delta^2$. Combining this with (\ref{boundaryestimateofvectoronshockmeetsshock}) gives
\[\int_{(\partial_F\Omega)\cap {(\partial_F\tilde\Omega)}}(\rho\nabla\varphi w-\tilde\rho\nabla\tilde\varphi w)\cdot\vn\leq C_p \delta^3.\]
By plugging this into (\ref{symmetricintegralonshockmeetsshock}), we can reduce $I_F$ to
\[I_F\leq \int_{(\partial_F\Omega)\cap \tilde\Omega}( \rho_1\nabla\varphi_1 w-\tilde\rho\nabla\tilde\varphi w)\cdot\vn
   +\int_{(\partial_F\tilde \Omega)\cap \Omega}( \rho\nabla\varphi w-\widetilderhoone\nabla\widetilde{\varphi_1} w)\cdot\vn+C_p\delta^3.\]
Then, we transfer first two terms into integrals out of $\overline\Omega_c$.
{\allowdisplaybreaks
\begin{align*}
I_F=&\int_{(\tilde\Omega\backslash\Omega)\backslash H^-}\nabla\cdot[-\rho_1 \nabla\varphi_1 w+\tilde\rho\nabla\tilde\varphi w ]
  +\int_{(\Omega\backslash\tilde\Omega)\backslash H^+}\nabla\cdot[-\rho\nabla\varphi w+\widetilderhoone\nabla\widetilde{\varphi_1} w]\\
 &+\int_{(\partial_F\tilde\Omega)\setminus(\Omega\cup\overline{H^-})}(\rho_1\nabla\varphi_1-\tilde\rho\nabla\tilde\varphi)\cdot\vn w
   +\int_{(\partial_F \Omega)\setminus(\tilde\Omega\cup\overline{H^+})}(\rho\nabla\varphi -\widetilderhoone\nabla\widetilde{\varphi_1})\cdot\vn w\\
  &+\underbrace{\int_{(\partial_S^+\tilde\Omega)\setminus\Omega}(\rho_1\nabla\varphi_1-\tilde\rho\nabla\tilde\varphi)\cdot\vn w
   +\int_{(\partial_S^-\Omega)\setminus\tilde\Omega}(\rho\nabla\varphi-\widetilderhoone\nabla\widetilde{\varphi_1})\cdot\vn w}_
{\leq C_p\delta^3,\text{ because } w\leq C_p \delta^2, \text{ and length of integral curve} \leq C_p\delta.}
    +C_p\delta^3\\
\leq&\int_{(\tilde\Omega\backslash\Omega)\backslash H^-}\nabla\cdot[-\rho_1 \nabla\varphi_1 w+\tilde\rho\nabla\tilde\varphi w ]
  +\int_{(\Omega\backslash\tilde\Omega)\backslash H^+}\nabla\cdot[-\rho\nabla\varphi w+\widetilderhoone\nabla\widetilde{\varphi_1} w]\\
&+\int_{(\partial_F\tilde\Omega)\setminus(\Omega\cup\overline{H^-})}\underbrace{(\widetilderhoone\nabla\widetilde{\varphi_1}-\tilde\rho\nabla\tilde\varphi)\cdot\vn}_{=0\ (RH\text{ condition})} w
   +\int_{(\partial_F \Omega)\setminus(\tilde\Omega\cup\overline{H^+})}\underbrace{(\rho\nabla\varphi -\rho_1\nabla\varphi_1)\cdot\vn }_{=0\ (RH\text{ condition})}w\\
&+\underbrace{\int_{(\partial_F\tilde\Omega)\setminus(\Omega\cup\overline{H^-})}(\rho_1\nabla\varphi_1-\widetilderhoone\nabla\widetilde{\varphi_1})\cdot \vn w
  +\int_{(\partial_F\Omega)\setminus(\tilde\Omega\cup\overline{H^+})}(\rho_1\nabla\varphi_1-\widetilderhoone\nabla\widetilde{\varphi_1})\cdot \vn w}
                                           _{ \leq C_p\delta^3,\ \left(\text{because } w\leq C_p \delta^2,\ \nabla(\widetilde{\varphi_1}-\varphi_1)\leq 2 u_1\delta,\ \rho_1=\widetilderhoone=\text{constant} \right)}\\
&+C_p \delta^3\\
\leq& C_p\delta^3
+\int_{(\tilde\Omega\backslash\Omega)\backslash H^-}2h^+[(-\rho_1\nabla\varphi_1+\tilde\rho\nabla\tilde\varphi)\cdot\nabla h+(\rho_1- \tilde\rho)h^+]\\
&+\int_{(\Omega\backslash\tilde\Omega)\backslash H^+}2h^+[(-\rho\nabla\varphi+\widetilderhoone\nabla\widetilde{\varphi_1})\cdot\nabla h+(\rho- \widetilde{\rho_1})h^+].
\end{align*}
}
In the following we discuss integrals in $(\tilde\Omega\backslash\Omega)\backslash H^-$ and $(\Omega\backslash\tilde\Omega)\backslash H^+$ separately.

For convenience, we define:
\begin{equation}\label{eq:fourvariablerhfunction}
RH(r_1,f_1,r_2,f_2)=(r_1\nabla f_1-r_2\nabla f_2)\cdot(\nabla f_1-\nabla f_2).
\end{equation}
With this notation, on $\gammashock$, $RH(\rho,\varphi,\rho_1,\varphi_1)=0$ is RH condition; and in $\Omega$, $RH(\rho,\varphi,\rho_1,\varphi_1)$ is the function RH defined and discussed in section \ref{gradientestimate} and \ref{estimateofRHfunction}.

1)In  $(\tilde\Omega\backslash\Omega)\backslash H^-$, by estimate of section \ref{estimateofRHfunction}
\[
 (-\rho_1\nabla\varphi_1+\tilde\rho\nabla\tilde\varphi)\cdot\nabla h=-RH(\rho_1,\varphi_1, \tilde\rho,\tilde\varphi)\leq C_p\delta-RH(\widetilderhoone,\widetilde{\varphi_1}, \tilde\rho,\tilde\varphi)\leq C_p\delta-(\tilde\rho-\rho_1)(\widetilde{\phi_1}-\tilde\phi)\]
 so $-RH(\rho_1,\varphi_1,\tilde\rho,\tilde\varphi)>0$ only in a $C_p\delta$ neighborhood of $\partial_F\tilde\Omega$, and in this neighborhood $-RH(\rho_1,\varphi_1,\tilde\rho,\tilde\varphi)\leq C_p \delta$
then
\[
\int_{ (\tilde\Omega\backslash\Omega)\backslash H^-}2h^+[(-\rho_1\nabla\varphi_1+\tilde\rho\nabla\tilde\varphi)\cdot \nabla h]\leq C_p\delta^3.
\]
And note that $\tilde\rho>\rho_1, \ \text{in }{\widetilde\Omega}$, so
\[\int_{(\tilde\Omega\backslash\Omega)\backslash H^-}(\rho_1-\tilde\rho)(h^+)^2\leq0.\]

2)In $(\Omega\backslash\tilde\Omega)\backslash H^+$, $h=\phi-\widetilde{\phi_1}$, so $h_\xi\leq{-u_1 \over 2}$ (providing  $\sigma$ small enough, s.t. $|\nabla\phi|\leq{u_1\over 2}$) and on $\partial_F(\Omega\cup\tilde\Omega)$ $|h|\leq C_p\delta$, so only in a $C_p\delta$ neighborhood of $\partial_F(\Omega\cup\tilde\Omega)$ $h^+\neq0$, and in this neighborhood $h^+\leq C_p\delta$, so we have:
 \[
\int_{(\Omega\backslash\tilde\Omega)\backslash H^+}2(\rho-\widetilderhoone){(h^+)}^2\leq C_p\delta^3
\]
and in $\{h\neq 0\}\cap(\Omega\backslash\tilde\Omega)\backslash H^+$, which is contained in a strip with width $\leq C_p\delta$:
\begin{align*}
  &(-\rho\nabla\varphi+\widetilderhoone\nabla\widetilde{\varphi_1})\cdot(\nabla\phi-\nabla\widetilde{\phi_1})\\
=&\underbrace{(-\rho\nabla\varphi+\rho_1\nabla\varphi_1)\cdot(\nabla\phi-\nabla\phi_1)}_{=-RH\leq0 \text(\ref{lowerboundofRH} )}+(-\rho\nabla\varphi+\widetilderhoone\nabla\widetilde{\varphi_1})\cdot(-\nabla\widetilde{\phi_1}+\nabla\phi_1)\\
&\ \ \ \ \  \ \ +(-\rho_1\nabla\varphi_1+\widetilderhoone\nabla\widetilde{\varphi_1})\cdot(\nabla\phi-\nabla\phi_1)\\
\leq&C_p\delta
\end{align*}
  so \[\int_{(\Omega\backslash\tilde\Omega)\backslash H^+}2h^+[(-\rho\nabla\varphi+\widetilderhoone\nabla\widetilde{\varphi_1})\cdot\nabla h+(\rho-\widetilderhoone)h^+]\leq C_p\delta^3.\]

Put estimates above together, we get
 $ I\leq C_p \delta^3$.

Then we define:
\begin{equation}\label{connection}
  \varphi^{(t)}=\tilde\varphi+t(\varphi-\tilde\varphi),\ \  \rho^{(t)}=\left(\rho_0^{\gamma-1}-\varphit-\frac{|\nabla \varphit|^2}{2}\right)^{ 1 \over \gamma -1}.
\end{equation}Again with computation in  \cite{CFcomparison},
\begin{align*}
I&\geq\int_{\Omega_c}\int^1_0\frac{d}{dt}[\rhot\nabla \varphit\nabla w-2\rhot w]\\
                        &=\int_{\Omega_c}2(\int^1_0\frac{{\rhot}^{2-\gamma}}{\gamma-1})(h^+)^3+2(\int^1_0 a^{(t)}_{ij})h_i h_j h^+\ \ \ 
													\left(a^{(t)}_{ij} =\rhot\delta_{ij}-\frac{{\varphit}_i{\varphit}_j{\rho^{(t)}}^{2-\gamma}}{\gamma-1}\right)\\
                        &\geq\frac{{\overline{\rho_2}}^{2-\gamma}}{\gamma-1}\int_{\Omega_c}(h^+)^3\ 
                              \left(\text{if $\sigma$ small enough s.t. $\rho^{2-\gamma}>\frac{\overline\rho_2^{2-\gamma}}{2}$}\right).
\end{align*}
So analysis in this section shows
\[\int_{\Omega_c}(h^+)^3\leq C_p I\leq C_p\delta^3.\]
Then, because $\phi-\tilde\phi$ is an odd function of $\eta$, we know :
\begin{equation}
\int_{\Omega_c}|\phi-\tilde\phi|^3\leq C_p\delta^3\label{integralsymmetricestimateinequality}.
\end{equation}
\subsection{Symmetric Estimates near the Corner of the Wedge}
In this section, from the integral estimate of last section we get:
\begin{itemize}

\item near the corner of the wedge, $|\phi-\tilde\phi|\leq C_p\delta$, (Section \ref{linftyestimatenearcornerofwedge});

\item on $\Gamma_{SA}$ near corner of wedge, $ |\phi_\eta|\leq C_p\delta d^{{\alpha\over 2}-1}_c$, (Section \ref{gradientestimatenearcornerofwedge})  .
\end{itemize}
In above $d_\mathcal{C}$ means distance to corner of wedge.
\subsubsection{$L^\infty$ Estimate near Corner of Wedge}\label{linftyestimatenearcornerofwedge}
Take $g\in C_0^{\infty}(B_R(\mathcal{C}))$, $g=1$ in $B_r(\mathcal{C})$ $(r<R\leq { Z \over 4})$ and $|\nabla g|\leq \frac{2}{R-r}$. In this section $\mathcal{B}$ denotes $B_R(\mathcal{C})\backslash W$, and define $k=(\phi-\tilde\phi)^+$.

Consider:($s\geq 2$)
\begin{align}
I=&\int_{\mathcal{B}}\rho\nabla\varphi\cdot\nabla(k^sg^2)-2\rho k^s g^2-\tilde\rho\nabla\tilde\varphi\cdot\nabla(k^s g^2)+2\tilde\rho k^sg^2\\
 =&\int_{\mathcal{B}}\nabla\cdot[\rho\nabla\varphi k^s g^2-\tilde\rho\nabla\tilde\varphi k^s g^2]\\
 =&\int_{\partial \mathcal{B}}\rho\varphi_{\vn}k^s g^2-\tilde\rho\tilde\varphi_{\vn}k^s g^2=0.
\end{align}
On the other hand
\begin{align*}
I&=\int_\mathcal{B}\int^1_0{d \over dt}[\rhot\nabla\varphit\cdot\nabla(k^s g^2)-2\rhot k^s g^2]\\
 &=\int_\mathcal{B}\int^1_0\frac{\rhot^{2-\gamma}}{\gamma-1}(-k-\nabla\varphit\cdot\nabla k)\nabla\varphit\cdot\nabla(k^s g^2)+\rhot\nabla k \cdot\nabla(k^s g^2)\\
&\ \ \ +\frac{2{\rhot}^{2-\gamma}}{\gamma-1}k^{s+1}g^2+\frac{2{\rhot}^{2-\gamma}}{\gamma -1}\nabla\varphit\cdot\nabla k k^s g^2\\
 &=\int_\mathcal{B}\mu_1 k^{s+1} g^2+s k^{s-1}A_{ij}k_ik_j g^2+2\mu_2k^s\nabla k\cdot\nabla g g+(2-s)\vec B\cdot \nabla k k^s g^2\\
&\ \ \ -2\vec B\cdot \nabla g k^{s+1} g-2C_{ij}k_ig_j g k^s,
\end{align*}
where, in above, $\rho^{(t)}$ and $\varphi^{(t)}$ are defined as in (\ref{connection}), and
\[
\begin{array}{cc}
\mu_1=\int^1_0\frac{2{\rhot}^{2-\gamma}}{\gamma-1}  &A_{ij}=\int^1_0\rhot\left(\delta_{ij}-\frac{{\rhot}^{1-\gamma}}{\gamma-1}\varphit_i\varphit_j\right)\\
\mu_2=\int^1_0\rhot                                                             &\vec B=\int_0^1\frac{{\rhot}^{2-\gamma}}{\gamma-1}\nabla\varphit  \ \ \ \ C_{ij}=\int_0^1\frac{{\rhot}^{2-\gamma}}{\gamma-1}\varphit_i\varphit_j
\end{array}.
\]

When $\sigma$ small enough, we have in $\mathcal{B}$
\[\mu_1\geq0, \ \ \ 0<\mu_2\leq 2\overline \rho_2, \ \ |\vec B|\leq\frac{2\overline\rho_2}{\overline c_2}, \ \  0\leq(C_{ij})\leq2 \overline\rho_2,\ \ 
\frac{\overline\rho_2}{C_p}\leq(A_{ij})\leq C_p\overline\rho_2.\].

So:
\begin{align*}
\int_\mathcal{B} sk^{s-1}\frac{\overline\rho_2}{C_p}|\nabla k|^2g^2 
\leq \int_\mathcal{B}\frac{4\overline\rho_2}{\overline c_2}k^{s+1}|\nabla g|g
+(s-2)\frac{2 \overline \rho_2}{\overline c_2}|\nabla k|k^s g^2+8\overline\rho_2|\nabla k||\nabla g| k^sg\\
\int_\mathcal{B} sk^{s-1}|\nabla k|^2g^2\leq     C_p\int_\mathcal{B} k^{s+1}\left(\frac{32}{\overline c_2}|\nabla g|g +\frac{64s}{\overline c_2^2}g^2+256|\nabla g|^2\right)\\
\|k^{s+1 \over 2} g\|^2_{4;\mathcal{B}}\leq C_p s^2\int_\mathcal{B} k^{s+1}\left(|\nabla g|^2+1\right).
\end{align*}
This implies, for $\frac{ Z}{8}\leq r< R\leq \frac{ Z}{4},$
\[\|k^{s+1 \over 2}\|_{4;B_r}\leq \frac{C_p s}{R-r}\|k^{s+1 \over 2}\|_{2;B_R}.\]
Then with Moser Iteration we get:
\[
|k|_{\infty;B_{ Z \over 8}\backslash W}\leq C_p\left(\int_{B_\frac{Z}{4}\backslash W} k^3\right)^{1 \over 3}\leq C_p\delta.
\]
So 
\begin{equation}
|\phi-\tilde\phi|\leq C_p\delta, \text{ in } B_{Z\over 8}\setminus W.
\end{equation}
\subsubsection{Gradient Estimate near Corner of Wedge}\label{gradientestimatenearcornerofwedge}
Here, consider the difference of equations of $\phi$ and $\tilde\phi$:
\begin{align}
\left[c^2-(\phi_\xi-\xi)^2\right]\phi_{\xi\xi}-2(\phi_\xi-\xi)(\phi_\eta-\eta)\phi_{\xi\eta}+\left[c^2-(\phi_\eta-\eta)^2\right]\phi_{\eta\eta}=0\\
\left[\tilde c^2-(\tilde\phi_\xi-\xi)^2\right]\tilde\phi_{\xi\xi}-2(\tilde\phi_\xi-\xi)(\tilde\phi_\eta-\eta)\tilde\phi_{\xi\eta}+\left[\tilde c^2-(\tilde\phi_\eta-\eta)^2\right]\tilde\phi_{\eta\eta}=0,
\end{align}
we get:
\begin{equation}\label{eq:equationofhbytakingsympledifference}
a_{ij}h_{ij}+b_1h_\xi+b_2h_\eta+b_3 h=0,
\end{equation}
where,
\begin{align*}
(a_{ij})&=\left[\begin{array}{cc}
                     c^2-(\phi_\xi-\xi)^2               &                   -(\phi_\xi-\xi)(\phi_\eta-\eta)\\
                      -(\phi_\xi-\xi)(\phi_\eta-\eta)&                   c^2-(\phi_\eta-\eta)^2
                      \end{array}
              \right]\\
b_1&=-{\gamma+1\over 2}(\tilde\varphi_\xi+\varphi_\xi)\tilde\phi_{\xi\xi}-2\tilde\varphi_\eta\tilde\phi_{\xi\eta}-{\gamma-1\over 2}(\tilde\varphi_\xi+\varphi_\xi)\tilde\phi_{\eta\eta}\\
b_2&=-{\gamma+1\over 2}(\tilde\varphi_\eta+\varphi_\eta)\tilde\phi_{\eta\eta}-2\varphi_\xi\tilde\phi_{\xi\eta}-{\gamma-1\over 2}(\tilde\varphi_\eta+\varphi_\eta)\tilde\phi_{\xi\xi}\\
b_3&=-(\gamma-1)\triangle\tilde\phi.
\end{align*}
By apriori estimate in section \ref{holdergradientestimateatcornerofwedge}, there exists $\alpha$ does not depend on $\sigma$ s.t. $|\vec b|\leq C_p r^{\alpha -1}$, and $|a_{ij}-\delta_{ij}|\leq C_pr^\alpha$, then by Lemma \ref{Lemma A2} we have in a small neighborhood of corner, along $\Gamma_{SA}$ :
\begin{align}
|h_\eta|\leq C_p r^{{\alpha \over 2}-1}\delta \Rightarrow |\phi_\eta|\leq C_p r^{{\alpha\over 2}-1}\delta.
\end{align}
\subsection{Symmetric Estimates near the Free Boundary}
In this part from integral estimate:
\[\int_{\Omega_c}|\phi-\tilde\phi|^3\leq C_p\delta^3,
\] 
we will get:
\begin{itemize}
\item near  $\gammashock\cap\Gamma_{SA}$, \ \  $|\phi-\tilde\phi|\leq C_p\delta \ \     $ (Section \ref{linftyestimatenearfreeboundary});
\item  near $\gammashock$ on $\Gamma_{SA}$,  \ \ $|\phi_\eta|\leq C_p\delta    \ \ \ \ $ (Section \ref{gradientestimatenearfreeboundary}).
\end{itemize}

\subsubsection{$L^\infty$ Estimate near Free Boundary}\label{linftyestimatenearfreeboundary}

In this subsection denote
 $P=\gammashock\cap\Gamma_{SA} $ and $Q=Q_{\min\{{Z\over 2},{Y\over 2}\}}(P)$, as illustrated in Figure \ref{fig:Linftyestimatenearfreeboundary}.

We want to estimate $|h|$ in $\Omega_c\cap Q$, and we will use equation of $h$:
\begin{align*}
a_{ij}h_{ij}+b_1h_\xi+b_2h_\eta+b_3h=0, \text{ in } \Omega_c,
\end{align*}
where, $a_{ij},b_i's $ have same expressions as those in Section \ref{gradientestimatenearcornerofwedge}, but here, when it's away from corner of wedge, we have better regularity, $\phi\in C^2(Q\cap \Omega_c)$ ( Proposition \ref{prop:C2chiestimateawayfromsoniccircleandcornerofwedge}, Section \ref{C2chiestimateawayfromsoniccircleandcornerofwedge}), so:
\[
\lambda\leq(a_{ij})\leq\Lambda, \ \ |b_i|\leq C_p.
\]
On $\partial_F(\Omega\cup\tilde\Omega),\ \ |h|=|\phi_1-\widetilde{\phi_1}|\leq2u_1\delta$, and in $(\Omega\setminus\tilde\Omega)\cap Q$, $\partial_\xi h\leq {-u_1 \over 2}<0$(if $\sigma$ small enough), so
\begin{align}h\leq C_p\delta,\text{ on }(\partial_F\tilde\Omega)\cap\overline\Omega.\label{dirichletboundofhonfreeboundarybyderivativepositive}
\end{align}
On $(\partial_F\Omega)\backslash\tilde\Omega, RH(\rho_1,\varphi_1,\rho,\varphi)=0,$ which is the RH condition. A reflection of this gives
 		      \[    \text{  on\ } (\partial_F\tilde\Omega)\backslash\Omega,  \ \ \      RH(\rho_1,\widetilde{ \varphi_1},\tilde\rho,\tilde\varphi)=0,\]
 			so with (\ref{lowerboundofRH}), we get
\[\text{on }(\partial_F\Omega)\cap\tilde\Omega ,          \ \ \     RH(\rho_1,\widetilde{ \varphi_1},\tilde\rho,\tilde\varphi)\geq0.\]
                       Then we change $\varphi_1$ to $\widetilde{\varphi_1}$, and get
\begin{equation}  \label{RHboundofhonfreeboundary}
RH(\rho_1, \varphi_1,\tilde\rho,\tilde\varphi)>-C_p\delta.
\end{equation}

Taking difference of (\ref{RHboundofhonfreeboundary}) and RH condition gives:
\begin{align}
C_p\delta\geq& RH(\rho_1,\varphi_1,\rho,\varphi)-RH(\rho_1,\varphi_1,\tilde\rho,\tilde\varphi)\nonumber\\
                          =&\int_0^1{d\over dt}RH(\rho_1,\varphi_1,\rho^{(t)}, \varphi^{(t)})\nonumber\\
                         \triangleq&\lambda h+D_\vbeta h,\label{Robinconditonofhonfreeboundary}
\end{align}
where,
\begin{align}
\lambda\geq\frac{\overline \rho_2^{2-\gamma} Z u_1}{\gamma-1}-C_p(\sigma+\delta),
\end{align}
and $\vbeta=(\beta_1,\beta_2)$, with
\begin{align}
\beta_1\leq{1\over 2}\overline\rho_2u_1(\frac{ Z^2}{\bar c_2^2}-1)-C_p(\sigma+\delta)\leq \frac{-1}{C_p}\\
|\beta_2|\leq C_p .
\end{align}
Then with Lemma \ref{lem:BoundaryKrylovSafonov}, there exists $r_1, C_p$, s.t. $h\leq C_p \delta$ in $B_{r_1}(P)$, which implies,
\begin{equation}|h|\leq C_p \delta\ \text{,\ in }B_{r_1}(P),
\end{equation}
since $h$ is symmetric function of $\eta$.
\begin{figure}[h] 
\begin{minipage}{0.5\linewidth}
       \centering
    \includegraphics[height=7cm]{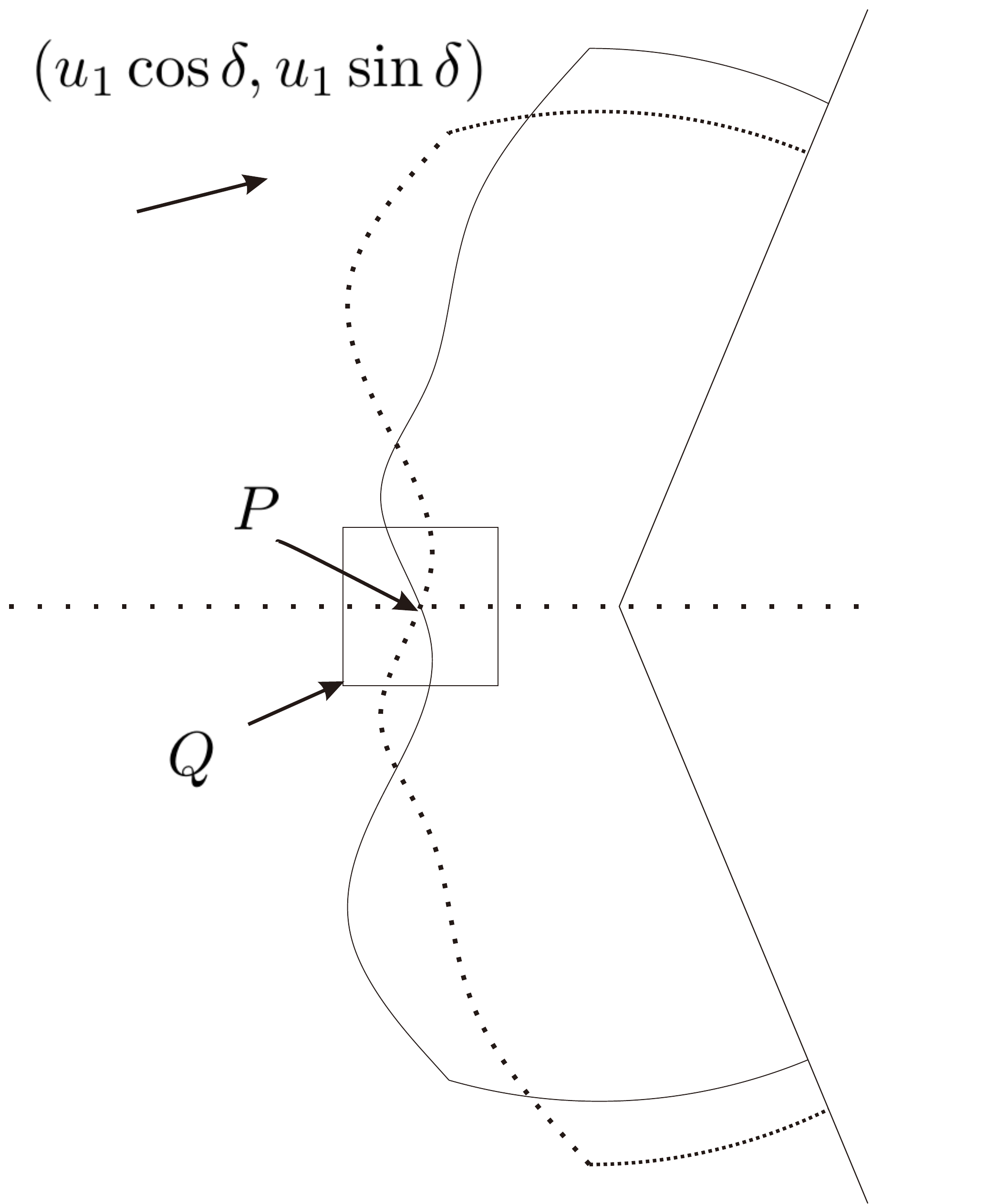}
\caption{$L^\infty$ Estimate near Free Boundary}
\label{fig:Linftyestimatenearfreeboundary}
\end{minipage}
\begin{minipage}{0.5\linewidth} 
       \centering
    \includegraphics[height=7cm]{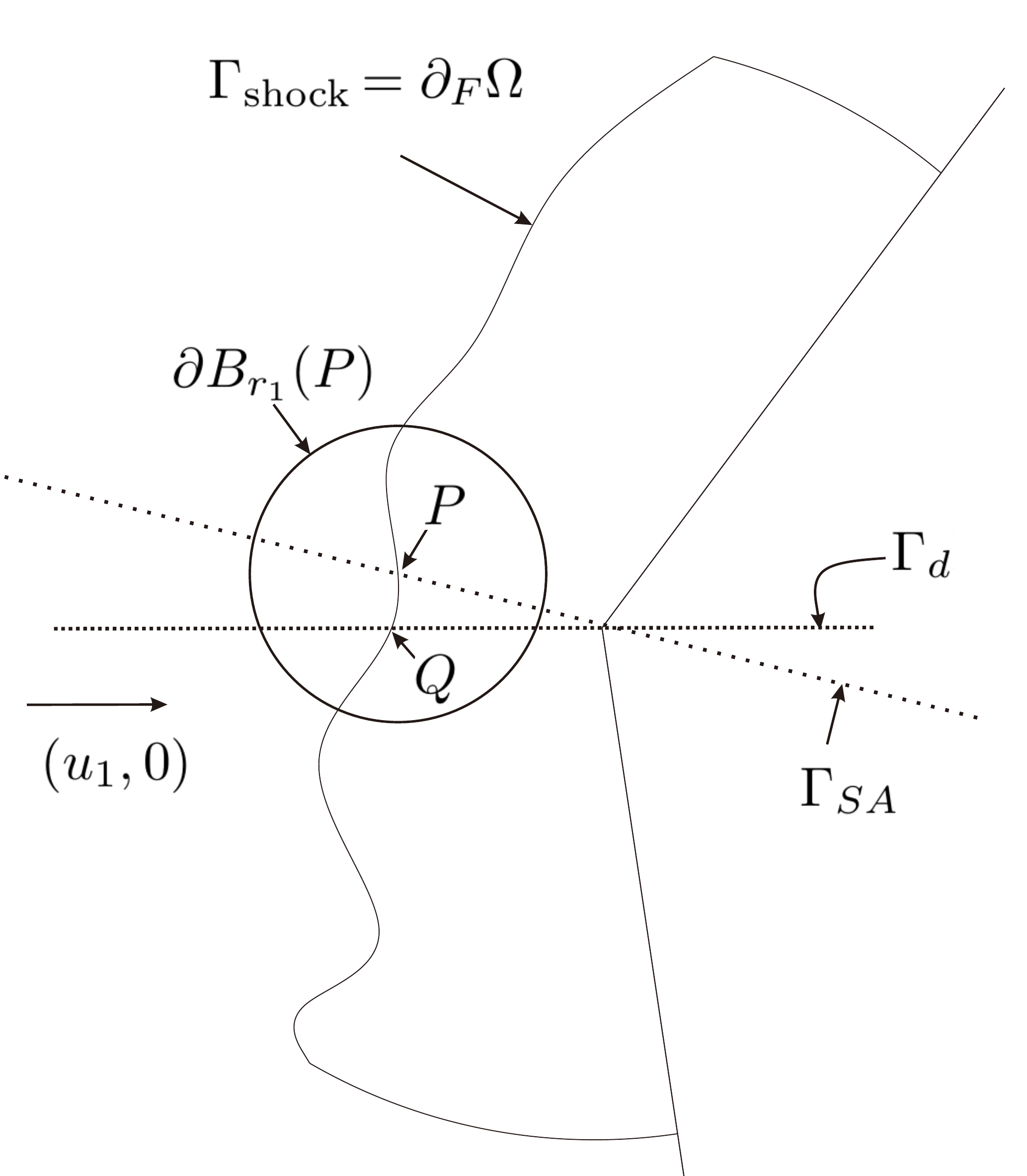}
\caption{Gradient Estimate near Free Boundary}
\label{fig:Gradientestimatenearfreeboundary}
\end{minipage}
\end{figure}
\subsubsection{Gradient Estimate near Free Boundary on $\G_{SA}$}\label{gradientestimatenearfreeboundary}

In this section, with result of last section,
\[|\phi-\tphi|\leq C_p\delta,\text{  in $B_{r_1}(P)$},\]
 we can prove on $B_{r_3}(P)\cap\G_{SA}\, |\phi_\eta|\leq C_p\delta$, where $r_3$ is some constant depends on physical quantities only.

We define $\G_d $ as the line parallels to the direction of upcoming flow, and passes zero. And denote $\G_d \cap \gammashock= Q$.
For any point $p$, reflection of $p$ across $\G_d$ is denoted by $\overline p$, and define $\overline\phi(p)=\phi(\overline p)$.

Since angle between $\G_d$ and $\gammashock$ is $\delta$, we know when $ Z \delta<{r_1 \over 8}$, $|\phi-\overline\phi|\leq C_p \delta$ in $B_{r_1\over 2}(Q)$.

Now we rotate coordinate (with $\MC$ as center) s.t. $\G_d$ is $\xi$-axis, in new coordinate the form of quasi-linear potential flow equation does not change and direction of up-stream flow becomes $(u_1,0)$.

Then we map $B_{r_1 \over 2}(Q)\cap\Omega$ to part of half plane with :
\[\MH(\xi,\eta)=(x,y)=(\phi_1-\phi,\eta)=(u_1\xi +C_Q-\phi,\eta) \ (\text{here we choose } C_Q=\phi_1(Q)-u_1\xi_Q, \text{so } \MH(Q)=(0,0)),
\]
 when $\sigma$ small enough, i.e. $\nabla \phi $ small enough, map above is well defined and is a $C^2$ map. Then we can choose $r_2$ s.t. half of $B_{r_2}(0)$ is contained in $\MH( B_{r_1\over 2}(Q)\cap \Omega)$ .
\begin{figure}[h]
\includegraphics[width=15cm]{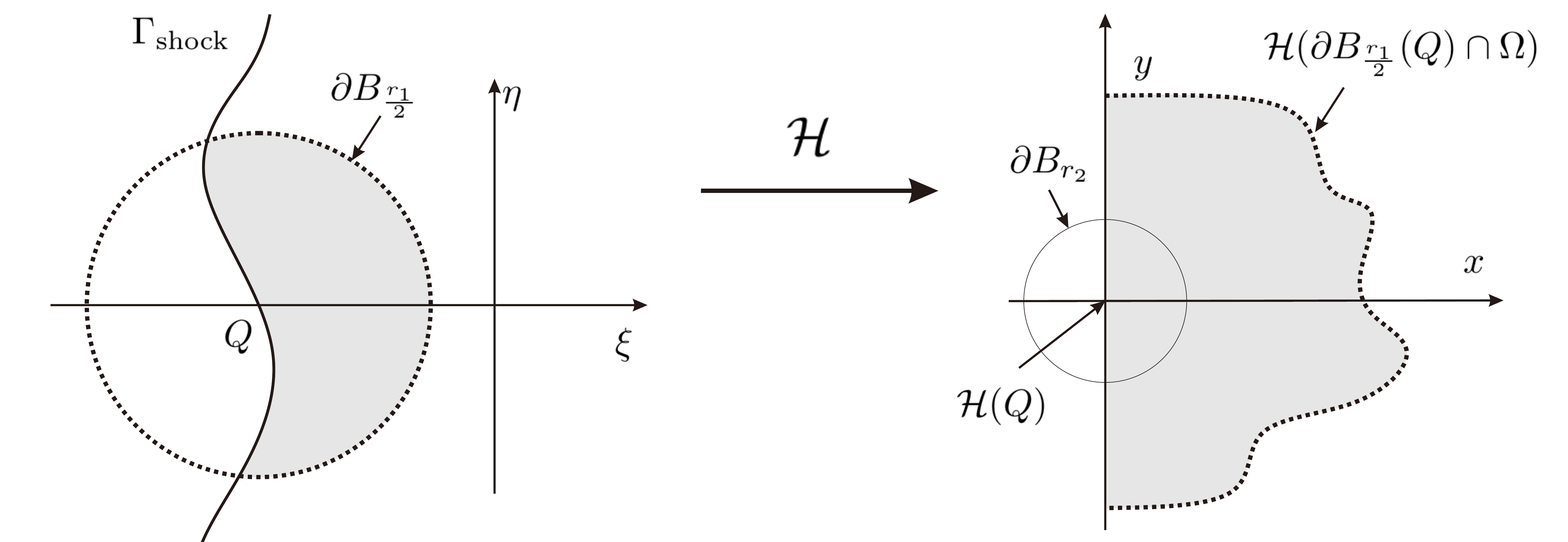}
\caption{Hodograph Map}
\end{figure}

Denote $\MH_*(\phi)$ by $g$, so $g(\MH(\xi,\eta))=\phi(\xi,\eta)$, and $\bar g(x,y)=g(x,-y)$, $\overline{(x,y)}=(x,-y)$, for $(x,y)\in B_{r_2}\cap\{x>0\}$, $\exists\ \xi_1, \xi_2$ 
\begin{align*}
   &|g(x,y)-g(x,-y)|=|g(\MH(\xi_1,y))-g(\MH(\xi_2,-y))|\\
=&|\phi(\xi_1,y)-\phi(\xi_2,y)|+|\phi(\xi_2,y)-\phi(\xi_2,-y)|\\
=&|\phi(\xi_1,y)-\phi(\xi_2,y)|+C_p\delta,
\end{align*}
 we have $u_1\xi_1+C_Q-\phi(\xi_1,y)=u_1\xi_2 +C_Q-\phi(\xi_2,y)$ so:
\[|u_1(\xi_1-\xi_2)|=|\phi(\xi_1, y)-\phi(\xi_2,y)|\leq|\nabla \phi||\xi_1-\xi_2|+C_p\delta
\]
\[\Rightarrow|\xi_1-\xi_2|\leq C_p\delta \ \ (\text{given  } |\nabla\phi|\leq{ u_1\over 2})\]
 then
 \begin{equation}|g-\overline g|\leq C_p\delta.\label{symmetryafterhodographtransform}
\end{equation}
By computation we get:
\[
\begin{array}{cc}
g_x=\frac{\phi_\xi}{u_1-\phi_\xi}             &   g_{xx}=\frac{u_1}{(u_1-\phi_\xi)^3}\phi_{\xi\xi}\\
g_y=\frac{u_1\phi_\eta}{u_1-\phi_\xi}    &  g_{xy}=\frac{u_1\phi_\eta}{(u_1-\phi_\xi)^3}\phi_{\xi\xi}+\frac{u_1}{(u_1-\phi_\xi)^2}\phi_{\xi\eta}\\
                                                                         &g_{yy} =\frac{u_1\phi^2_\eta}{(u_1-\phi_\xi)^3}\phi_{\xi\xi}+\frac{2u_1\phi_\eta}			   {(u_1-\phi_\xi)^2}\phi_{\xi\eta}+\frac{u_1}{u_1-\phi_\xi}\phi_{\eta\eta}
\end{array}
\]
and,
\[
\begin{array}{cc}
\phi_\xi=\frac{u_1g_x}{1+g_x}  &      \phi_{\xi\xi}=\frac{u_1^2}{(1+g_x)^3}g_{xx}\\
\phi_\eta=\frac{g_y}{1+g_x}      &      \phi_{\xi \eta}=\frac{-u_1g_y}{(1+g_x)^3}g_{xx}+\frac{u_1}{(1+g_x)^2}g_{xy}\\
                                                        &       \phi_{\eta\eta}=\frac{g_y^2}{(1+g_x)^3}g_{xx}-\frac{2g_y}{(1+g_x)^2}g_{xy}+\frac{1}{1+g_x}g_{yy}
\end{array}
\]

 \[\frac{1}{C_p}|p_1-p_2|\leq|\MH(p_1)-\MH(p_2)|\leq C_p| p_1-p_2|,\] so, with Proposition \ref{prop:C2chiestimateawayfromsoniccircleandcornerofwedge}, $g $ is still a $C^{2,\chi}$ function on $(x,y)$-plane.

Plug $\phi_\xi,\ \phi_\eta,\ \phi_{\xi\xi},\ \phi_{\xi\eta},\ \phi_{\eta\eta}$ into quasi-linear equation of $\phi$ and $RH$ condition we get, in $B_{r_2}\cap\{x> 0\}$ $g$ satisfies a quasi-linear equation which is elliptic (when equation for $\phi$ is elliptic), and on $B_{r_2}\cap \{x=0\}$ $g$ satisfies a nonlinear boundary condition.

Then we consider $k=g-\overline g$ as in section \ref{gradientestimatenearcornerofwedge}, $k$ satisfies a linear equation, with coefficients depend on $D^2 g,\ \nabla g,\ g$.
 
So, equation of $k$:
    \[a_{ij}k_{ij}+b_ik_i+ck=0\] has $C^\chi$ coefficients,
and on $B_{r_2}\cap\{x=0\}$, there exists $\vbeta$, s.t. $k$ satisfies :
   \[k-D_{\vbeta} k=0,\]
where, $\vbeta $ is oblique(giver $\sigma$ small enough) and ${\vbeta}$ a $ C^{1,\chi}$ function of $y$. Then by Schauder boundary estimate and (\ref{symmetryafterhodographtransform}) \[ |\nabla k|\leq \sup_{B_{r_2}}C_p|k|\leq C_p\delta\ \ \ \text{in }B_{r_2\over 2}(0).\]
Transform back to $(\xi,\eta)$-plane, with expression of gradient $\phi$, we have :
\[\ \ \ |\phi_\eta|\leq C_p\delta,\ \text{on  } \MH^{-1}(\{y=0\}\cap B_{r_2\over 2}(Q)).\]
When $\delta$ small enough there exists $ r_3$, s.t. $\MH^{-1}(B_{r_2 \over 2})\supset B_{r_3}(Q)$. And because $|D^2\phi|\leq C$, when rotate coordinate back, we get:
\[|\phi_\eta|\leq C_p\delta   \    \text{ in }B_{r_3\over 2}(P).\]

\subsection{Summary}
From (\ref{integralsymmetricestimateinequality}), we have derived $L^\infty$ and $\eta-$derivative estimate near conner of wedge and where symmetry axis intersects shock. Actually, when it's away from corner of wedge and shock, the estimate is easier. And as a summary, we will have the following estimate:

\begin{proposition}\label{prop:symmetricestimatesummary}
For $\phi$, a regular solution in the sense of Definition \ref{def:definitionofregularsolution}, there exists $\epsilon_{S}$ and $C_p$, which only depend on physical quantities, i.e. does not depend on $\sigma,\delta$, s.t. for $(\xi,\eta)\in\Omega\cap\tilde\Omega$ and $|\eta|<\epsilon_{S}$,
\[|\phi(\xi,\eta)-\phi(\xi,-\eta)|\leq C_p\delta.\]
And for any point $p$ on symmetry axis of wedge,
there exists, $0<\alpha<1$ and $C_p$ also do not depend on $\sigma, \delta$, s.t.
\[|\partial_\eta\phi|(p)\leq C_p d^{\alpha-1}(p,\MC),\]
where, $d(p,\MC)$ denotes distance of $p$ to corner of wedge.
\end{proposition}


\section{Anti-Symmetric Estimates}\label{antisymmetricestimate}
\subsection{Symmetrization of Upcoming Flow}
 In this section we symmetrize upcoming flow (and so free boundary condition), w.r.t. $\eta$. With this symmetrization, we will be able to do more delicate estimate. 

\begin{figure}
       \centering
    \includegraphics[height=6cm]{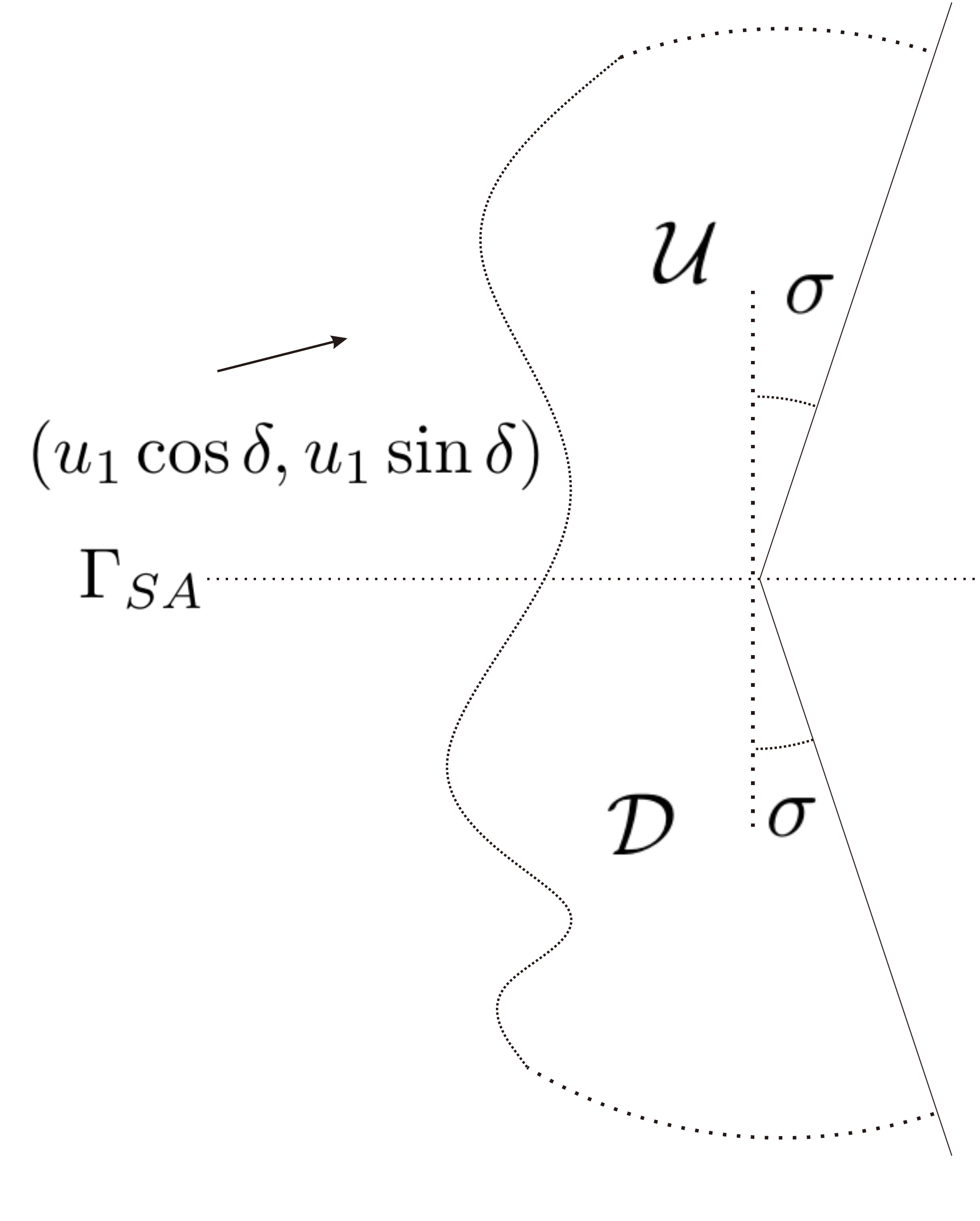}
\end{figure}
\subsubsection{Above Symmetry Axis of Wedge}\label{5.1.1}
In this section denote $\MU=\Omega\cap\{\eta>0\}$, in $\MU$ define:
   \begin{align*}
&\psi=\phi-(u_1\sin\delta\tan\sigma)\xi-(u_1\sin\delta)\eta,\\
&\psi_1=\phi_1-(u_1\sin\delta\tan\sigma)\xi-(u_1\sin\delta)\eta,\\
&\psi_2=\phi_2^+-(u_1\sin\delta\tan\sigma)\xi-(u_1\sin\delta)\eta.\\
\end{align*}
Note that now $\nabla\psi_1=(u_1\cos \delta-u_1\sin\delta\tan \sigma,0),$ so it's parallel to symmetry axis of wedge, and $\psi$, $\psi_2$ still satisfy Neumann boundary condition on $\gammawedge$.

We consider a new coordinate:
\[
x=\xi-u_1\sin\delta\tan\sigma,\ \ y=\eta-u_1\sin\delta.
\]
The physics meaning of this transformation is that, now we observe flow in a reference system moving with velocity $(u_1\sin\delta\tan\sigma,u_1\sin\delta)$, comparing to original coordinate system.

In new coordinate:\begin{align*}
\rho^{\gamma-1}=&\rho_0^{\gamma-1}-\psi-u_1\sdtc(x+u_1\sdtc)-u_1\sin\delta(y+u_1\sin\delta)\\
                               &+(\psi_x+u_1\sdtc)(x+u_1\sdtc) +(\psi_y+u_1\sin\delta)(y+u_1\sin\delta)\\
                               &-{1\over 2}(\psi_x+u_1\sdtc)^2-{1\over 2}(\psi_y+u_1\sin\delta)^2\\
                            =&\rho_0^{\gamma-1}-\frac{u_1^2\sin^2\delta}{2\cos^2\sigma}-\psi+\psi_x x+\psi_y y -{1\over 2}|\nabla\psi|^2\\
                            =&K-\psi+\psi_x x+\psi_y y -{|\nabla\psi|^2\over 2} \ \left(K\triangleq\rho_0^{\gamma-1}-\frac{u_1^2\sin^2\delta}{2\cos^2\sigma}\right),
\end{align*}
\begin{align*}
\psi_1=&K-\rho_1^{\gamma-1}+u_1(\cos\delta-\sdtc)x-{u_1^2(\cos\delta-\sdtc)^2\over 2},\\
\psi_2=&K-{\rho_2^+}^{\gamma-1}+(u_2^+-u_1\sdtc)x+(v_2^+-u_1\sin\delta)y\\
                        &  -\frac{(u_2^+-u_1\sdtc)^2+(v_2^+-u_1\sin\delta)^2}{2}.
\end{align*}

\subsubsection{Below Symmetry Axis of Wedge}\label{5.1.2}
 Then we consider the area below $\Gamma_{SA}$, we denote $\Omega\cap \{\eta<0\}$ by $\MD$, denote its reflection into $\Gamma_{SA}$ by $\tMD$ as in section \ref{symmetricestimate}. And also as in section 4 we define $\tphi(\xi,\eta)=\phi(\xi,-\eta)$. 
In $\tMD$, define:
\begin{align*}
&\tpsi=\tphi+u_1\sin\delta\tan\sigma\xi+u_1\sin\delta\eta,\\
&\tpsi_1=\widetilde{\phi_1}+u_1\sin\delta\tan\sigma\xi+u_1\sin\delta\eta,\\
&\tpsi_2=\widetilde{\phi_2^-}+u_1\sin\delta\tan\sigma\xi+u_1\sin\delta\eta.\\
\end{align*}
As in section \ref{5.1.1}, now $\nabla\tpsi_1=(u_1\cos \delta+u_1\sin\delta\tan \sigma,0)$, i.e. it's parallel to symmetry axis of wedge.
We consider a new coordinate:
\[
 x=\xi+u_1\sin\delta\tan\sigma,\ \  y=\eta+u_1\sin\delta.
\]
The physics meaning of this transformation is that, now we observe flow in a reference system moving with velocity
 $(-u_1\sin\delta\tan\sigma,-u_1\sin\delta)$, comparing to original coordinate system.

Also under new coordinate:
\begin{align*}
\hrho^{\gamma-1} =&K-\tpsi+\tpsi_x x+\tpsi_y y -{|\nabla\tpsi|^2\over 2}\ \ \text{(as in section \ref{5.1.1}, $K=\rho_0^{\gamma-1}-\frac{u^2_1\sin^2\delta}{2\cos^2\sigma}$)}\\
\tpsi_1=&K-\rho_1^{\gamma-1}+u_1(\cos\delta+\sdtc)x-{u_1^2(\cos\delta+\sdtc)^2\over 2}\\
\tpsi_2=&K-{\rho_2^-}^{\gamma-1}+(u_2^-+u_1\sdtc)x+(-v_2^-+u_1\sin\delta)y\\
                        &  -\frac{(u_2^-+u_1\sdtc)^2+(v_2^--u_1\sin\delta)^2}{2}
\end{align*}

\subsubsection{Uniform Coordinate}
Then we let $(x,y)$ coordinate of section \ref{5.1.1} and that of section \ref{5.1.2} coincide. The position of $\MU, \tMD$ will be as illustrated as in Figure \ref{fig:delicateintegralcomparison}, according to the computation in section \ref{strongercomparisonbasedonderivative}.

 And we compare $\tpsi_1$ with $\psi_1$ and $\tpsi_2$ with $\psi_2$ :
\[
|\tpsi_1-\psi_1|=|-2u_1\sdtc x+2u_1^2\cos\delta\sdtc|\leq C_p\sigma\delta
\]
\[
|\nabla\tpsi_1-\nabla\psi_1|\leq C_p\sigma\delta
\]
\begin{equation}
\psi_2-\tpsi_2\geq\frac{u_1 Z}{2X}\delta y-C_p\sigma\delta>0,\label{crucialinequalityofpotential}
\end{equation}
in $\left\{y>{\sqrt{\overline c_2^2- Z^2}\over 2}\right\}, $providing $\sigma$ small enough.

\subsection{Integral Comparison}

\begin{figure}
     \centering
   \includegraphics[height=7cm]{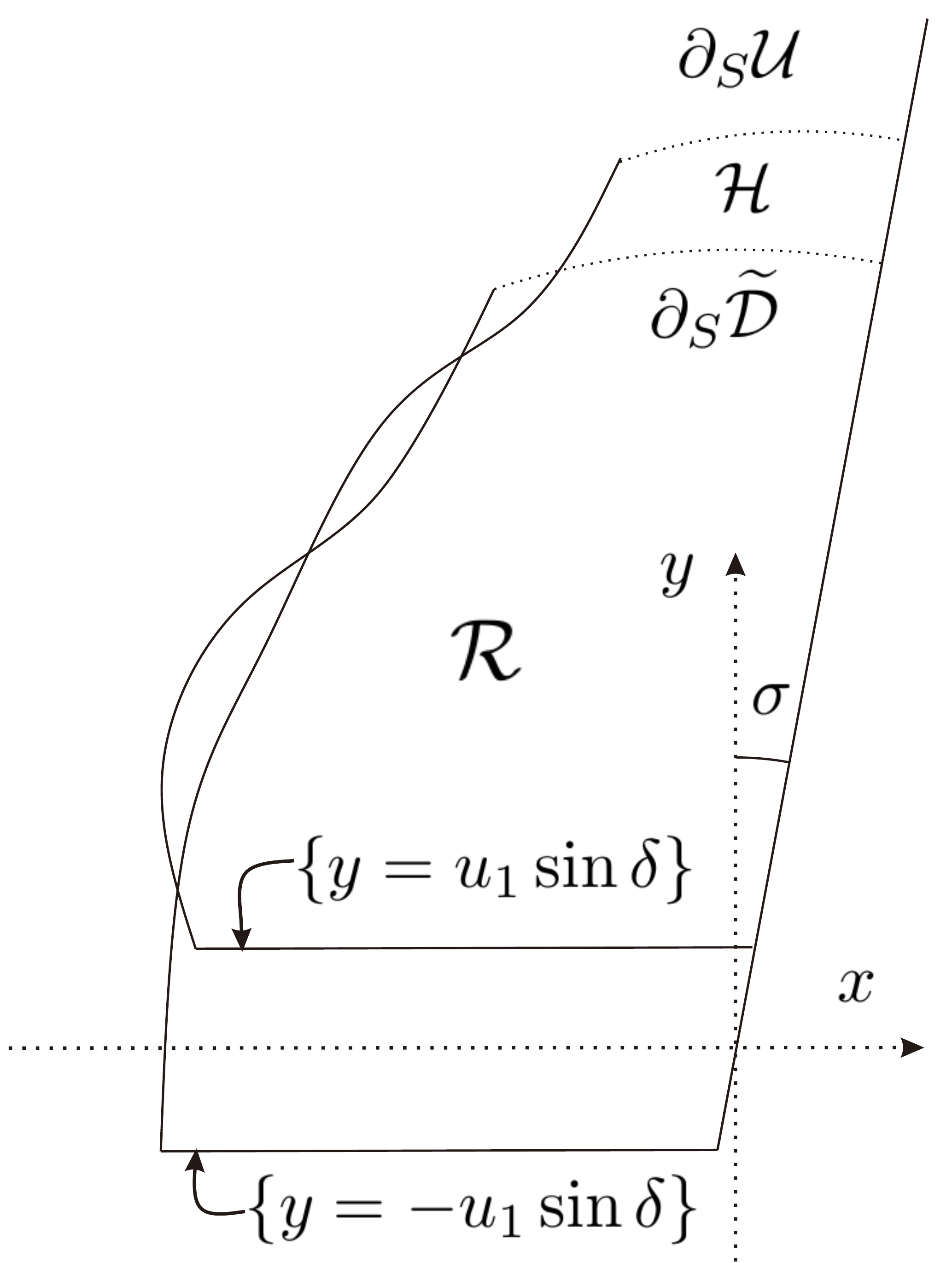}\caption{Integral Comparison}
\label{fig:delicateintegralcomparison}
\end{figure}

Define:
\[ \MR=\MU\cap\tMD .\]
As in section \ref{integralsymmetricestimate}, we can separate $\partial \MR(\partial \MU,\partial \tMD)$ into free boundary part, sonic circle part and symmetry axis part and denote them as $\partial_F\MR(\partial_F \MU,\partial_F \tMD),\partial_S \MR(\partial_S \MU,\partial_S \tMD)$ and $\partial_A \MR(\partial_A \MU,\partial_A \tMD)$ respectively. We will use $\MH$ to denote the part of $\MU\setminus\tMD$ above $\partial_S\tMD$.

In new coordinate, pseudo-potential function and densities are
\[ \ f=\psi-{x^2+y^2\over 2},\ \ \tf=\tpsi-{x^2+y^2\over 2},\]
\[ \rho=(K-f-{|\nabla f|^2\over 2})^{1\over \gamma-1},\ \  \hrho=(K-\tf-{|\nabla \tf|^2\over 2})^{1\over \gamma-1}.
\]
Define:
\[k=\left\{\begin{array}{cc}
                 \tf-f=\tpsi-\psi & \text{in }\MR\\
                \tpsi_1-\psi        &\text{in }((\MU\bs\MR)\bs\MH) \cap\{y>u_1\sin \delta\}\\
                \tpsi-\psi_1        &\text{in }\tMD\setminus\MR
                 \end{array}
       \right.   
\]
\[k=\max\{k,0\},\ \ \ 
w=(k^+)^2.
\]
Now consider the following integral:
\begin{align*}
   &\int_{\MR}\hrho\nabla\tf\cdot\nabla w-\rho\nabla f\cdot\nabla w-2\hrho w+2\rho w\\
=&\int_{\MR}\nabla\cdot[\hrho\nabla\tf w-\rho\nabla f w]\\
=&\int_{\pMR}[\hrho\nabla\tf w-\rho\nabla f w]\cdot\vn\\
\triangleq&\int_{\pFMR}+\int_{\pSMR}+\int_{\pAMR}.
\end{align*}
Here, we don't need to consider integral on $\gammawedge$, since $f_\vn=0$ on $\gammawedge$. And $\int_{\pSMR}=0$, because on $\pSMR$ $\psi\geq\tpsi_2$, $k=0$.

So we only need to consider integration on $\partial_F\MR$ and $\partial_A\MR$.
\subsubsection{Integral on Free Boundary}\label{antisymmetricintegralonfreeboundary}
In this subsubsection, all analysis is in area between $\partial_S\tMD$ and $\{y=u_1\sin\delta\}$, so for convenience of notation, any set $A$ should be understood as  $(A\cap {\{y\geq u_1\sin\delta\}})\setminus\MH$.
\begin{align*}
\int_{\pFMR}=&\int_{(\pFMU)\cap\tMD}\vn \cdot[\hrho\nabla\tf w-\rho\nabla f w]+\int_{(\pFtMD)\cap\MU}\vn\cdot[\hrho\nabla\tf w-\rho\nabla f w]\\
&+\int_{(\pFMU)\cap(\pFtMD)}[\hrho\nabla\tf w-\rho\nabla f w]\cdot\vn.\end{align*}
With the same argument in section \ref{integralsymmetricestimate} from  (\ref{symmetricintegralonshockmeetsshock}) to (\ref{boundaryestimateofvectoronshockmeetsshock}), we know 
\[\int_{(\pFMU)\cap(\pFtMD)}[\hrho\nabla\tf w-\rho\nabla f w]\cdot\vn=\int_{(\pFMU)\cap(\pFtMD)}[\hrhoone\nabla\tf_1 w-\rho_1\nabla f_1 w]\cdot\vn\leq C_p\delta^3\sigma.\]

\begin{remark}
Note, here, we can control integral by $C_p\delta^3\sigma$, while in section \ref{integralsymmetricestimate}, our control can only be as small as $C_p\delta^3$. The reason is that, in this section, we symmetrized upcoming flow, s.t. 
\[|f_1-\tf_1|+|\nabla f_1-\nabla\tf_1|\leq C_p\sigma\delta.\]
\end{remark}
So,
\begin{align*}
\int_{\pFMR}   =&\int_{(\pFMU)\cap\tMD}\vn \cdot[\hrho\nabla\tf w-\rho_1\nabla f_1 w]+\int_{(\pFtMD)\cap\MU}\vn\cdot[\hrhoone\nabla{\tf}_1 w-\rho\nabla f w]+C_p\delta^3\sigma\\
                       =&\int_{\tMD\setminus\MU}-\nabla\cdot[\hrho\nabla\tf w-\rho_1\nabla f_1 w]-\int_{\MU\setminus\tMD}\nabla\cdot[\rho_1\nabla\tf_1 w-\rho\nabla f w]\\
                         &+\int_{(\pFtMD)\setminus\MU}\underbrace{\vn\cdot[\hrho\nabla \tf -\rho_1\nabla f_1 ]}_{\leq C_p \sigma\delta}w
                          + \int_{(\pFMU)\setminus\tMD}\underbrace{\vn\cdot[\rho_1\nabla\tf_1-\rho\nabla f ]}_{\leq C_p\sigma\delta}w\\
                        &+\underbrace{\int_{(\pStMD)\setminus\MU}\vn\cdot[\hrho\nabla\tf w-\rho_1\nabla f_1 w]}
                             _{\leq C_p \sigma^2\delta^3(k^+\leq C_p\sigma\delta, \ |(\pStMD)\setminus\MU|\leq C_p\delta)}
                              +\underbrace{\int_{I_\triangle}\vn\cdot[\rho_1\nabla\tf_1 w-\rho \nabla f w]}
                                  _{\leq C_p\sigma^3\delta^3(k^+\leq C_p\sigma\delta,|I_\triangle|\leq C_p\sigma\delta)}+C_p\delta^3\sigma
\end{align*}
 In above $I_\triangle$ is an interval (may degenerate to be empty set), 
\[I_\triangle=\left\{\begin{array}{cc}
                    (\{y=u_1\sin\delta\}\cap\MU)\bs \pAtMD & \text{if }(\{y=u_1\sin\delta\}\cap\overline\MU)\supset \pAtMD\\
                     \pAtMD\bs(\{y=u_1\sin\delta\}\cap\MU) & \text{if }(\{y=u_1\sin\delta\}\cap\overline\MU)\subset \pAtMD
                    \end{array}         
           \right..
\]
In each case length of $I_{\triangle}\leq C_p\sigma\delta$, since ${\left|1\over \text{slope of } \gammashock\right|}\leq C_p\sigma$ (\ref{estimateofslopeofshock}). So,
\begin{align*}
          \int_{\pFMR}        \leq&\int_{\tMD\setminus\MU}2k^+[(\hrho-\rho_1)k^+-(\hrho\nabla\tf-\rho_1\nabla f_1)\cdot\nabla k]\\
                         &+\int_{\MU\setminus\tMD}2k^+[(\rho_1-\rho)k^+-(\rho_1\nabla\tf_1-\rho\nabla f)\cdot\nabla k]
                                 +C_p\sigma\delta^3
\end{align*}

Then we consider integrals in $\tMD\setminus\MU$ and $\MU\setminus\tMD$ separately.

1)In $\tMD\setminus\MU$, \[k=\tpsi-\psi_1, \  k_\xi<-{u_1\over 2}<0\  (\text{providing } \sigma \text{ small enough}),\]
   and on $\partial(\tMD\cup\MU)$ $|k|\leq C_p \sigma\delta$, so $k^+\neq 0$ only in a $C_p\sigma\delta$ neighborhood of $\partial(\tMD\cup \MU)$, and in this neighborhood $k^+\leq C_p\sigma\delta$, which implies
\[
\int_{\tMD\setminus\MU}2(k^+)^2(\hrho-\rho_1)\leq C_p\sigma^3\delta^3.
\]
And 
\[
-(\hrho\nabla\tf-\rho_1\nabla f_1)\cdot\nabla k\leq-(\hrho\nabla\tf-\hrhoone\nabla\tf_1)\cdot(\nabla\tf-\nabla\tf_1)+C_p\sigma\delta\leq C_p\sigma\delta
\]so integral on $\tMD\setminus\MU$ is controlled by $C_p\sigma^3\delta^3$.

2)In $\MU\setminus\tMD$ 
\begin{align*}
       &-(\rho_1\nabla\tf_1-\rho\nabla f)\cdot(\nabla\tf_1-\nabla f)\\
\leq&-(\rho_1\nabla f_1-\rho\nabla f)\cdot(\nabla f_1-\nabla f)+C_p\sigma\delta\\
\leq&C_p\sigma\delta-(f_1- f)(\rho-\rho_1)
\end{align*}so $-(\rho_1\nabla\tf_1-\rho\nabla f)\cdot(\nabla \tf_1-\nabla f)>0$ only in a $C_p\sigma\delta$ neighborhood of $(\pFMU)\setminus\tMD$, and in this neighborhood it's controlled by $C_p\sigma\delta$, also $k\leq C_p\sigma\delta$, which implies
\[
\ \ \int_{\MU\setminus\tMD}-2k^+(\rho_1\nabla\tf_1-\rho\nabla f)\cdot\nabla k\leq C_p\sigma^3\delta^3.
\]
So we get
\[\int_{\pFMR}\leq C_p\sigma\delta^3.\]
\subsubsection{Integral on Symmetry axis}
To estimate $\int_{\pAMR}$ we need to estimate $|\nabla(\psi-\tpsi)|$ along $\pAMR=\{y=\sin\delta u_1\}\cap\partial\MR.$

First, denote length of $\Omega\cap\Gamma_{SA}$ by $T$ ,so:
\[
\{(x,y)=(t-u_1\sdtc,-u_1\sin\delta)\mid t\in(-T,0)\}=\pAMU,
\]
and\[
\{(x,y)=(t+u_1\sdtc,+u_1\sin\delta)\mid t\in(-T,0)\}=\pAtMD.
\]
By definition of $\tpsi$, for $t\in (-T,0)$
\begin{align*}
    \tpsi_x(t+u_1\sdtc,u_1\sin\delta)
= \psi_x(t-u_1\sdtc,-u_1\sin\delta)\ 
\end{align*}
So,
\begin{align*}
  |  \tpsi_x(t+u_1\sdtc,u_1\sin\delta)-\psi_x(t+u_1\sdtc,u_1\sin\delta)|\leq C_p\delta|t|^{\alpha-1},
\end{align*}
because $|D^2\psi|\leq C_p d_{\MC}^{\alpha-1},$ where, $ d_{\MC}$ is distance to corner of wedge.
And also by definition,
\[
\psi_y(t-u_1\sdtc,-u_1\sin\delta)=-\tpsi_y(t+u_1\sdtc,u_1\sin\delta),
\]
by estimate in section \ref{gradientestimatenearcornerofwedge} and \ref{gradientestimatenearfreeboundary}, we have:
\[
|\psi_y(t-u_1\sdtc,-u_1\sin\delta)|\leq C_p\delta|t|^{{\alpha\over 2}-1}\ (-T<t<0)
\]
so,
\[|\psi_y(t+u_1\sdtc,u_1\sin\delta)|\leq C_p\delta|t|^{{\alpha\over 2}-1}+C_p\delta|t|^{\alpha-1}.
\]
Now, on $\pAMR$ we have
\begin{itemize}
\item
$|\nabla \psi-\nabla\tpsi|\leq C_p\delta |t|^{\frac{\alpha}{2}-1}$
                      \item  $ |\rho-\hrho|\leq C_p\delta |t|^{\frac{\alpha}{2}-1}       $     
                                                                        \item     $     |\psi-\tpsi|\leq C_p\delta\sigma$.
\end{itemize}
Plug above into $\int_{\pAMR}$, we get
\begin{equation}
 \int_{\pAMR}[\hrho\nabla\tf w-\rho\nabla f w]\cdot\vn\leq C_p(\sigma\delta)^2\int^T_0\delta|t|^{{\alpha\over 2}-1}dt\leq C_p\sigma^2\delta^3\label{antisymmetricboundaryintegralonsymmetricaxis}.
\end{equation}
Combine this with estimate of section \ref{antisymmetricintegralonfreeboundary} we get
\[\int_{\MR}(\hrho\nabla\tf-\rho\nabla f)\cdot\nabla w-2(\hrho-\rho)w\leq C_p\sigma\delta^3.
\]
Repeat what we did in section \ref{integralsymmetricestimate} from (\ref{connection}) to (\ref{integralsymmetricestimateinequality}), we get
\[\int_{\MR}(k^+)^3\leq C_p\int_{\MR}\hrho\nabla\tf\cdot\nabla w-\rho\nabla f\cdot\nabla w-2\hrho w+2\rho w.\]

\subsection{Contradiction}\label{5.3}
Take $r_M=\min{\{{ \epsilon_S},\epsilon_{A1}\}}$  (where, $\epsilon_S$ is the small constant in Proposition \ref{prop:symmetricestimatesummary}  and $\epsilon_{A1}$ is based on Lemma \ref{lem:singularityestimate} and equation of $k$), with Moser iteration, we have:
\[k^+\leq C_p\sigma^{1\over 3}\delta\text{ on } B_{r_M}\cap\{y\geq u_1\sin\delta\}\]

Now we go back to original $(\xi,\eta)$-coordinate, above implies on $( B_{r_M\over 2})\cap\MU$:
\[
\psi(\xi,-\eta)-\psi(\xi+2u_1\sdtc,\eta+2u_1\sin\delta)\leq C_p\sigma^{1\over 3}\delta,\]
so
\[
\psi(\xi,-\eta)-\psi(\xi,\eta)\leq C_p\sigma^{1\over 3}\delta, \text{ on } \partial B_{r_M\over 2}.
\]
Then by definition of $\psi$,
\[
 \ \ h= \phi(\xi,\eta)-\phi(\xi,-\eta)\geq 2u_1(\sin\delta)\eta-C_p\sigma^{1\over 3}\delta
                                                                                 \geq \delta({\eta \over C_p}-C_p\sigma^{1\over 3}).
\]
So $C_ph/ \delta$ satisfies the condition of Lemma {\ref{lem:singularityestimate}}, then when $\sigma$ small enough, such that
$C_p^2\sigma^{1\over 3}<\epsilon_{A3}$,
we find $\phi \not\in C^{1,\alpha}(\Omega\cup\MC) $,which contradicts with estimate in Proposition \ref{prop:gradientestimateatcornerofwedge}. 
So we proved our main theorem.
\

\appendix
\section{Estimates for Linear Equations with Singular Coefficients}

\begin{figure}
       \centering
    \includegraphics[width=6cm]{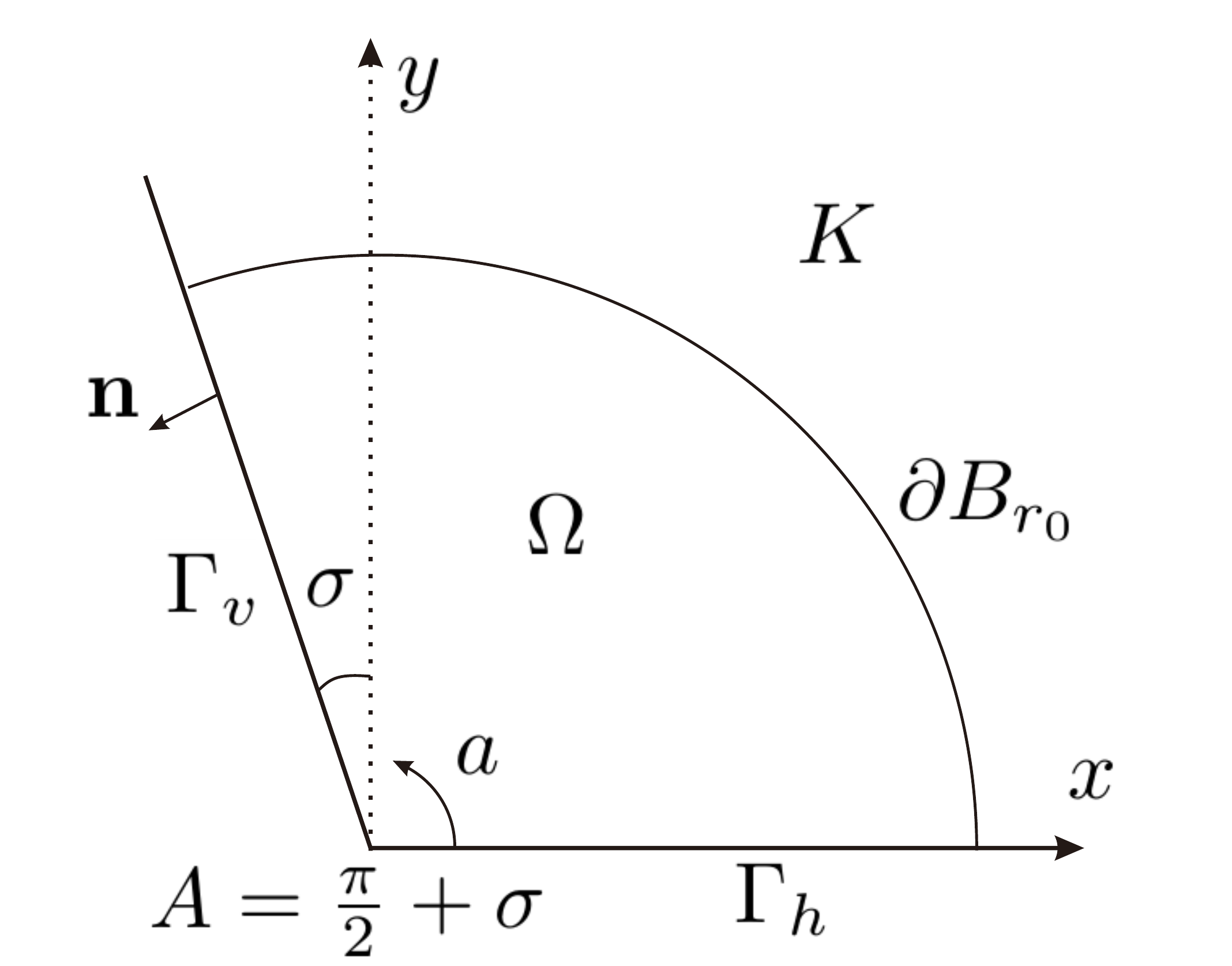}
\caption{Notation of Appendix A}
\label{fig:NotationofAppendixA}
\end{figure}

Lemma \ref{lem:maximumprinciple},\ref{Lemma A2},\ref{lem:singularityestimate} are estimates around the corner of a cone $K$(with angle $\sigma+{\pi \over 2},\ 0<\sigma<{\pi \over 2}$), in the cone $K$ we use coordinate $(x,y)=r(\cos a,\sin a)$, so $K=\{0\leq a\leq \frac{\pi}{2}+\sigma\}$.

In this appendix we define region $\Omega=(K \cap B_{r_0})^\circ$. And 
$(a_{ij}),\ b_k\ (1\leq i,j\leq2,\ 1\leq k\leq 3)$ will be coefficients of some equations in $\Omega$, satisfying:
\begin{eqnarray}
\lambda                \leq&(a_{ij})&\leq \Lambda,\label{cc}\\  
 |a_{ij}-\delta_{ij}| &\leq & C_E r^\alpha ,       \label{cc2}\\ 
|\vec{b}|      &\leq & C_E r^{\alpha-1}      \label{cc3}.
\end{eqnarray}\label{cc}

We denote $\Gamma_h=\{y=0\}\cap \partial K\cap B_{r_0},$ $\ \Gamma_v=\{a=\sigma+\pi/2\}\cap \partial K\cap B_{r_0}$, and $A={\pi \over 2}+\sigma$. $\vn$ is outer normal direction of $\Omega$ on $\Gamma_v$.

\subsection{Maximum Principle with Singular Coefficients}

\label{MaximumPrinciplewithSingularCoefficients}
\begin{lemma}\label{lem:maximumprinciple}

Given $w\in  C^0(\overline\Omega)\cap C^1(\Omega\cup\Gamma_v\setminus(0,0))\cap C^2(\Omega)$,
\begin{align}
&a_{ij}w_{ij}+b_1w_x+b_2w_y+b_3w\geq0\ \ \   &  &\text{ in } \Omega,\\
&w\leq 0  &  &\text{ on }\Gamma_h\cup((\partial B_{r_0})\cap K),\\ 
&D_{\vn}w\leq0\ \ &  &\text{ on }\Gamma_v.
\end{align}
There exists $\epsilon_{A1}$ depends on $C_E,\  \alpha,\ \lambda,$ if $r_0<\epsilon_{A1}(C_E,\alpha,\lambda)$ ,
then in $\Omega$, $w\leq 0.$   
\end{lemma}

\begin{proof}

Like in section 9.1 of \cite{GT}, we define upper contact set $\Gamma^+$ and normal map $\chi$, with more restriction:
\begin{align*}
\Gamma^+=&\{(x,y)\in B_1(0,H)|\exists L: \text{linear function, graph of $L$ upper contacts with}\\ &\text{graph of $w^+$ at $(x,y)$}, 
                        \nabla L\cdot(\cos\sigma,\sin\sigma)\leq0\}\\
\chi(x,y)=&\{\nabla L\mid \text{graph of $L$ upper contacts with graph of $w^+$ at $(x,y)$},\nabla L\cdot(\cos\sigma,\sin\sigma)<0\}
\end{align*}

 It means  $\Gamma^+\cap\Gamma_{v}=\O $.
 We have:\\
\begin{align}
\sup w^+\leq2r_0k\left(\exp^{\frac{4\int_{\Omega} |\vec b|^2+\frac{c^2(w^+)^2}{k^2}}{\lambda^2 \pi}}-1\right)^{1 \over 2},
\end{align}
in above we denote:\[ k=(2\int_{ \Omega} c^2(w^+)^2)^{1 \over 2}.\]
Then if 
\[
2r_0\left(\exp^{\frac{(4\int_{\Omega} |\vec b|^2)+2}{\lambda^2 \pi}}-1\right)^{1 \over 2}(2\int_{\Omega} c^2)^{1 \over 2}<{1 \over 2},
\]
we have $w\leq 0$.

So we can choose:
\[
\epsilon_{A1}=\min\left\{\frac{\sqrt{\alpha}}{8C_E\sqrt{\pi}}\left(\exp^{\frac{4}{\pi\lambda^2}\left(1+\frac{\pi}{\alpha}C_E^2\right)}-1\right)^{-\frac{1}{2}},{1 \over e^4}\right\}.
\]
\end{proof}

\subsection{Gradient Estimates near the Corner}

\begin{lemma}\label{Lemma A2}
There exists $\epsilon_{A2}$ depends on $C_E, \Lambda, \lambda,\alpha$.
Given $r_0< \min\{\epsilon_{A1},\epsilon_{A2}\}$, and $h$ satisfies:
\begin{align}
h\in C^0(\overline \Omega)\cap C^1(\Omega \cup \Gamma_v\setminus(0,0))\cap C^2( \Omega),\\
a_{ij}h_{ij}+b_1h_x+b_2 h_y +b_3h=0, \text{ in }\Omega,\\
h=0,                  \text{ on }\Gamma_h,\\
h_{\vn}=0,\text{ on }  \Gamma_v .
\end{align}

Then there exists an $r_{A2}<r_0$, which depends on $C_E, \Lambda,\lambda,r_0, \alpha$, s.t. on $\{\eta=0\}\cap B_{r_{A2}}$
\begin{equation}
|h_y|\leq C(C_E,\alpha,\lambda,\Lambda,r_0)\cdot r^{\frac{\alpha}{2}-1}|h|_{0;B_{r_0}}
\end{equation}\end{lemma}

\begin{proof}
\begin{figure}[h]
       \centering
    \includegraphics[width=15cm]{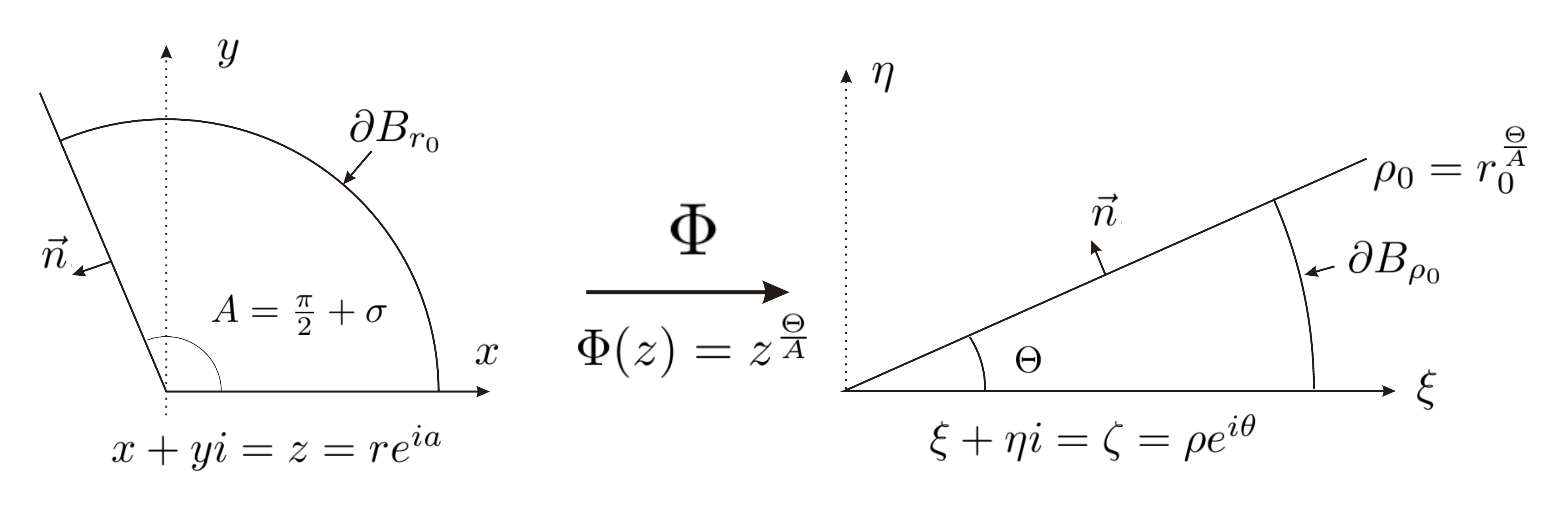}
\caption{Desingularization}
\end{figure}
Conformally map original wedge to a narrow wedge with top angle $\Theta$(to be determined), in new cone, we use coordinate $(\xi,\eta)=\rho(\cos \theta,\sin \theta)$. 

\begin{remark}
It's convenient to consider $a,r,\theta,\rho,a_{ij},b_k$ as functions on both sides, it means, for example on right-hand side $r(\xi,\eta)\triangleq r(\Phi^{-1}(\xi,\eta))$, and on left-hand side $\theta(x,y)\triangleq\theta(\Phi(x,y))$. So  on right-hand side we can still say $|b_i|\leq C_E r^{\alpha-1}$ and we have $r(\xi,\eta)=\rho^{A \over \Theta}(\xi,\eta)$.
\end{remark}

Under new coordinate, on right-hand side (we still denote $\Phi_* h$ by $h$) $h$ satisfies:
\begin{align}
\tilde a_{ij}h_{ij}+\tilde b_1 h_\xi +\tilde b_2h_\eta+\tilde b_3h=0,\text{ in } \Phi_*(\Omega),\\
h_{\vn}=0, \text{ on } \{\theta=\Theta\}\cap\overline{\Phi_*(\Omega)},\\
h=0, \text{ on } \{\eta=0\}\cap\overline{\Phi_*(\Omega)}.
\end{align}
To describe $\tilde a_{ij}$, $\tilde b_k$ we define:
\[
P=\left(\begin{array}{cc}
         \cos[a({\Theta\over A}-1)]&\sin[a({\Theta \over A}-1)]\\
         -\sin[a({\Theta\over A}-1)]&\cos[a({\Theta \over A}-1)]
        \end{array}      
\right)\]
then:
\begin{align}
(\tilde a_{ij})=({\Theta \over A})^2P^T(a_{ij})P\\
\tilde b_1={\Theta \over A}({\Theta\over A}-1)\{\cos[a({\Theta\over A}-2)]a_{11}-2\sin[a({\Theta\over A}-2)]a_{12}-\cos[a({\Theta\over A}-2)]a_{22}\}r^{-\frac{\Theta}{A}}\nonumber\\
 +{\Theta \over A}r^{1-{\Theta\over A}}(b_1 \cos[a({\Theta \over A}-1)]-b_2\sin[a({\Theta \over A}-1)])\\
\tilde b_2={\Theta \over A}({\Theta\over A}-1)\{\sin[a({\Theta\over A}-2)]a_{11}+2\cos[a({\Theta\over A}-2)]a_{12}-\sin[a({\Theta\over A}-2)]a_{22}\}r^{-\frac{\Theta}{A}}\nonumber\\
 +{\Theta \over A}r^{1-{\Theta\over A}}(b_1 \sin[a({\Theta \over A}-1)]+b_2\cos[a({\Theta \over A}-1)])\\
\tilde b_3=b_3\cdot r^{2-{2\Theta\over A}}.
\end{align}
So if we choose $\Theta={\alpha \over 2}A$, then $\tilde b_k$ are bounded, specifically:
\begin{align}
{\lambda \alpha^2 \over 4}\leq(\tilde a_{ij})\leq{\Lambda \alpha^2\over 4},\\
|\tilde b_i|\leq C(C_E,\alpha).
\end{align}
For any $0<t<{\rho_0\over 2}$, construct:
\[
w^t=M[1-e^{-\gamma{[(\xi-t)^2+(\eta+1)^2-1}]}]\cdot \sup|h|
\] 
In $\Phi(\Omega)$:
\begin{align*}
          &\tilde a_{ij}w^t_{ij}+\tilde b_1w^t_\xi+\tilde b_2w^t_\eta+\tilde b_3w^t\\
\leq&M\sup|h|e^{-\gamma[((\xi-t)^2+(\eta+1)^2)-1]}[-4\gamma^2\lambda+4\Lambda\gamma+4\gamma C(C_E,\alpha)+12 \gamma C(C_E,\alpha)e^{9\gamma}r_0]\\
\leq&0,
\end{align*}
 if we choose  $  r_0\leq \frac{e^{-9\gamma}}{12\gamma C(C_E,\alpha)}\triangleq\epsilon_{A2},$ and $\lambda$ large enough s.t. $ -4\gamma^2\lambda+4\Lambda\gamma+4 C(C_E,\alpha)\gamma+1<0$.

So if $M, \ \gamma$ big enough(depending on $C_E,\ \Lambda,\ \lambda,\ \rho_0$ and $\alpha$) :
\begin{align}
\tilde a_{ij}w^t_{ij}+\tilde b_1w^t_x+\tilde b_2w^t_y+\tilde b_3w^t\leq 0 \text{ in }\Phi(\Omega),\\
\partial_{\vn}w^t\geq 0 \text{, on } \{ \theta= \Theta\}\\
w^t\geq h \text{, on } \partial B_{\rho_0}\cup\{\eta=0\}.
\end{align}
And,
\[|\nabla w^t|(t,0)\leq C(C_E,\lambda,\Lambda,\alpha,\rho_0)|h|_{0;\Phi(\Omega)},\,\, \text{for $0<t<{\rho_0\over 2}$}.
\]
Transforming above back to original domain $\Omega$, gives
\[
|\partial_y \Phi^*(w^t)|(t,0)\leq t^{{\Theta\over A}-1}C(C_E,\lambda,\Lambda,\alpha,r_0)|h|_{0;\Omega},\,\, \text{for $0<t<{2^{-{A \over\Theta}}r_0 }$}.
\]
Then by Lemma \ref{lem:maximumprinciple},  $h\leq \Phi^*(w^t)$ for any $t$($\ 0<t<{ 2^{-{A \over\Theta}}r_0}$),  on $\{0<x<{ 2^{-{A \over\Theta} }r_0},y=0\}$:
  \begin{equation}
 h_y(t,0)\leq t^{{\Theta\over A}-1}C(C_E,\lambda,\Lambda,\alpha,r_0)|h|_{0;\Omega}.
   \end{equation}
Apply process above to $-h$, we get Lemma \ref{Lemma A2}
\end{proof}

\subsection{Singularity Estimates}
\label{SingularityEstimate}

\begin{lemma}\label{lem:singularityestimate}
Suppose $r_0\leq\epsilon_{A1}(\lambda,C_E,\alpha)$, and given $h\in C^1(\overline\Omega )\cap C^2(\Omega\cup\Gamma_v\setminus(0,0)),$
\begin{align}
a_{ij}h_{ij}+b_1h_x+b_2h_y+b_3h=0,&\text{ in }\Omega{\label{equationofhsigularitylemma}},\\
h=0              ,         &\text{ on }\Gamma_h,{\label{dirichletconditionofhsingularitylemma}}\\
h_{\vn}=0        ,      &\text{ on }  \Gamma_v,{\label{neumannconditionofhsingularitylemma}}\\
h\geq y-\varepsilon      ,    &\text{ on }\partial B_{r_0}\cap K{\label{supportconditionofhsingularitylemma}},
\end{align}
then exists $\epsilon_{A3}(r_0,\lambda,\Lambda,C_E,\alpha)$, if $\varepsilon\leq\epsilon_{A3}$, $h$ cannot be in $C^{1,\alpha}(\Omega\cup(0,0))$.
\end{lemma}
\begin{proof}
In the following, we construct a subsolution, $w$, to system (\ref{equationofhsigularitylemma}) (\ref{dirichletconditionofhsingularitylemma})
(\ref{neumannconditionofhsingularitylemma}) and (\ref{supportconditionofhsingularitylemma}), which has growth rate $\frac{r}{(-\log r)^{\gamma}}$ along some direction near corner of wedge.

Since $h$ satisfies $h=\nabla h=0$ at corner of wedge, so if $h\in C^{1,\alpha}(\Omega\cup(0,0))$, we have in $\Omega$,
\[|\nabla h|\leq |h|_{1,\alpha}r^{\alpha}, \ |h|\leq\frac{1}{\alpha}|h|_{1,\alpha}r^{\alpha+1},\]
which contradicts with $h\geq w$.

We define
\[w\triangleq \frac{y}{(-\log y)^{\gamma}}-\mu x^2.\]
In $\Omega$,
\begin{align*}
        &a_{ij}w_{ij}+b_1w_x+b_2w_y+b_3w\\
\geq&\lambda\left(\frac{y}{(-\log y)^\gamma}\right)_{yy}-2\Lambda\mu-2C_Er^{\alpha-1}\mu|x|-C_Er^{\alpha-1}\left(\frac{y}{(-\log y)^\gamma}\right)_{y}-C_Er^{\alpha-1}\left(\frac{y}{(-\log y)^\gamma}+\mu x^2\right).
\end{align*}
Plug in the following
\[\left(\frac{y}{(-\log y)^\gamma}\right)_{y}=\frac{1}{(-\log y)^\gamma}+\frac{\gamma}{(-\log y)^{\gamma+1}},\]
\[\left(\frac{y}{(-\log y)^\gamma}\right)_{yy}=\frac{\gamma}{y(-\log y)^{\gamma+1}}+\frac{\gamma(\gamma+1)}{y(-\log y)^{\gamma+2}},\]
\[-r^{\alpha-1}\geq -y^{\alpha-1},\ -r^{\alpha-1}\geq-|x|^{\alpha-1},\ |x|\leq r_0, y\leq r_0,\]
we get 
\begin{align*}
	&a_{ij}w_{ij}+b_1w_x+b_2w_y+b_3w\\
\geq &\frac{\lambda\gamma}{y(-\log y)^{\gamma+1}}-2\Lambda\mu-2C_Er^\alpha_0\mu-C_E\mu r_0^{\alpha+1}\\
	&+\frac{1}{y(-\log y)^{\gamma+2}}\left(\lambda\gamma(\gamma+1)-C_E y^{\alpha}(-\log y)^2-C_E\gamma y^\alpha(-\log y)-C_E y^{\alpha+1}(-\log y)^2\right).
\end{align*}
Define 
\[C_1={1\over \lambda}\sup\{{2C_Ey^\alpha(-\log y)}\mid 0<y<r_0\},\]
\[C_2={1\over \lambda}\sup\{C_E(y^\alpha(-\log y)^2+y^{\alpha+1}(-\log y)^2)\mid 0<y<r_0\}.\]
And let
\[\gamma=\max\{C_1,C_2\}+1,\]
\[\mu=\frac{1}{2\Lambda+3C_E+1}\inf\{\frac{\lambda \gamma}{y(-\log y)^{\gamma+1}}\mid 0<y<r_0\}.\]
Then
\[a_{ij}w_{ij}+b_1w_x+b_2w_y+b_3w\geq0.\]

Recall that $r_0\leq e^{1\over 4}$, and note that $\gamma>1$, so
\[\frac{1}{(-\log r_0)^{\gamma}}\leq {1\over 4}.\]

On $\{\xi\geq\frac{r_0}{2}\},$
\[w\leq y-\mu\xi^2\leq y-\frac{r_0^2}{4}\mu,\]
 if $\varepsilon\leq \frac{r_0^2}{4}\mu$, then $w\leq y-\varepsilon$.

On $\{\xi\leq\frac{r_0}{2}\},$
\[w\leq {y\over 2}-\mu\xi^2\leq {y\over 2}=y-{y\over 2}\leq y-\frac{r_0}{4},\]
if $\varepsilon\leq \frac{r_0}{4}$, then $w\leq y-\varepsilon$.

And since $\frac{y}{(-\log y)^\gamma}$ is an increasing function of $y$ and note that $\vn$ is outer normal direction of $\Omega$, we have
\[w_{\vn}\leq 0 \text{, on }\Gamma_v.\]

So, we define $\epsilon_3=\min\{\frac{r_0^2\mu}{4}, \frac{r_0}{2}\}$. When $\varepsilon\leq \epsilon_3$, $w<h$, which gives a contradiction.

\end{proof}


\section{Boundary Krylov-Safonov Estimates}
\label{sec:boundaryKrylovSafonov}

In this section we prove the following generalization of Krylov-Safonov estimates:

\begin{lemma}
\label{lem:BoundaryKrylovSafonov}


Consider part of a ball $\Omega=B_2((0,0))\cap\{y>f(x)\}$, where $f$ is a Lipschitz function of $x$, and $|f|_{Lip}\leq\epsilon$ (we denote  $\overline\Omega\cap\{f(x)=y\}$ as $L$).Given $h\in C^1(\Omega\cup L)\cap C^2(\Omega)$, satisfying
\begin{itemize}
\item
$a_{ij}h_{ij}+b_1h_x+b_2h_y+b_3h=0\text{, in }\Omega, 
   \text{where, } \lambda\leq (a_{ij})\leq\Lambda,   |b_i|\leq C_E;$
\item    $   \int_{\Omega}{h^+}^2\leq\delta^2;$

\item
and $\exists\, C_\beta\geq 1$, at every point of $L$, one of the following two things happens:\\
1)$h\leq\delta$;\\
2)$\exists \vbeta=(\beta_1,\beta_2)$, $|\beta_1|\leq C_\beta$, $\beta_2\geq {1 \over C_\beta}$,
  $-D_\beta h+h\leq \delta$.

Then we have, there exist $\epsilon_K(C_\beta)$ and $r_K(C_\beta)$, if $\epsilon\leq\epsilon_K$ :
\[
h\leq C(\lambda,\Lambda,C_E,C_\beta)\delta \text{ in }B_{r_K}.
\]
\end{itemize}

\begin{figure} [h]
\centering
    \includegraphics[height=5cm]{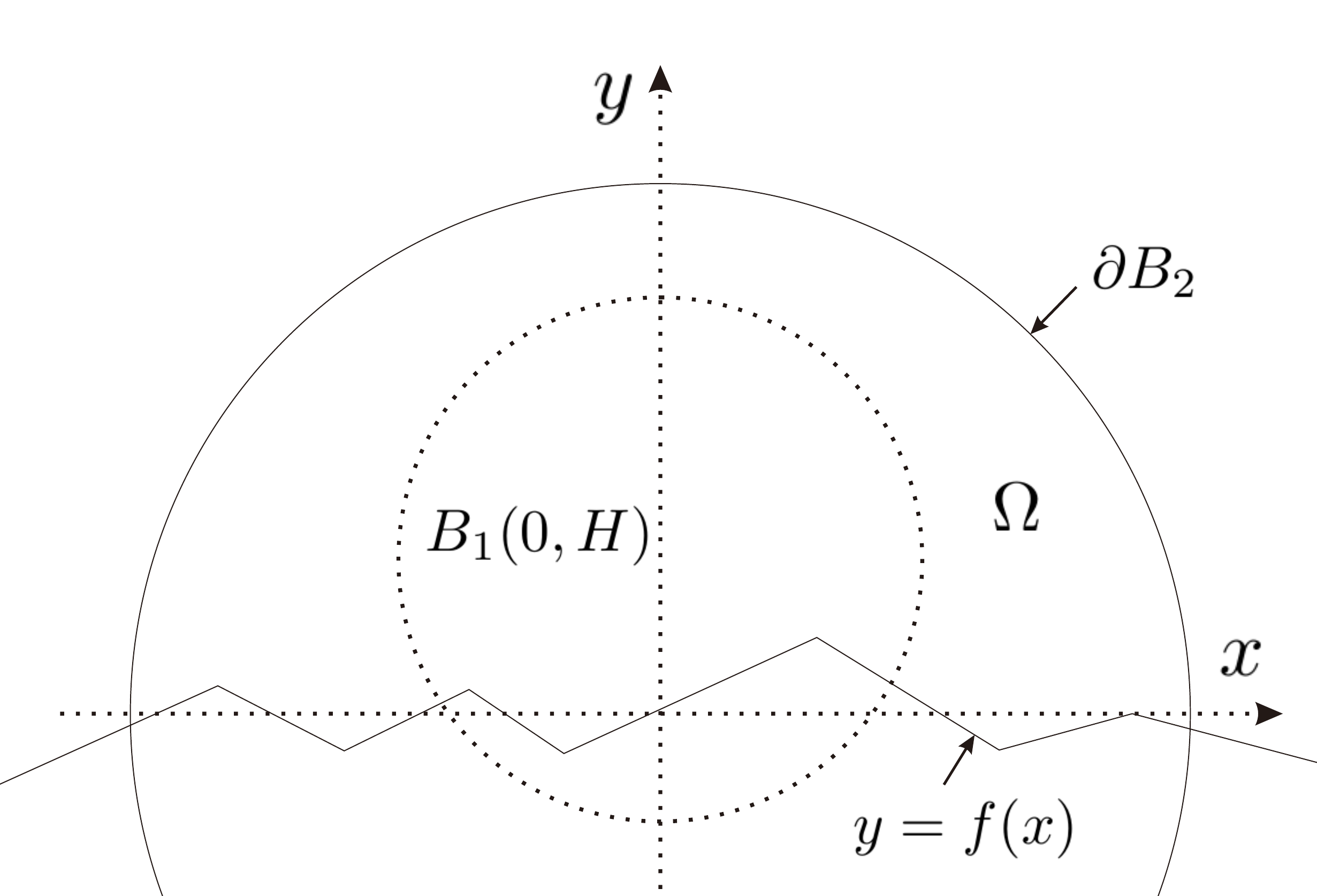}
\caption{Local Maximum Principle}
\end{figure}

\begin{proof}
Consider $g=1-x^2-(y-H)^2$, $w=(h-\delta)^+g^\gamma$, $H,\gamma$ to be determined later. Then, on $L$,
\begin{align*} 
D_\vbeta g&=D_\vbeta(-x^2-(y-H)^2)\\
                  &=-2x\beta_1+2(H-y)\beta_2\\
                  &\geq-2|x|C_\beta+2(H-y){1 \over C_\beta}\\    
                  &\geq{1\over 2 C_\beta}-2|x|C_\beta   \ \ \ \ (\text{here we need } H>{1\over 2} ,\epsilon <{1 \over 4})\\
                 &\geq 0\  (\text{     choose } H=1-{1 \over 64 C^4_\beta},\, \epsilon_K={1\over 64C_\beta^4})
\end{align*}
So on $\{h>\delta\}\cap L$,
\begin{align*}
w-D_\vbeta w&=(h-\delta)g^\gamma-D_\vbeta hg^\gamma-\gamma g^{\gamma-1}D_\vbeta g(h-\delta)\\
                     &=(h-D_\vbeta h-\delta)g^\gamma-\gamma g^{\gamma-1}D_\vbeta g(h-\delta)\\
                      &=-\gamma g^{\gamma-1}D_\vbeta g(h-\delta)\\
                     &\leq0
\end{align*}

Like section 9.1 of \cite{GT} we define upper contact set $\Gamma^+$ and normal map $\chi$, with more restriction:
\begin{align*}
\Gamma^+=&\{(x,y)\in B_1(0,H)|\exists L: \text{linear function, graph of $L$ upper contacts with graph of $w^+$ at $(x,y)$}, \\
                 &\nabla L\cdot(C_\beta,{1\over C_\beta})\leq0, \nabla L\cdot(-C_\beta,{1\over C_\beta})\leq0\},\\
\chi(x,y)=&\{\nabla L\mid \text{graph of $L$ upper contacts with graph of $w^+$ at $(x,y)$},\\
                  &\nabla L\cdot(C_\beta,{1\over C_\beta})\leq0, \nabla L\cdot(-C_\beta,{1\over C_\beta})\leq0\}.
\end{align*}

So at $p\in\Gamma^+\cap L $ if $w(p)>0$ by computation above, $D_\vbeta w>0$, $p\notin \Gamma^+$. It means if $w$ does not identically equal to zero, $\Gamma^+\cap L=\O $.
Then we have:
\begin{align}\label{KS}
\left(\arctan{1 \over C_\beta^2}\right){(\sup w^+)^2\over 2}\leq|\chi_{w^+}(\Gamma^+)|\leq\int_{\Gamma^+}\det D^2 w\leq\int_{\Gamma^+}
\frac{(a_{ij}w_{ij})^2}{\lambda^2}.
\end{align}

On $\Gamma^+,$
\begin{align*}
        a_{ij}w_{ij}
      =&a_{ij}h_{ij}g^\gamma+a_{ij}(g^\gamma)_{ij}(h-\delta)+a_{ij}h_i(g^\gamma)_j+a_{ij}h_j(g^\gamma)_i\\
        &+\sum_{k=1,2}(b_kh_k g^\gamma      +b_k(g^\gamma)_k(h-\delta)-b_kw_k)\\
\geq&-C(C_E,\Lambda,\gamma)(h^++\delta)-\Lambda\gamma|\nabla h|g^{\gamma-1}|\nabla g|-C_E|\nabla w|
\end{align*}

and 
  \[
|\nabla w|\leq \frac{w}{1-\sqrt{|x|^2+(y-H)^2}}\leq\frac{2w}{1-x^2-(y-H)^2}\leq2h g^{\gamma-1},
\]
which implies:
  \[
|\nabla h|\leq \frac{2hg^{\gamma-1}+(h-\delta)\gamma g^{\gamma-1}|\nabla g|}{g^{\gamma}}
,\]
so,
\[
|\nabla h||\nabla g|g^{\gamma-1}\leq2h|\nabla g|g^{\gamma-2}+(h-\delta)\gamma g^{\gamma-2}|\nabla g|^2.
\] 
Take $\gamma=2$ we get:
\[
a_{ij}w_{ij}\geq-C(C_E,\Lambda)(h^++\delta), \text{ on } \Gamma^+.
\]
Plug above into (\ref{KS}) we get:
  \[
(\sup w^+)^2\leq C(\lambda,\Lambda,C_E,C_\beta)\left(\int_\Omega h^2+\delta^2\right),
\]
so,
\[
\sup_{B_{r_K}}h\leq  C(\lambda,\Lambda,C_E,C_\beta)\delta,
\]
where we can take $r_K= \frac{1}{4 C_\beta^2}$.
\end{proof}
\end{lemma}

\section*{Acknowledgements}
The author would like to thank Mikhail Feldman for his patient tutorship in the past many years. He is also grateful to Guiqiang Chen, Xiuxiong Chen, Wei Xiang, Bin Xu, Bing Wang, Guohuan Qiu, Jiyuan Han, Jingrui Cheng for the supports and very helpful discussions.



\begin{thebibliography}{9}
\bibitem{ABR}S.Axler, P. Bourdon, W Ramey {\em Harmonic Function Theory} GTM137, Springer-Verlag 1992.
\bibitem{BCF}M. Bae, G. Chen, M. Feldman {\em Regularity of solution to regular shock reflection for potential flow}. Invent. Math. 175, 505-543.
\bibitem{CS}L.Caffarelli, S.Salsa {\em A Geometric Approach to Free Boundary Problem} 2005 American Mathematical Society
\bibitem{CFHX}G.Chen, M.Feldman, J.Hu, W.Xiang {\em Loss of Regularity of Solutions of the Lighthill Problem for Shock Diffraction for Potential Flow } SIAM J. Math. Anal., 52(2), 1096–1114. 
\bibitem{CF}G. Chen, M. Feldman {\em Global solution of shock reflection by large angle wedge for potential flow}. Ann. of Math. (2), 71, 1067-1182.
\bibitem{CFbook}G. Chen, M. Feldman {\em The Mathematics of Shock Reflection-Diffraction and von Neumann’s Conjectures.} Research Monograph, Princeton University Press: Princeton, 2017. 
\bibitem{CFcomparison}G. Chen, M. Feldman {\em Comparison principles for self-similar potential flow} Proc. of  AMS 140(2012), no.2 651-663
\bibitem{Textbook_of_Fluid_Dynamics}F. Chorlton {\em Textbook of Fluid Dynamics} 1976 D. Van Nostrand Company LTD
\bibitem{Courant} R. Courant  and  K. O. Friedrichs {\em Supersonic Flow and Shock Waves} 1948 , Springer-Verlag

\bibitem{GT}D. Gilbarg and N.Trudinger {\em Elliptic Partial Differential Equations of Second Order}.  Springer-Verlag, Berlin, 2001.
\bibitem{HS}Stefan Hilderbrandt and Friedrich Sauvigny {\em Minimal surfaces in a wedge \RN{2}.The edge-creeping phenomenon} Arch.Math.(1997)164-176.
\bibitem{KS}D. Kinderlehrer, G. Stampacchia {An Introduction to Variational Inequalities and Their Applications} 1980 Academic Press.
\bibitem{Li}G. Lieberman {\em Oblique Derivative Problems for Elliptic Equations} 2013 World Scientific


\bibitem{MNP}V. Maz'ya, E. Nazarow, B. Plamenevskij {\em Asymptotic Theory of Elliptic Boundary Value Problems in Singularly Perturbed Domains} 2000, Birkh\"auser Verlag
\bibitem{Ni}Nirenberg,L {\em On nonlinear elliptic partial differential equations and H\"older continuity}. Comm.Pure.Appl.Math. 6.(1953)103-156.
\bibitem{GNN}B.Gidas, Wei-Ming Ni and L. Nirenberg {\em Symmetry and Related Properties via the Maximum Principle}. Commun. Math. Phys. 68, 209—243 (1979)
\bibitem{PW}M. Protter, H. Weinberger {\em Maximum Principle in Differential Equations} 1967 Prentice-Hall 
\bibitem{JS}J. Serrin {\em A Symmetry Problem in Potential Theory} Arch. Ration. Mech. Anal. 43(1971), 304-318 
\bibitem{SS}E.Stein, R.Shakarchi {\em Real Analysis} 2005 Princeton University Press
\bibitem{ST}M. Sun and K. Takayama {\em Vorticity production in shock diffraction} J.Fluid Mech.(2003),vol. 478, pp. 237-256.
\bibitem{LucCarneige} L.Tartar {\em An Introduction to Sobolev Spaces and Interpolation Spaces} 2007 Springer
\end{thebibliography}
\end{document}